\documentclass[11pt]{amsart}
\usepackage{amsmath, amsthm, amssymb}
\usepackage[all]{xy}
\usepackage{stmaryrd}
\usepackage{mathrsfs}
\usepackage{caption}
\usepackage{thmtools}
\usepackage{graphicx}
\usepackage{float}
\usepackage{enumerate}
\usepackage{mathtools}
\usepackage[english]{babel}
\usepackage[dvipsnames]{xcolor}
\usepackage{tikz-cd}
\usepackage[hypertexnames=false,colorlinks=false, linkcolor=blue, citecolor=magenta, urlcolor=blue]{hyperref}

\usepackage[]{biblatex}
\DeclareLabelalphaTemplate{
  \labelelement{
    \field[final]{shorthand}
    \field{label}
    \field[strwidth=1,strside=left]{labelname}
  }
  \labelelement{
    \field[strwidth=2,strside=right]{year}
  }
}

\newtheorem{theorem}{Theorem}[section]
\newtheorem{corollary}[theorem]{Corollary}
\newtheorem{prop}[theorem]{Proposition}
\newtheorem{lemma}[theorem]{Lemma}
\newtheorem{definition}[theorem]{Definition}
\newtheorem{remark}[theorem]{Remark}

\newcommand{\op}[1]{\operatorname{#1}}

\newcommand\C{{\mathbb C}}

\newcommand\Q{{\mathbb Q}}
\newcommand\Z{{\mathbb Z}}
\newcommand\prj{{\mathbb P}}

\newcommand\cC{{\mathcal C}}
\newcommand\cD{{\mathcal D}}
\newcommand\cE{{\mathcal E}}
\newcommand\cF{\mathcal {F}}
\newcommand\cG{\mathcal {G}}

\newcommand\cI{{\mathcal I}}
\newcommand\cK{{\mathcal K}}

\newcommand\cM{{\mathcal M}}

\newcommand\cO{{\mathcal O}}
\newcommand\cP{{\mathcal P}}

\newcommand\cS{{\mathcal S}}

\newcommand\Pic{{\operatorname{Pic}}}

\newcommand\Gr{{\operatorname{Gr}}}
\newcommand{\bu}{\bullet}
\newcommand{\wt}[1]{\widetilde{#1}}
\newcommand{\nt}[1]{\tilde{#1}}
\newcommand{\wh}[1]{\widehat{#1}}
\newcommand{\llhb}[1]{\llbracket {#1} \rrbracket}
\newcommand{\llb}[1]{\llbracket {#1} \rrbracket^{\operatorname{1st}}}
\newcommand{\llbb}[1]{\llbracket {#1} \rrbracket^{\operatorname{2nd}}}

\newcommand\To{\longrightarrow}

\newcommand\INTO{\ \ar@{^(->}[r]<-.2ex>}
\newcommand{\Into}{\ensuremath{\lhook\joinrel\relbar\joinrel\rightarrow}}

\newcommand\pt{\operatorname{pt}}

\newcommand\ch{\operatorname{ch}}
\newcommand\rk{\operatorname{rk}}
\newcommand\rank{\operatorname{rank}}
\newcommand\vir{\operatorname{vir}}
\newcommand\vd{\operatorname{vd}}

\newcommand\coker{\operatorname{coker}}
\newcommand\im{\operatorname{im}}
\newcommand\id{\operatorname{id}}

\newcommand\hd{\operatorname{hd}}

\DeclareMathOperator{\sHom}{\kern -2pt\mathscr{H}\textit{\kern -4pt {om}}\,}
\DeclareMathOperator{\sExt}{\kern -2pt\mathscr{E}\textit{\kern -2pt {xt}}\,}

\newcommand\Jac{\operatorname{Jac}}

\newcommand\Cone{\operatorname{Cone}}

\renewcommand{\a}{\alpha}
\renewcommand{\b}{\beta}
\newcommand{\s}{\sigma}
\newcommand{\ga}{\gamma}

\newcommand{\de}{\delta}
\newcommand{\sff}[1]{\mathsf{#1}}
\newcommand{\scr}[1]{\mathscr{#1}}
\newcommand{\bb}[1]{\mathbb{#1}}
\newcommand{\bff}[1]{\mathbf{#1}}
\newcommand{\cc}[1]{\mathcal{#1}}
\newcommand{\fr}[1]{\mathfrak{#1}}
\newcommand{\xr}[1]{\xrightarrow{#1}}
\newcommand{\sh}{\mathfrak{h}}
\newcommand{\rkh}[1]{\mathsf{h}_{#1}}

\newcommand{\onto}[1]{\xymatrix{\ar@{->>}[r]^-{#1} & }}
\newcommand{\into}[1]{\xymatrix{\ar@{^{(}->}[r]^-{#1} & }}
\newcommand\beq[1]{\begin{equation}\label{#1}}
\newcommand\eeq{\end{equation}}
\newcommand\beqa{\begin{eqnarray*}}
\newcommand\eeqa{\end{eqnarray*}}
\newcommand{\Perf}{\operatorname{Perf}}

\usepackage{tabularx} 

\newcolumntype{N}{>{\bfseries}p{2cm}}
\newcolumntype{D}{>{\raggedright\arraybackslash}X}

\tikzset{curve/.style={settings={#1},to path={(\tikztostart)
    .. controls ($(\tikztostart)!\pv{pos}!(\tikztotarget)!\pv{height}!270:(\tikztotarget)$)
    and ($(\tikztostart)!1-\pv{pos}!(\tikztotarget)!\pv{height}!270:(\tikztotarget)$)
    .. (\tikztotarget)\tikztonodes}},
    settings/.code={\tikzset{quiver/.cd,#1}
        \def\pv##1{\pgfkeysvalueof{/tikz/quiver/##1}}},
    quiver/.cd,pos/.initial=0.35,height/.initial=0}
\usetikzlibrary{calc}

\numberwithin{equation}{section}

\begin{document}
\title[Virtual cycles of 3-term complexes]{Virtual cycles of 3-term complexes \\ and the Hilbert schemes of surfaces}

\author[E. Dominguez and A. Gholampour]{Emilio Dominguez and Amin Gholampour}

\maketitle
\begin{abstract}
    Given a 3-term perfect complex $\sff E$ over a quasi-projective variety $X$ and a nonnegative integer $r$, we define two virtual cycles and their refinements supported over the $r$-th degeneracy loci of $\sff E$. This is done by modifying the complex $\sff E$ after pulling it back to certain blow ups of $X$.  We establish several Thom-Porteous, comparison, duality and wall-crossing formulas for these virtual cycles. We apply this construction to perfect complexes arising from the universal objects over the Picard variety and the Hilbert schemes of non-singular complex projective surfaces. We recover, reprove and strengthen some of the known results involving the reduced cycles and the virtual cycles of the Hilbert schemes related to the curve counting theory and Vafa-Witten theory, respectively.  In the case of elliptic surfaces, we provide an explicit calculation generalizing that of Seiberg-Witten invariants.
\end{abstract}
\tableofcontents

\addtocontents{toc}{\protect\setcounter{tocdepth}{2}}

\section{Introduction}

\subsection{Overview} Virtual fundamental classes play an important role in modern enumerative geometry,  specially in curve counting theories, such as Gromov-Witten and Pandharipande-Thomas theories. The virtual fundamental class of a complex quasiprojective scheme  $M$ is usually  defined by means of a perfect obstruction theory \cite{behrend1997intrinsic}, which is a 2-term complex of vector bundles on $M$ approximating a truncation of the cotangent complex of $M$. In addition to instances above, the moduli spaces of stable sheaves on nonsingular projective surfaces and  Calabi-Yau or Fano threefolds admit natural perfect obstruction theories leading to Donaldson and Donaldson-Thomas invariants, respectively.

A typical example of a scheme $M$ with perfect obstruction theory is the zero locus of a section of vector bundle $E$ over a nonsingular quasi-projective variety $X$, or more generally, certain virtual resolutions of the degeneracy loci of a map of vector bundles $E_0\to E_1$ over $X$. In \cite{G1, G2}, this point of view was pursued for $M$ the nested Hilbert schemes of nonsingular projective surfaces to study the properties of their virtual cycles.

On the other hand, there are important moduli spaces that do not generally admit natural perfect obstruction theories. For example, the moduli spaces of stable sheaves on nonsingular varieties of dimension $\ge 3$. Having this limitation in mind, we build on the approach of \cite{G1,G2} to construct virtual cycles supported over the degeneracy loci of 3-term complexes of vector bundles over $X$
$$\sff E=\{E_{0}\to E_{1} \to E_{2}\}$$ sitting in degrees $[0,2]$. 
This enables us to define intersection numbers that are invariant under  \emph{certain deformations} of $X$ or $\sff E$ over nonsingular bases (cf. Section \ref{sec:defInvar}), 
and that only depend on the quasi-isomorphism class of $\sff E$ (cf. Lemma \ref{lem:indep} and Appendix \ref{quasiinvariance}). We do not insist on having a perfect obstruction theory on each degeneracy locus or its virtual resolution, though when there is a natural one, the resulting virtual cycle coincides with that of this paper (cf. Remark \ref{rem:hd=1,2}). We may therefore relax the assumption that $X$ is nonsingular, but it is important in what follows that $X$ is kept reduced and irreducible. 

Our main technique is to first pullback $\sff E$ to a suitable blow up $\nu\colon Y\to X$ over which $\sff E$ could essentially be thought of as being decomposed into a direct sum of two 2-term complexes of vector bundles sitting in degrees $[0,1]$ and $[1,2]$\footnote{In general, such a decomposition occurs only after further pulling back to an affine bundle over $Y$, which is an affine variety.}. Any such blow up will be referred to as an \emph{$\sff E$-suitable blow up}.  We then fix a positive integer $r$ and  define two virtual cycles, one for each direct summand,
\begin{align*}
    &\llb{\sff{E}, r}\in A_{\dim X-r(r-\sff e+\rkh{\sff E}^2)}(\wt{\op D}_r(\sff E)),\\
&\llbb{\sff{E},r}\in A_{\dim X-r(r-\rkh{\sff E}^2)}(\wt{\op D}_r(\sff E^\vee)),
\end{align*}
where $\sff e:=\rank(\sff E)$, $\rkh{\sff E}^2$ is the rank of the second cohomology of $\sff E$ on $X$, which is well-defined by the integrality condition on $X$, and $\wt{\op D}_r(\sff E)$ (respectively, $\wt{\op D}_r(\sff E^\vee))$ is a certain virtual resolution of the  $r$-th degeneracy locus $\op D_r(\sff E)$ of $\sff E$ (respectively, $\op D_r(\sff E^\vee)$ of $\sff E^\vee$), where the rank of the 0-th cohomology of $\sff E|_{\op D_r(\sff E)}$  (respectively, $\sff E^\vee|_{\op D_r(\sff E^\vee)}$) is at least $r$ (cf. Definition \ref{def:1st2nd}). These classes are shown to be independent of the choice of an $\sff E$-suitable blow up $\nu\colon Y\to X$  (cf. Lemma \ref{lem:indep}). 

When $X$ is nonsingular and the restriction $\sff E|_{{\op{D}}_r(\sff{E})}$ has vanishing second cohomology (cf. Condition \eqref{UDr}),  the first class 
$\llb{\sff E, r}$ coincides with the virtual cycle $\big[\wt{\op{D}}_r(\sff{E})\big]^{\vir}$ defined via a natural perfect obstruction theory and studied in \cite{G1, G2}. This is the case, for example, when $\sff E$ is a 2-term complex of vector bundles in degrees $[0,1]$, for which the second class vanishes.

The main application we consider in this paper is to the complexes arising from the universal objects over the Picard variety and the Hilbert scheme of nonsingular complex projective surfaces. Our construction can be readily extended to the case of perfect complexes of arbitrary lengths. In particular, in a future work we plan to use a similar construction to define and study virtual cycles for the Hilbert scheme of divisors on threefolds for which no perfect obstruction theory is known.  

\subsection{Properties of the virtual cycles}


Denote the $i$-th cohomology sheaf of the complex $\sff E$ of rank $\sff e$ as above by  $\frak h^i(\sff E)$ and its rank by $\sff h^i_{\sff E}$. We will also define a finite sequence of virtual cycles of  $\sff E$ indexed by nonnegative integers $k$  (cf. Definition \ref{def:refinedE}) 
$$\llb{\sff E, r}_k\in A_{\dim X-r(r-\sff e+\rkh{\sff E}^2)-k}(\wt{\op D}_r(\sff E))$$ with $\llb{\sff E, r}=\llb{\sff E, r}_0$. If $\sff E$ is a complex of vector bundles that is more generally sitting in degrees $[a,a+2]$, we define the above classes for $\sff E$ by applying our construction to the shifted complex $\sff E[-a]$.

We will prove a number of formulas comparing the first and the second classes of two or more complexes of vector bundles. We summarize a few of them in the following theorem. See Section \ref{sec:moddual} for more details and general results.  

\begin{theorem}[Comparisons]
Suppose $u\colon \sff F
\to \sff E$ is a map in the derived category $D^b(X)$ between two perfect complexes of ranks $\sff f, \sff e$ sitting in degrees $[0,2]$. Let  $\sff{G}:=\Cone(u)$ and $Q_{\widetilde{\op{D}}_r(\sff{F})}$, $Q_{\widetilde{\op{D}}_r(\sff{E})}$ be the universal rank $r$ quotient bundles over $\widetilde{\op{D}}_r(\sff{F})$ and $\widetilde{\op{D}}_r(\sff{E})$, respectively. 
\begin{enumerate}
    \item If $\sff{G}$ is represented by a locally free sheaf $G$ in degree 0, then there exists a closed immersion $\iota:\widetilde{\op{D}}_r(\sff{F})\hookrightarrow\widetilde{\op{D}}_r(\sff{E})$ and for any $k\ge 0$, we have
    $$\iota_*\llb{\sff{F},r}_k=c_{r(\sff e-\sff f)}({G}\otimes Q_{\widetilde{\op{D}}_r(\sff{E})})\cap \llb{\sff{E},r}_k.$$
    Moreover, $\widetilde{\op{D}}_r(\sff{F}^\vee)\cong\widetilde{\op{D}}_r(\sff{E}^\vee)$ with 
    $\llbb{\sff F, r}=\llbb{\sff E, r}$.
    \item If $\sff h^2_{\sff E}=0$  and  $\sff{G}$ is represented by a locally free sheaf $G$ in degree 1 then $\widetilde{\op{D}}_r(\sff{F})\cong\widetilde{\op{D}}_r(\sff{E})$ and
           \begin{align*}
               \llb{\sff{E}, r}&=\sum_{i=0}^{r(\sff f-\sff e-\rkh{\sff F}^2)}c_{{r(\sff f-\sff e-\rkh{\sff F}^2)-i}}(G\otimes Q_{\wt{\op{D}}_r(\sff F)})\cap\llb{\sff F,r}_i.
           \end{align*}
       
    \item If $\sff G$ is represented by a locally free sheaf $G$ in degree $2$ then 
    $$\widetilde{\op{D}}_r(\sff{F})\cong \widetilde{\op{D}}_r(\sff{E})\text{ and }\llb{ \sff{F},r}=\llb{ \sff{E},r}.$$

    \item For any $(\sff E$, $\sff E^\vee)$-suitable blow up $\nu\colon Y\to X$, 
    $$\llb{\sff{E}, r}=\nu_*\left(c_{r\rkh{\sff{E}}^1}(T\otimes Q_{\widetilde{\op{D}}_r(\sff{E})})\cap\llbb{\nu^*\sff{E}^\vee,r}\right),$$
    where $T$ (defined in \eqref{def:T}) is a locally free sheaf on $Y$ of rank $\rkh{\sff E}^1$.
    
    \end{enumerate}
    \end{theorem}

We will prove a Thom-Porteous formula for the classes associated to the complex $\sff E$ as above. For this, fix a locally free sheaf $B$ of rank $b$ on $X$ with a surjection $B^*\twoheadrightarrow \frak h^0(E^\vee)$. This gives a canonical closed embedding of the virtual resolution into the Grassmannian  of $r$-quotients of $B^*$ with the universal quotient bundle $Q_{\Gr}$ $$\iota\colon \wt{\op  D}_r(\sff E)\hookrightarrow \Gr(B^*,r).$$  The following result is proven in Theorem \ref{thm:T-Pformulas} along with a similar formula for the second class $\llbb{\sff E^\vee,r}$.


\begin{theorem}[Thom-Porteous] \label{thm:thom-por} 
For any $\sff E$-suitable blow up  $\nu:Y\to X$ let $\Gr(\nu^*B^*,r)\to \Gr(B^*,r)$ be the induced blow up. Then,  
\begin{align*}
&\iota_*\llb{\sff{E},r}=\nu_*\left(c_{r(b+\rkh{\sff{E}}^2-\sff e)}\big((\nu^*B+\frak h^2(\nu^*\sff E)-\nu^*\sff{E})\otimes Q_{\Gr}\big)\right),
\end{align*} where the pullback symbols to the Grassmannian are omitted. 
\end{theorem}

The following two theorems are proven for the case $r=1$. Let $$p:\wt{\op D}_1(\sff E)\to X,\qquad   \hat p\colon \wt{\op D}_1(\sff E^\vee)\to X$$ be the natural projections. In Theorem \ref{thm:gendual} we prove the following result. 
\begin{theorem}[Wall-Crossing] Suppose $\sff E$ is a 3-term complex of vector bundles for which $\op{D}_1(\sff E)\cap \op{D}_1(\sff E^\vee)=\emptyset$. Then, 
    $$p_*\llb{\sff{E},1}-(-1)^{\sff e}\hat p_*\llb{\sff{E}^\vee,1}=c_{1-\sff e}(-\sff{E})$$
in $A_{\dim X+\sff e-1}(X)$.
\end{theorem}

Suppose $B$ is a vector bundle chosen as above, so that the projective bundle $$q\colon \prj(B)\to X$$ contains both $\wt{\op{D}}_1(\sff{E})$ and $\wt{\op{D}}_1(\sff{E}^\vee)$ as closed subschemes. If for some positive integer $\ell$ there are maps in the derived category 
$$v:\cO_{\prj(B)}^{\oplus \ell}[-2]\to q^*\sff E\otimes\cO_{\prj(B)}(1),\qquad \hat v:\cO_{\prj(B)}^{\oplus \ell}\to q^* \sff E^\vee \otimes\cO_{\prj(B)}(1),$$
such that $\Cone(v|_{\wt{\op{D}}_1(\sff{E})})$ and $\Cone(\hat v[-2]|_{\wt{\op{D}}_1(\sff{E^\vee})})$ have vanishing second cohomologies. In Theorem \ref{thm:dualdeep}, we show the following.
\begin{theorem}[Duality] Suppose there are maps $v,\hat v$ in $D^b(\bb P(B))$ as above. 
    Then, 
    $$p_*\llb{\sff E,1}_{\ell-\rkh{\sff E}^2}=(-1)^{\sff e}\hat p_*\llb{\sff E^\vee,1}_{\ell-\rkh{\sff E}^0}$$ in $A_{\dim X+\sff e-\ell-1}(X)$.
\end{theorem}

\bigskip
\subsection{Nested Hilbert schemes} \label{sec:introNested} Let $S$ be a nonsingular complex projective surface with a $\beta\in H^2(S,\Z)$, whose image in $H^2(S,\bb C)$ is of $(1,1)$ type. Fix nonnegative  integers $n_1,n_2$. Let $$\Pic_\b(S)\cong \op{Jac}(S)$$ denote the Picard variety of line bundles with the first Chern class $\b$, and 
$S^{[n_i]}$ and $S_\b$ denote the Hilbert scheme of $n_i$ points and that of class $\b$ divisors on $S$, respectively.  
We apply our construction to the perfect complex 
$$\sff E:= R\sHom_\pi(\cI_1,\cI_2(\cP_\beta))\quad \text{over}\quad X:=S^{[n_1]}\times S^{[n_2]}\times \Pic_\beta(S)$$ sitting in degrees $[0,2]$. 
Here,  $\cP_\beta$ is a Poincar\'e line bundle over $S\times\Pic_\beta(S)$, $\cI_i$ is the universal ideal sheaf on $S\times S^{[n_i]}$, $\pi\colon S\times X\to X$, and we have omitted the obvious pullback symbols in the definition of $\sff E$. It is well-known that $X$ is a nonsingular projective variety of dimension $2n_1+2n_2+h^1(\cO_S)$. Consider the nested Hilbert scheme
$$S_{\beta}^{[n_1,n_2]}:=\{I_1(-D)\subset I_2\subset \cO_S : [D]=\beta, \text{ length}(\cO_S/I_i)=n_i\}.$$
In \cite{G2} it was shown that there is a canonical isomorphism  over $X$
$$S_{\beta}^{[n_1,n_2]}\cong \wt{\op D}_1(\sff E).$$
This allows us to define a virtual cycle on $S_\beta^{[n_1,n_2]}$ as

$$\llhb{S_\beta^{[n_1,n_2]}}:=\llb{\sff{E,1}}\in A_{n_1+n_2+\vd_\beta+p_g-\rkh{\sff E}^2}(S_\beta^{[n_1,n_2]}),$$ where $\vd_\b:=\b(\b-K_S)/2$ is the expected dimension of the moduli space in Seiberg-Witten theory and $p_g:=h^2(\cc O_S)$ is the geometric genus of $S$.
When $H^2(L)=0$ for every effective  $L\in \Pic_\beta(S)$, we will show that this class coincides with the reduced cycle $[S_\beta^{[n_1,n_2]}]^{\op{red}}$ constructed in \cite{gholampour2020nested, G2}. More generally, we define a finite sequence of virtual cycles 
$$\llhb{S_\beta^{[n_1,n_2]}}_k:=\llb{\sff{E,1}}_k\in A_{n_1+n_2+\vd_\beta+p_g-\rkh{\sff E}^2-k}(S_\beta^{[n_1,n_2]})$$
for $0\leq k\leq p_g-\rkh{\sff E}^2$. We will refer to the $k$-th one as \emph{the $k$-th refined cycle} on the nested Hilbert scheme, where the 0-th cycle is the one defined a few lines above. Note that even if $\b'\in H^2(S,\bb Z)$ and $\b$ differ by a torsion, and $\cc P_{\b'}$, $X'$, $\sff E'$ are defined analogously, we may have $\sff h^2_{\sff E}\neq \sff h^2_{\sff E'}$, so that the cycles $\llhb{S_\beta^{[n_1,n_2]}}_k$ and $\llhb{S_{\beta'}^{[n_1,n_2]}}_k$ have different virtual dimensions (cf. Remark \ref{rmk:diffvirdim}).  

The following is a direct consequence of Proposition \ref{proph2l=0} and Theorem \ref{redtovir}.

\begin{theorem} 
    If $[S_\beta^{[n_1,n_2]}]^{\vir}$ denotes the virtual cycle of the nested Hilbert scheme constructed in \cite{gholampour2020nested, G2}, then
$$\llhb{S_\beta^{[n_1,n_2]}}_{p_g-\rkh{\sff{E}}^2}=[S_\beta^{[n_1,n_2]}]^{\vir}.$$
 Moreover, if $H^2(L)=0$ for every effective  $L\in \Pic_\beta(S)$,  then  
 $$\sum_{k=0}^{p_g-\rkh{\sff E}^2}\llhb{S_\beta^{[n_1,n_2]}}_k=[S_\beta^{[n_1,n_2]}]^{\op{red}},$$ i.e. all the higher refined classes vanish in this case.
\end{theorem}

As a result, we can apply Theorem \ref{thm:thom-por} to get Thom-Porteous formulas for the refined classes of the nested Hilbert scheme generalizing \cite[Theorem 3]{G2}. Moreover, in Theorems \ref{thm:comphilbCO} and \ref {thm:compn1n2}, we will prove comparison formulas for these refined classes reproving and generalizing \cite[Theorem 5]{G2}.

The realization of the virtual and reduced cycles of the nested Hilbert schemes in terms of the refined cycles, also allows us in Theorems \ref{thm:hilbwall} and \ref{thm:wallhilb} to obtain uniform proofs for the wall-crossing and duality theorems for the nested Hilbert schemes arising from Seiberg-Witten theory (cf. \cite[Theorem 8]{G2} and \cite{DKO}) that do not use any of the comparison formulas, case by case analysis, the classification of surfaces, or the blow up formulas. 

\begin{theorem}[Wall-Crossing and Duality]
    If $\sff E|_{\op{D}_1(\sff E)}$ has vanishing second cohomology, then for $\hat{\beta}:=K_S-\b$ and the natural projections $p:S_\beta^{[n_1,n_2]}\to X$ and $\hat p:S_{\hat{\beta}}^{[n_2,n_1]}\to X$,
$$p_*[S_\beta^{[n_1,n_2]}]^{\op{red}}-(-1)^{\sff e}\hat p_*[S_{\hat{\beta}}^{[n_2,n_1]}]^{\op{red}}=c_{\sff 1-\sff e}(-\sff E),$$ where $\sff e=\rk(\sff E)=\chi(\cO_S)+\vd_\beta-n_1-n_2.$
    In particular, when $p_g=0$ the reduced and virtual cycles coincide and we get
    $$p_*[S_\beta^{[n_1,n_2]}]^{\op{vir}}-(-1)^{\sff e}\hat p_*[S_{\hat \beta}^{[n_2,n_1]}]^{\op{vir}}=c_{1-\sff e}(- \sff E).$$
    When  $p_g>0$ we have 
    $$p_*[S_\beta^{[n_1,n_2]}]^{\vir}=(-1)^{\sff e}\hat p_*[S_{\hat{\beta}}^{[n_2,n_1]}]^{\vir}.$$
\end{theorem}

\subsection{Hilbert scheme of divisors}
In \cite{DKO}, the proof of the identification of Poincar\'e invariants and Seiberg-Witten invariants of algebraic surfaces was reduced to proving a single formula  
(cf. \cite[Conjecture 0.1]{DKO})  $$\deg\; [S_{K_S}]^{\vir}=(-1)^{\chi(\cO_S)}$$ for a minimal  general type surface $S$. This formula was later proven in \cite{chang2012}.  In Theorem \ref{thm:classKS}, we prove the following result that in particular gives a new proof of the key formula above (cf. Corollary \ref{cor:chang}). Suppose that $B$ is a vector bundle of rank $b$ on $\Pic_{K_S}(S)$ chosen as above, so that  $\iota\colon S_{K_S}\hookrightarrow \bb P(B)\xr{q} \Pic_{K_S}(S)$.

\begin{theorem} Let $S$ be a nonsingular complex projective surface with  $p_g=h^2(\cO_S)>0$. If $g=h^1(\cc O_S)>0$, 
    then    \begin{align*}\iota_*\sum_{k=0}^{p_g}\llhb{S_{K_S}}_k&=c_{b-1+g-p_g}\big(q^*(B-\sff{E})(1)\big)\\&+\sum_{k=0}^{p_g}(-1)^{k+g-1} c_1(\cO_{\bb P(B)}(1)|_Z)^{b-1+p_g-k},\end{align*}
     where $\sff{E}=R\pi_*\cP_{K_S}$ and $\prj^{b-1}\cong Z\subset\prj(B)$ is the fiber of $q$ over $\omega_S:=\Lambda^2\Omega_S$.
    In the case that $h^1(\cO_S)=0$, 
    $$\sum_{k=0}^{p_g-1}\llhb{S_{K_S}}_k=s(\cO_{\bb P(B)}(1)|_{S_{K_S}}).$$
\end{theorem}

In Theorem \ref{thm:blow up}, we prove blow up formulas for the refined classes of the Hilbert scheme of divisors $S_\b$ generalizing those in Seiberg-Witten theory (cf. \cite{DKO}). 

In Section \ref{sec:elliptic}, we explicitly calculate some of the refined classes of the Hilbert scheme of divisors for elliptic surfaces. Let $\pi:S\to C$ be a relatively minimal elliptic surface over a general curve $C$ of genus $g$.  Denote by $F_1,\dots, F_n$ the multiple fibers of $S$ and denote the multiplicity of $F_i$ of by $m_i$. We will prove the following result that provides a good generalization of the evaluation of Seiberg-Witten invariants for the elliptic surfaces.
\begin{theorem} Let $D$ be a divisor of degree $d$ on $C$, and $\beta\in H^2(S,\bb Z)$ be the class of $\pi^*D+\sum_i a_iF_i$, where $\forall i\; 0\le a_i< m_i$. If $0\le d\le p_g$ then for each $0\le k\le d$,
$$\llhb {S_\beta}_k=(-1)^k\sum_{i=0}^{k}{p_g-d-1+k\choose k-i}\frac{\theta^ix^{k-i}}{i!},$$ where  $\theta$ is the class of the theta divisor on $\Pic_d(C)\cong \Pic_\b(S)$, and $x$ is the class of the  universal line bundle on $S_\b$. If $p_g< d\le p_g+g-1$ then for each $0\le k \le p_g$, the push down of $\llhb {S_\beta}_k$ to $\Pic_\b(S)$ is  $$(-1)^{d-p_g+k}\frac{\theta^{g-p_g+k}}{g-p_g+k!}{g-1-p_g+2k\choose d-p_g+k}.$$
Finally, if $d>p_g+g-1$ then only nonzero refined class is $$[S_\b]^{\op{red}}=\sum_{j=0}^{d-p_g}(-1)^j\frac{\theta^j x^{{d-p_g}-j}}{j!}.$$
\end{theorem}
We derive several corollaries of this result in Section \ref{sec:elliptic} and prove a stronger version of the duality theorem  in Theorem \ref{thm:dualityell} that involves all the refined classes.

\subsection{Curve counting}\label{intro:curve} Section \ref{sec:curvecounting} builds on the work of Kool-Thomas in \cite{Kool_2014, Kool_2014_2} and strengthen some of their results. For simplicity, we denote the nested Hilbert scheme $S^{[0,n]}_\b$ by $S^{[n]}_\b$. It is identified with the moduli space of stable of pairs on $S$. Let $\cc D_\b \subset S\times S^{[n]}_\b$ be the universal divisor, and $\cO(\cc D_\b)^{[n]}$ be the associated rank $n$ tautological bundle over $S^{[n]}_\b$.
For any $\sigma_i\in H^*(S,\beta)$, let $\tau(\sigma_i):=\pi_*(c_1(\cO(\cD_\beta))\cap p^*(\sigma_i))$, where 
	$$S \xleftarrow{p} S\times S_\beta^{[n]}\xr{\pi}  S_\beta^{[n]}$$
are the projections. Define the  residue invariant  of $S$ (see the footnote in Section \ref{sec:curvecounting}) by
    $$\llhb{S^{[n]}_\b(m,\bar\gamma)}_{\op{res}}:=\int_{\llhb{S_\beta^{[n]}}}c(\sff T)\big(\prod^{b_1(S)}_{i=1}\tau(\gamma_i)\big)\tau([\pt])^{m},$$
    where $\sff T=T_{S^{[n]}}-\cO(\cc D_\b)^{[n]}+{R}\pi_*\cO(\cD_\beta)-R\pi_*\cO$ in $K$-theory, 
   $[\pt]$ is the class of a point on $S$, and $\bar \ga=\{\gamma_1,\dots, \ga_{b_1(S)}\}$ is a normalized oriented integral basis of $H_1(S,\Z)/$torsion.

For any effective line bundle $L\subset \Pic_\b(S)$ and any $a$-dimensional sublinear  system $\bb P^a\subset |L|$, define 
\beq{equ:SnPa}S_{\prj^a}^{[n]}:=\{\cO_S(-D)\subset I\subset \cO_S\colon D\in \bb P^a\}\subset S^{[n]}_\b,\eeq and let its topological Euler characteristic be $e(S^{[n]}_{\bb P^a})$.
In Theorem \ref{thm:eulerref} we show the following result, which was proven in \cite{Kool_2014} with the additional assumption that $H^2(L)=0$ for every effective $L\in\Pic_\beta(S)$ (in which case the invariant above coincides with a  reduced residue stable pair invariant of $S$). 

\begin{theorem} Suppose a class $\beta\in H^2(S,\bb Z)$ is such that $$m:=\chi(L)-h^2(\beta)-1-\delta\geq0$$ for some nonnegative integer $\de$, where $h^2(\beta):=\min \{h^2(L)\colon L\in \Pic_\b(S)\}$. Then, for any $0\leq n\leq\delta$, any $\de$-very ample $L\in \Pic_\beta(S)$ and a general chain of  sublinear systems $\bb P^0 \subset \bb P^1\subset \cdots \subset \prj^\delta\subset|L|$, we have $$\llhb{S^{[n]}_\b(m,\bar\gamma)}_{\op{res}}=\sum_{j=0}^{\de}a_j e(S_{\prj^{\delta-j}}^{[n]})$$
 for some constants $a_j\in\Z_{\geq0}$ only depending on $\delta$ and $\chi(L)$ with $a_0=1$.
    \end{theorem}

 It is shown in \cite[Theorem 3.4]{Kool_2011} that for a $\delta$-very ample $L\in \Pic_\beta(S)$, $j\in \{0,\dots, \de\}$ and a general sublinear system $\prj^{\delta-j}\subset|L|$, there are integers $\a^j_r$, such that 
    \beq{eq:eulernr}q^{1-h}\sum_{n=0}^\infty e(S_{\prj^{\delta-j}}^{[n]}) q^n=\sum_{r=h-\delta+j}^h \a^j_rq^{1-r}(1-q)^{2r-2},\eeq where $h$ is the common arithmetic genus of the curves $|L|$. Moreover,
the coefficient $\a^0_{h-\delta}$ equals the number of $\delta$-nodal curves in $\prj^\delta$. For a class $\beta$ satisfying the hypotheses of the previous theorem, we use \eqref{eq:eulernr} to define the integers $\eta_r$ as
\begin{equation*}
    \sum_{j=0}^\de q^{1-h}\sum_{n=0}^\infty a_j e(S_{\prj^{\delta-j}}^{[n]}) q^n=\sum_{r=h-\delta}^h \eta_rq^{1-r}(1-q)^{2r-2}.
\end{equation*}
 By \eqref{eq:eulernr}, for a fixed $j>0$, the terms $a_je(S^{[n]}_{\prj^{\de-j}})$ do not contribute to $\eta_{h-\de},\dots,$ $\eta_{h-\de+j-1}$. In particular, \beq{equ:nodalcount}\eta_{h-\de}=\a^0_{h-\de}.\eeq After inverting the expression above and using the previous theorem, we find that 
$$\eta_{h}=
    \sum_{j=0}^\delta a_je(S^{[0]}_{\prj^{\de-j}})=\llhb{S^{[0]}_\b(m,\bar\gamma)}_{\op{res}},$$ and similarly for $h-\de\le r< h$, 
    $$\eta_r=\llhb{S^{[h-r]}_\b(m,\bar\gamma)}_{\op{res}}-\sum_{k=r+1}^{h}(-1)^{k-r}{2k-2\choose k-r}\eta_{k}.$$
This in particular shows that  the count of the $\delta$-nodal curves in $\bb P^\de$ given by \eqref{equ:nodalcount} can be expressed as a linear combination of the residue invariants $$\llhb{S_\b^{[n]}(m,\bar\gamma)}_{\op{res}} \qquad  n=0,\ldots,\delta.$$   
 

\bigskip

\subsection{Glossary of notation} For the convenience of the readers here we provide a list of the frequently used symbols and notation.
\begin{itemize}
    \item $X$: \qquad a complex quasi-projective variety of dimension $d$. 
    \item $D^b(X)$: \qquad the bounded derived category of $X$. 
    \item $\Perf(X,a,b)$: \qquad the full subcategory of $D^b(X)$ consisting of perfect complexes $\sff E$ of amplitude $[a,b]$. (The same definition when $X$ is any quasi-projective complex scheme.)
    \item $\cE, \cF,\cG,\dots$: \qquad   coherent sheaves on $X$ of ranks $e, f, g,\dots$.
    \item $E, F, G, \dots$: \qquad locally free sheaves on $X$ of ranks $e, f, g, \dots$.
    \item $\sff{E}, \sff F ,\sff G,\dots $: \qquad objects of $\Perf(X,a,b)$ of ranks $\sff e, \sff f, \sff g,\dots$.
    \item $\sff{E}^\vee$: \qquad the derived dual of $\sff E$.
    \item $\tau^{\le c}(\sff E$): \qquad the truncation of $\sff E$ at degrees bigger than $c$.
    \item $\hd(\cF)$: \qquad the homological dimension of $\cF$.
    \item $\Gr(\cc F,r)$:\qquad the Grassmannian of rank $r$ quotients of $\cc F$.
    \item $Q_{\Gr(\cc F,r)}$: \qquad the universal rank $r$ quotient bundle of $\Gr(\cc F,r)$.
      \item $\sh^i(\sff{E})$: \qquad the $i$-th cohomology sheaf of  $\sff E$.
    \item $\rkh{\sff{E}}^i$: \qquad the rank of the coherent sheaf $\sh^i(\sff E)$.
    \item $\op{D}_r(\sff E)$: \qquad the subscheme of $X$ corresponding to the fitting ideal $F_{r-1}(\sh^{-a}(\sff E^\vee))$.
    \item $\wt{\op{D}}_r(\sff E)$: \qquad  $\Gr(\sh^{-a}(\sff E^\vee),r)$.
    \item $Q_{\wt{\op{D}}_r(\sff E)}$:\qquad the universal quotient bundle of ${\wt{\op{D}}_r(\sff E)}$.
\item  $\llb{\sff E,r}_k, \llbb{\sff E,r}$: \qquad the first and second virtual cycles of $\sff E$.
 \item $S$: \qquad a nonsingular complex  quasi projective surface. 
    \item $\beta$: \qquad a class in $H^2(S,\Z)$.
    \item $\Pic_\b(S)$: \qquad the Picard variety of lines bundles $L$ with $c_1(L)=\b$.
    \item $S^{[n]}, S_\b$: \qquad the Hilbert schemes of $n$ points and of class $\b$ divisors on $S$ respectively. 
    \item $\op{AJ}\colon S_\b\to\Pic_\b(S)$:\qquad the Abel-Jacobi morphism.
     \item $S_\beta^{[n_1,n_2]}$: \qquad the nested Hilbert scheme of points and curves on $S$.
     \item $S^{[n]}_\b=S_\beta^{[0,n]}$:\qquad the moduli space of stable pairs on $S$.
    \item  $\llhb{S_\beta^{[n_1,n_2]}}_k$: \qquad the $k$-th refined virtual cycle of $S_\beta^{[n_1,n_2]}$.
\end{itemize}

\bigskip
\subsection*{Acknowledgment} We would like to thank Richard Thomas for sharing insights and for useful conversations.

\section{Virtual cycles of 2-term complexes--recap}

\subsection{Degeneracy loci and  virtual cycles}\label{virres}
Let $X$ be a complex quasi-projective variety\footnote{All varieties are assumed to be irreducible.}  of dimension $d$ and
$$\sigma:E_0\to E_1$$
 be a map of vector bundles of ranks $e_0$ and $e_1$ over $X$. For any $r\leq e_0$ the $r$-th degeneracy locus of $\sigma$ is the subscheme of $X$ defined by
\beq{defDr}\op{D}_r(\sigma):=\{x\in X|\dim \ker (\sigma|_x)\geq r)\},\eeq
with scheme structure given by the vanishing of $\wedge^{e_0-r+1}\sigma$, or equivalently, corresponding to the fitting ideal  $F_{r-1}(\coker(\sigma^*))$. We consider its virtual resolution 
$$\wt{\op{D}}_r(\sigma):=\Gr(\coker(\sigma^*),r)\overset{\iota}{\subset}\Gr(E_0^*,r),$$ where $p:\Gr(E_0^*,r)\to X$ is the relative Grassmannian of rank $r$ quotients of $E_0^*$. It is equipped with the rank $r$ universal bundle $Q_{\wt {\op{D}}_r(\sigma)}$ of the Grassmannian. Under the duality identification $\Gr(E_0^*,r)\cong \Gr(r,E_0)$, this virtual resolution  corresponds to the subscheme
$$\{(x,V)|V\subseteq\ker(\sigma|_x)\}\subset \Gr(r,E_0)$$
given by the vanishing locus of the composition
\beq{eq:secforDr}
Q^*_{\Gr(E_0^*,r)}\into{}p^* E_0\xr{p^*\sigma}p^*E_1,
\eeq
where $Q_{\Gr(E_0^*,r)}$ is the rank $r$  universal quotient bundle. If we assume that $X$ is smooth, from its description as the vanishing of \eqref{eq:secforDr}, or equivalently as that of the corresponding section of the bundle $Q_{\Gr}\otimes p^*E_1$, $\wt{\op{D}}_r(\s)$ acquires a natural perfect obstruction theory and hence a virtual cycle 
$$[\wt{\op{D}}_r(\sigma)]^{\vir}\in A_{d+r(e_0-r)-re_1}(\wt{\op{D}}_r(\sigma))$$ that only depends on the quasi-isomorphism class of the complex $$\sff E:=\{E_0\xr \s E_1\}$$
sitting in degrees $[0,1]$ (cf.  \cite[Proposition 1]{G2}). Its pushforward to $\Gr(E_0^*,r)$ is given by the Thom-Porteous formula
\begin{equation}\label{eq:TP-Gr}\iota_*[\wt{\op{D}}_r(\sigma)]^{\vir}=c_{re_1}( p^*E_1 \otimes Q_{\Gr(E_0^*,r)})\in A_{d-r(r-e_0+e_1)}(\Gr(E_0^*,r)),\end{equation}
and by \cite[Theorem 14.4]{fu}
\begin{equation}\label{eq:TP-X}
    p_*\iota_*[\wt{\op{D}}_r(\sigma)]^{\vir}=\Delta_{ r-\sff e}^rc(-\sff E)\in A_{d-r(r-\sff e)}(X).
\end{equation}
In the rest of this section, we will denote the degeneracy locus \eqref{defDr} and its virtual resolution by $\op{D}_r(\sff E)$ and   $\wt{\op{D}}_r(\sff E)$, respectively.


\subsection{Deepest degeneracy loci}\label{sec:setupB}
For the complex $\sff E=\{E_0\xr{\sigma} E_1\}$ as above there is always a minimum $k\ge 0$ such that $\op{D}_{k+1}(\sff E)=\emptyset$. For such $k$, we say $\op{D}_{k}(\sff E)$ is the \emph{deepest degeneracy locus}. The restriction of the cohomology sheaves $\sh^i(\sff E)|_{\op{D}_{k}(\sff E)}$ are then locally free for $i=0,1$, and hence the natural projection
$$p \colon \wt{\op{D}}_{k}(\sff E)\to \op{D}_{k}(\sff E)$$  is an isomorphism.

For general $r$, we may view the $r$-th degeneracy locus as a deepest one after modifying the complex $\sff E$ as follows. 
Such a modification was used in \cite{G2} and it will also prove to be useful in this paper.  Choose a locally free sheaf $B$ on $X$, such that $B^*$ surjects onto $\sh^0(\sff E^\vee)$. 
Suppose that there is a lift $\psi$ of this as in  \eqref{fig:defsurjB}\footnote{Such a lift always exists after pulling back to an affine bundle over $X$, which is an affine variety.}. For example, one could take $B$ to be $E_0$, but more natural choices of $B$ will be desirable in most of the applications. 
\begin{equation}\label{fig:defsurjB}
    \begin{tikzcd}
	&& {B^*} & \\
	{E_1^*} & {E_0^*} & {\sh^0(\sff E^\vee)} & 0
	\arrow["\psi"', dashed, from=1-3, to=2-2]
	\arrow[two heads, from=1-3, to=2-3]
	\arrow[from=2-1, to=2-2]
	\arrow[from=2-2, to=2-3]
	\arrow[from=2-3, to=2-4]
\end{tikzcd}
\end{equation}
Then, the virtual resolution $\wt{\op{D}}_r(\sff E)$  is identified with the deepest degeneracy locus $\op{D}_r(\rho)$, for 
$$\rho: q^*E_0 \xr{(q^*\sigma,q^*\psi^*)}q^*E_1\oplus q^*B\onto{} q^*E_1\oplus q^*B/Q^*_{\Gr(B^*,r)}$$ over $q:\Gr(B^*,r)\to X$, where the second arrow is the obvious projection. 

By the construction, the map $(q^*\sigma,q^*\psi^*)$ is an injection of bundles, so $B-\sff E$ is represented by a vector bundle of rank $b-\sff e$ in the $K$-group. When $X$ is nonsingular, we furthermore have the identification of the virtual cycles $[\wt{\op{D}}_r(\sff E)]^{\vir}=[\op{D}_r(\rho)]^{\vir}$, and as in \eqref{eq:TP-X} their pushforward to $\Gr(B^*,r)$ is given by Thom-Porteous formula
\begin{align}\label{eq:TP-Gr(B)} \Delta ^{(r)}_{b-\sff e}c\big(q^*B/Q^*_\Gr-q^*\sff E\big)&=\Delta ^{(r)}_{b-\sff e}c\big(q^*B-q^*\sff E-Q_\Gr\big)\\ \nonumber
&=c_{r(b-\sff e)}\big(q^*(B-\sff E)\otimes Q _\Gr\big).\end{align}

\subsection{Comparison theorem}
Given a map of 2-term complexes $u:\sff F\to \sff E$ of vector bundles, \cite[Section 3]{G2} provides formulas to compare $[\wt{\op{D}}_r(\sff E)]^{\vir}$ and $[\wt{\op{D}}_r(\sff F)]^{\vir}$,  whenever $\Cone(u)$ is represented by a locally free sheaf sitting in degree 0 or 1. 

\begin{theorem}[{\cite[Theorem 3.5]{G2}}] \label{compGT} Let $u \colon  \sff F\to \sff{E}$ be a map in the derived category between two 2-term complexes of vector bundles sitting in degrees $[0,1]$ over a nonsingular complex variety $X$.
\begin{enumerate}
    \item If $\Cone(u)$ is represented by a locally free sheaf $G$ sitting in degree 0, then there exists a closed immersion $\iota \colon  \wt{\op{D}}_r(\sff{F}) \hookrightarrow \wt{\op{D}}_r(\sff{E})$, such that the universal bundle $Q_{\wt{\op{D}}_r(\sff{E})}$ pulls back to that on $\wt{\op{D}}_r(\sff{F})$, and we have
$$\iota_*\big [\wt{\op{D}}_r(\sff{F})\big]^{\vir}=c_{r(\sff e-\sff f)}\big( p^*G\otimes Q_{\wt{\op{D}}_r(\sff{E})}\big)\cap \big[\wt{\op{D}}_r(\sff{E})\big]^{\vir}.$$

\item If $\Cone(u)$ is represented by a locally free sheaf $G$ sitting in degree 1, then $\wt{\op{D}}_r(\sff{F}) \cong \wt{\op{D}}_r(\sff{E})$,  and we have
$$\big [\wt{\op{D}}_r(\sff{E})\big]^{\vir}=c_{r(\sff f-\sff e)}\big(p^* G\otimes Q_{\wt{\op{D}}_r(\sff{E})}\big)\cap \big[\wt{\op{D}}_r(\sff{F})\big]^{\vir}.$$ 

\end{enumerate}
\end{theorem}

\section{Virtual cycles of 3-term complexes} \label{Sec:vir3-term}

In this section, we will work over a quasi-projective complex variety\footnote{All varieties are assumed to be irreducible.} $X$ of dimension $d$. 
Given $a,b\in\Z$, we  denote by $$\Perf(X,a,b)$$ the full subcategory of $D^b(X)$ consisting of perfect complexes $\sff E$ that admit a locally free resolution
$$\sff E\simeq \{E_a\to E_{a+1}\to\cdots\to E_b\}.$$ 
For any morphism $\nu\colon Y\to X$ of quasi-projective varieties, we denote the derived pullback of $\sff E$ by $$\nu^*\sff E\simeq \{\nu^*E_a\to\cdots \to \nu^*E_b\} \in \Perf(Y,a,b).$$
Analogous to Definition \eqref{defDr}, we let the $r$-th degeneracy locus of $\sff E$ be the subscheme of $X$ 
defined by
$$\op{D}_r(\sff E):=\{x\in X|\dim \sh^a(\sff E|_x)\geq r)\}$$
with scheme structure given by the fitting ideal $F_{r-1}(\sh^{-a}(\sff E^\vee)).$ Also, define the  virtual resolution of it as 
$$\wt{\op{D}}_r(\sff{E}):=\Gr(\sh^{-a}(\sff{E}^\vee),r)\xr{p} \op D_r(\sff E),$$ 
which is equipped with a rank $r$ universal quotient bundle that we denote by $Q_{\wt{\op{D}}_r(\sff E)}$.
In this paper, we will be mainly interested in 3 term complexes of vector bundles, i.e. the objects of  $\Perf(X,a,a+2)$.

To construct a well-behaved virtual cycle for $\wt{\op{D}}_r(\sff{E})$, we need to have a good control on the homological dimensions of the sheaf cohomologies $\sh^i(\sff E)$. This can be achieved by pulling back everything to some blow up $Y$ of $X$. We will construct a cycle over $Y$ and then push it down to $X$. There are various ways of modifying coherent sheaves (cf. \cite{rabano2024desingularizationssheaveshighergenus} for a good review of those methods). In the next section, we will give a construction of a blow up with the desired property. In Appendix \ref{equivblow up}, we will show that our construction is equivalent to one of the existing ones.

\subsection{Modification of 3-term complexes}\label{sec:modcomp} Our modification of perfect complexes will be based on the following key lemma.

\begin{lemma}\label{lem:blhd1}
     Let $\cF$ be a coherent sheaf on $X$. There exists a quasi-projective variety $Y$ and a birational morphism $\nu\colon  Y\to X$, such that $\cF$ has homological dimension at most 1. 
\end{lemma} 
\begin{proof}
Fix a presentation $G\to H\to \cF\to 0$ with $G, H$ locally free. Let $\cK:=\ker(G\to H)$ and $\cC:=\coker(H^*\to G^*)$. Then, $\cK=\cC^*:=\sHom(\cC, \cO_X)$. 

If $\cC$ is locally free, then $S:=\ker(G^*\to \cC)$ is locally free and dualizing the short exact sequence
$0\to S\to G^*\to \cC\to 0$ gives the short exact sequence $0\to \cK\to G\to S^*\to 0$, showing that $G/\cK$ is locally free, and hence $\hd(\cF)\leq1$. If $\cC$ is torsion then $\cK=0$ and the same will be true.

Otherwise, let $U\subseteq X$ be the open set where $\cC$ is locally free of rank $c\geq1$. Consider the relative Grassmannian $p:\Gr(G^*,c)\to X$. Since $\cC|_{U}$ is locally free of rank $c$, $p$ admits a local section
$$s_{\cC}:U\to \Gr(G^*|_U,c),\qquad x\mapsto(G^*|_x\twoheadrightarrow \cC|_x).$$
Let $Y$ be the closure of the image of $s_{\cC}$, and $\nu: Y\to X$ be the natural projection (the restriction of $p$). Note that the surjection $G^* \twoheadrightarrow \cC $ induces a canonical closed immersion $$\Gr(\cC,c) \xhookrightarrow{i} \Gr(G^*,c),$$ and 
$Y$ is then the unique  irreducible component of  the image of $i$ dominating $X$.  By the description above, we can see that the canonical surjection $$\nu^*\cC \twoheadrightarrow Q_{\Gr}|_{Y}$$ is an isomorphism over the open dense subset $\nu^{-1} (U)$. Therefore, if $\cc T$ is the torsion subsheaf of $\nu ^* \cc C$, we have the short exact sequence  $$0\to \cc T\to \nu^* \cc C \to Q_{\Gr}|_{Y}\to 0.$$
After dualizing, we find that $Q^*_\Gr|_Y \cong (\nu^* \cC)^*\cong \ker(\nu^* \sigma)$ 
is locally free. By construction $$\nu^*G/Q^*_\Gr|_Y  =(p^*G/Q^*_{\Gr})|_{Y}$$ is locally free, so the short exact sequence $$0\to \nu^*G/Q^*_\Gr|_Y \to \nu^* H\to \nu^* \cc F\to 0$$ shows that $\hd(\nu^*\cc F)\le 1$, as desired.
\end{proof}

\begin{remark}\label{remk:blsupp}
     \begin{enumerate}[i)]
         \item In Appendix \ref{equivblow up}, we show that the blow up of Lemma \ref{lem:blhd1} satisfies a universal property. In particular, it is independent of the choice of a presentation for $\cF$.
     \item In the case that $\cc F$ is locally free over its reduced support $\op{Supp}(\cc F)\subsetneq X$, one can give a simpler construction by taking $\nu \colon Y\to X$ to be the blow up of $X$ along  $\op{Supp}(\cc F)$. In that case, $ \nu^* \cc F$ will be locally free over the exceptional divisor of $Y$, which is a Cartier divisor, and hence $\hd(\nu^*\cc F)=1$ (cf. Remarks \ref{rmk:D3} and \ref{rmk:LiHu} for generalizations). 
     
     \end{enumerate}
\end{remark}

In the notation of the proof of Lemma \ref{lem:blhd1}, letting $E=G^*$ and $F=H^*$, we get a presentation
\begin{equation}\label{eq:presentationC}
    F\overset{\varphi}{\To} E\To \cC\To 0.
\end{equation}
We know that (cf. \eqref{eq:secforDr}) $\Gr(\cC,c)\subset \Gr(E,c)\xr{p} X$ is the zero locus $Z(\s)$ of the composition
    \beq{eq:sectionsigma}\sigma:p^*F\xr{p^*\varphi} p^*E \onto{} Q_{\Gr(E,c)}.\eeq
    The scheme $Z(\sigma)$ has expected codimension $cf$, and $\Gr(E,c)$ has dimension $d+c(e-c)$. In particular, if $\varphi$ is injective, since $f=e-c$, the expected dimension of $Z(\sigma)$ is $d$.

\begin{corollary}\label{cor:irred}
If $\Gr(\cC, c)$ is irreducible then it is identified with the blow up $Y$ constructed in Lemma \ref{lem:blhd1}. As a result, if $\varphi$ is injective   
$$[Y]=\Z(\s)\in A_d(Z(\sigma)),$$ where $\s$ is given in \eqref{eq:sectionsigma} and $\Z(\s)$ is the localized top Chern class of $p^*F^*\otimes Q_\Gr$ (cf. \cite[Section 14.1]{fu}).
In particular, for the inclusion $i:Y\hookrightarrow \Gr(E,c)$, we have
$$i_*[Y]=\Delta^{(f)}_{c}c\big({Q}_{_{\Gr}}-p^*F \big).$$
\end{corollary}
\begin{proof}
The first claim is immediate from the proof of Lemma \ref{lem:blhd1}, where $Y$ was defined to be an irreducible component of $\Gr(\cC,c)$. For the next claims,
    since $Y=Z(\sigma)$ is irreducible of the expected dimension $d$,  there is some integer $t$, such that $\Z(\sigma)=t[Y]$. By Thom-Porteous formula
    \begin{align*}
        i_*\Z(\sigma)&=c_{c(e-c)}(p^*F^*\otimes{Q}_{\Gr})        =\Delta_{c}^{(f)}c({Q}_{\Gr}-p^*F).
    \end{align*}
     Since  $\nu_*\left(\Delta_{c}^{(f)}c({Q}_{\Gr}-p^*F)\right)=[X]$ (cf. \cite[Example 14.2.1]{fu}) and $\nu_*[Y]=[X]$, we find that $t=1$ and the claim follows.
\end{proof}

\begin{corollary} \label{cornu}
For any $\sff{E}\in \Perf(X,0,2)$, there exists a blow up $\nu\colon Y\to X$, such that 
$\op{hd}(\sh^{2}(\nu^*\sff{E}))\le 1$. As a result,  $\tau^{\le 1}(\nu^* \sff{E})\in\Perf(Y,0,1)$. 
\end{corollary}
\begin{proof} Fix a locally free resolution
$$\sff{E} \simeq \{ E_0\xr{\s_0} E_1\xr{\s_1}E_{2}\}.$$
Applying Lemma \ref{lem:blhd1} to $\cc F:=\sh^2(\sff{E})$, there is a blow up $\nu\colon Y\to X$, such that $\nu^*\cc F=\sh^2(\nu^* \sff{E})$ has homological dimension at most 1. Then $K:=\ker(\nu^*\sigma_1)$ is locally free since $\nu^*E_1\to \nu^*E_2\to\nu^*\cF\to 0$ is exact. Thus,
$$\tau^{\leq1}(\nu^*\sff{E})\simeq\{\nu^*E_0\to K\}\in\Perf(Y,0,1)$$
\end{proof}

\begin{definition} \label{def:Esuit}
    \emph{For any $\sff{E}\in \Perf(X,0,2)$, we will call a blow up of $X$ with the property from Corollary \ref{cornu} an \emph{$\sff E$-suitable blow up}}.
\end{definition}

Let $\sff{E}\in \Perf(X,0,2)$ and $\nu\colon Y\to X$ be an $\sff E$-suitable blow up. Fix a locally free resolution
$$\sff{E} \simeq\  \{E_0\xr{\s_0} E_1\xr{\s_1}E_{2}\}.$$
For $K:=\ker(\nu^*\sigma_{1})$, we have the quasi-isomorphisms
 \begin{align}\label{eq:df35}
     &\tau^{\le 1}(\nu^*\sff{E}) \simeq \big \{\nu^* E_0 \xr{\nu^* \s_0}  K\big\},\\ &\sh^2(\nu^*\sff{E})[-2]\simeq\{\nu^* E_{1}/K\xr{\nu^*\sigma_1}\nu^*E_2\}\notag\end{align}  
 to 2-term complexes of locally free sheaves. For any $r\geq 1$, we have the commutative diagrams of fiber squares
$$\xymatrix{\wt{\op{D}}_r(\nu^*\sff{E})=\Gr(\sh^0(\nu^*\sff E^\vee),r) \ar@{^(->}[r] \ar[d]^{\nu} &\op {Gr}( \nu^*E_0^*,r) \ar[r]^-{p} \ar[d]^-{ \nu} & Y\ar[d]^\nu\\
 \wt{\op{D}}_r(\sff{E})= \Gr(\sh^0(\sff E^\vee),r) \ar@{^(->}[r] &\op {Gr}( E_0^*,r) \ar[r]^-{p} & X,}$$
\begin{equation}\label{fig:def1st}
    \xymatrix{\wt{\op{D}}_r(\nu^*\sff E) \ar[d] \ar[r]& \Gr(\nu^* E_0^*,r)\ar[d]^-{\wt {\nu^*\s_0} }\\ \Gr(\nu^*E_0^*,r) \ar[r]^-{0_{K_1}} &  Q_{\Gr}\otimes p^* K=: K_1,}
\end{equation} 
where $\wt{\nu^*\sigma_0}$ is the section of the vector bundle $K_1$ induced by the composition
$$Q^*_{\Gr(\nu^*E_0^*,r)}\into{}p^* \nu^*E_0\xr{p^*\nu^*\sigma_0}p^*K.$$
Similarly, from the commutative diagrams of fiber squares
$$\xymatrix{\wt{\op{D}}_r(\nu^*\sff E^\vee)=\Gr(\sh^2(\nu^*\sff E),r) \ar@{^(->}[r] \ar[d]^{\nu} &\op {Gr}( \nu^*E_2,r) \ar[r]^-{p} \ar[d]^-{ \nu} & Y\ar[d]^\nu\\
 \wt{\op{D}}_r(\sff E^\vee)=\Gr(\sh^2(\sff E),r)  \ar@{^(->}[r] &\op {Gr}( E_2,r) \ar[r]^-{p} & X,}$$ 
\begin{equation}\label{fig:def2nd}
\xymatrix{\wt{\op{D}}_r(\nu^*\sff{E}^\vee) \ar[d] \ar[r]& \Gr(\nu^* E_2,r)\ar[d]^-{\wt {\nu^*\s_1^*} }\\ \Gr(\nu^*E_2,r) \ar[r]^-{0_{K_2}} & Q_{\Gr}\otimes p^*(\nu^*E_1/K)^*=:K_2,}
\end{equation}
where the section $\wt{\nu^*\sigma_1^*}$ of $K_2$ is induced by  the composition
$$Q^*_{\Gr(\nu^*E_2,r)}\into{}p^* \nu^*E_2^*\xr{p^*\nu^*\sigma^*_1}p^*(\nu^*E_1/K)^*.$$

\begin{definition} \label{def:1st2nd}
\emph{Given $\sff{E}\in\Perf(X,a,a+2)$ and an  integer $r\ge 1$, after replacing $\sff{E}$ by $\sff{E}[a]$, we may assume that $a=0$. Then, we use \eqref{fig:def1st} and \eqref{fig:def2nd} to define \emph{the first virtual cycle} and \emph{the second virtual cycle} of $\sff{E}$ as} 
    \begin{align*} 
\label{class} &\llb{\sff{E},r}:= \nu_*\left(0_{ K_1}^! [\op {Gr}( \nu^*E_0^*,r)]\right) \in A_{d-r(r-\sff e+\rkh{\sff E}^2)}(\wt{\op{D}}_r(\sff E)),\\
&\llbb{\sff{E},r}:=\nu_*\left(0_{ K_2}^! [\op {Gr}( \nu^*E_2,r)]\right) \in A_{d-r(r-\rkh{\sff{E}}^2)}(\wt{\op{D}}_r(\sff E^\vee)).\end{align*}
\emph{In Lemma \ref{lem:indep}, we show that these classes are independent of the choices of an $\sff E$-suitable blow up and of a resolution of $\sff E$ justifying the notation used above. If the number $r$ is clear from the context, we omit it  and use $\llb{\sff{E}}$ and $\llbb{\sff{E}}$ to denote these two classes.}

\end{definition}

\begin{lemma} \label{Lq*} Let $q\colon Z\to Y$  be a birational morphism of quasi-projective varieties.
 If $\cc F$ is a coherent sheaf of homological dimension at most 1 on  $Y$, then $ L q^* \cc F=q^*\cc F$.  
\end{lemma}
\begin{proof}
By assumption  we can choose a locally free resolution $$0\to G\to H\to \cc F\to0.$$ So $L^{i\geq 2}q^*\cc F$ are clearly 0. Since $\cF$ is locally free on an open dense subset of $Y$, if we apply the derived functor $Lq^*$ to this sequence we see that $L^1q^* (\cc F)$ is a torsion subsheaf of $q^* G$, and hence it must be 0 as well.
\end{proof}

\begin{lemma} \label{lem:indep}
Both classes $\llb{\sff{E}}$ and $ \llbb{\sff{E}}$ are independent of the choices of an $\sff E$-suitable blow up and of a resolution of $\sff E$. In particular, they only depend on the quasi-isomorphism class of $\sff{E}$.
\end{lemma}
\begin{proof} We prove the claims for $\llb{\sff E}$. For $\llbb{\sff E}$ they are proven similarly. 
Suppose $\nu'\colon Y'\to X$ is another $\sff E$-suitable blow up. Working with the same locally free resolution for $\sff E$, we can set up a similar fiber square to that of \eqref{fig:def1st} in which $\nu, Y, K, K_1$ are replaced by $\nu', Y', K', K'_1$, respectively. Let $Z$ be the irreducible component of $Y\times_X Y'$ dominating $X$, and consider the commutative square 
\beq{ZYY'}\xymatrix{Z\ar[r]^-q \ar[d]^-{q'}&Y\ar[d]^-{\nu} \\  Y' \ar[r]^-{\nu'} &X.}\eeq
By Lemma \ref{Lq*}, $\hd(\sh^2(q^*\nu^*\sff E))=\hd(\sh^2(q'^*\nu'^*\sff E))\le 1$ and also $$q^* \tau^{\le 1}(\nu^*\sff{E})\cong  \tau^{\le 1} (q^*\nu^*\sff{E})\cong \tau^{\le 1} (q'^*\nu'^*\sff{E})\cong q'^* \tau^{\le 1}(\nu'^*\sff{E}).$$  
By \cite[Prop 14.1.d]{fu}, 
\begin{align*}
    0_{K_1}^! [\op {Gr}( \nu^*E_0^*,r)]&=q_*0_{q^*K_1}^! [\op {Gr}( (\nu\circ q)^*E_0^*,r)],\\
     0_{K'_1}^! [\op {Gr}( \nu'^*E_0^*,r)]&=q'_*0_{q'^*K'_1}^! [\op {Gr}((\nu'\circ q')^*E_0^*,r)].
\end{align*}
From these and the identification $q^*K_1\cong q'^*K'_1$, we get 
$$\nu_*0_{K_1}^! [\op {Gr}( \nu^*E_0^*,r)]=\nu'_*0_{K'_1}^! [\op {Gr}( \nu'^*E_0^*,r)].$$ 
This shows that $\llb{\sff{E}}$ is independent of the choice of an $\sff E$-suitable blow up. 

 Suppose that $\{E'_0\xr{\s'_0} E'_1\xr{\s'_1}E'_{2}\}$ is another locally free resolution for $\sff E$. Let $\nu\colon Y\to X$ be an $\sff E$-suitable blow up, and $\mu\colon Z\to Y$ be any blow up with $Z$ nonsingular. By Lemma \ref{Lq*}, $Z$ is also an $\sff E$-suitable blow up of $X$. By the first part of the proof, \beq{EbuE'bu}\llb{\nu^* E_\bu}=\mu_*\llb{\mu^*\nu^* E_\bu},\qquad \llb{\nu^* E'_\bu}=\mu_*\llb{\mu^*\nu^* E'_\bu}.\eeq
On the other hand, the 2-term complexes of vector bundles $$\{\mu^*\nu^*E_0\xr{\mu^*\nu^*\s_0} \ker(\mu^*\nu^*\s_1)\},\qquad \{\mu^*\nu^*E'_0\xr{\mu^*\nu^*\s'_0} \ker(\mu^*\nu^*\s'_1)\}$$
are quasi-isomorphic to $\tau^{\le 1}(\mu^*\nu^*\sff E)$. Since $Z$ is nonsingular, and the virtual cycles resulted from a perfect obstruction theory discussed in Section \ref{virres} only depend on the quasi-isomorphism type,  $$\llb{\mu^*\nu^* E_\bu}=[\wt {\op D}_r(\mu^*\nu^*\s_0)]^{\vir}=[\wt {\op D}_r(\mu^*\nu^*\s'_0)]^{\vir}=\llb{\mu^*\nu^* E'_\bu}.$$ Applying $\nu_*\mu_*$ to both sides, and using \eqref{EbuE'bu} and the first part of this proof again, we get $\llb{E_\bu}=\llb{E'_\bu}$, as desired.
\end{proof}

\begin{remark} \label{rem:remonappA}
    The last part of the proof of Lemma \ref{lem:indep} relies on the resolution of singularities to ensure the existence of a nonsingular suitable blow up. In Appendix \ref{quasiinvariance}, we will give an alternative proof of independence from the choice of locally free resolutions.
\end{remark}

\begin{remark} \label{rem:hd=1,2}
    i) In the case that $\sff{E}$ satisfies $\hd(\sh^2(\sff{E}))\leq1$ and $X$ is nonsingular,  no blow up is required to define  $\llb{\sff{E}}$ and it coincides  with  $[\wt{\op{D}}_r(\tau^{\leq1}(\sff{E}))]^{\vir}$ (cf. Section \ref{virres}). An important special case is when $\sff E\in \Perf(X,0,1)$ in which case $$\llb{\sff E}=[\wt{\op{D}}_r(\sff{E})]^{\vir},\qquad \llbb{\sff E}=0.$$ ii) If $X$ is nonsingular and  $\hd(\sh^2(\sff{E}))=2$, even though  $\tau^{\leq1}(\sff{E})\in\Perf(X,0,1)$,  the classes  $[\wt{\op{D}}_r(\tau^{\leq1}(\sff{E}))]^{\vir}$ and $\llb{\sff{E}}$ may not still agree.
\end{remark}

\begin{corollary} \label{cor:pullpush}
        For any blow up $\nu \colon Y\to X$, we have $$ 
\llb{\sff{E}}=\nu_*\llb{\nu^*\sff{E}},\qquad 
\llbb{\sff{E}}=\nu_*\llbb{\nu^*\sff{E}}.$$  
\end{corollary}
\begin{proof}
Choose a $\nu^*\sff E$-suitable blow up $\mu\colon Z\to Y$, and use the argument for the first part of the  proof of Lemma \ref{lem:indep}.
\end{proof}

In some of our applications, it will be more appropriate to define a (finite) sequence of 1st classes, starting with $\llb{\sff E,r}$.

\begin{definition}\label{def:refinedE}
    \emph{Given $\sff E\in\Perf(X,0,2)$, integers $r\ge 1, k\ge 0$ and an $\sff E$-suitable $\nu:Y\to X$  blow up, we define} 
    $$\llb{\sff E,r}_k:=\nu_*\left(s_k\big(\sh^2(\nu^*\sff E)\otimes Q_{\wt{\op{D}}_r(\nu^*\sff E)}\big)\cap\llb{\nu^*\sff E, r}\right)$$
    \emph{in $ A_{d-r(r-\sff e+\rkh{\sff E}^2)-k}(\wt{\op{D}}_r(\sff E))$. Here, to simplify the notation, we have omitted derived pullback inside the Segre class, and think of (cf. \eqref{eq:df35}) $$\sh^2(\nu^*\sff E)\otimes Q_{\wt{\op{D}}_r(\nu^*\sff E)} \in \Perf(\wt{\op{D}}_r(\nu^*\sff E),-1,0).$$ By the same argument as in Lemma \ref{lem:indep}, these classes are also independent of the choices of  blow ups and resolutions for $\sff E$, and as before, we sometimes omit $r$ from the notation. Note that by Corollary \ref{cor:pullpush}, $\llb{\sff E}_0=\llb{\sff E}$.}
\end{definition}

\subsection{Special 3-term complexes}
These are 3-term complexes of vector bundles that in our applications behave like the 2-term ones. 
Let $\sff E\in\Perf(X,0,2)$. We say that $\sff E$ is \emph{$r$-special} if one of the following equivalent conditions is satisfied:
\begin{equation} \label{UDr}
\begin{split}
 & \op{D}_r(\sff E)\cap \op{Supp}(\frak h^2(\sff E))=\emptyset\iff {\mathsf{E}}|_{\mathrm{D}_r(\sff E)}\in\Perf(\op{D}_r(\sff E),0,1) \iff  \\ & 
\exists \text{ open neighborhood } U  \text{ of } \mathrm{D}_r(\sff E)  
\text{ such that } \sff{E}|_{U}\in\Perf(U,0,1).
\end{split}
\end{equation}
\begin{remark}\label{rem:UDr}
When $r=1$, Condition \eqref{UDr} can be equivalently phrased as $${\op{D}}_1(\sff E)\cap {\op{D}}_1(\sff E^\vee)=\emptyset.$$ 
\end{remark}

If $\sff E$ is $r$-special and $X$ is nonsingular, then the construction of Section \ref{virres} can be applied to ${\sff E}|_{U}$ to equip $\wt{\op{D}}_r(\sff{E})$ with a perfect obstruction theory, which only depends on the quasi-isomorphism class of $\sff E\in D^b(X)$ (cf. \cite{G2}). We denote the resulting virtual cycle by
$$[\wt{\op D}_r(\sff E)]^{\vir}\in A_{d-r(r-\sff e)}(\wt{\op D}_r(\sff E)).$$
However, even in this case, we cannot in general get a useful Thom-Porteous formula as in Section \ref{virres} for the pushforward of  $[\wt{\op{D}}_r(\sff{E})]^{\vir}$ to $X$. As we will see, such a formula exists for the first and second classes of $\sff E$ over any $\sff E$-suitable blow up (cf. Theorem \ref{thm:T-Pformulas}), and so does for $[\wt{\op D}_r(\sff E)]^{\vir}$ by the following proposition.

\begin{prop}\label{prop:red1st}
    If $\sff E\in\Perf(X,0,2)$ is $r$-special and $X$ is nonsingular, then
$$\big[\wt{\op{D}}_r(\sff{E})\big]^{\vir}=\llb{\sff E, r}.$$
\end{prop} 
\begin{proof}
    If Condition \eqref{UDr} is satisfied, the blow up $\nu$ constructed in Lemma \ref{lem:blhd1} is an isomorphism over the open set $U$, where $\sh^2(\sff E|_U)=0$. Therefore, $$\sff{E}|_{U} \cong \tau^{\le 1}( \nu^* \sff{E}) |_{\nu^{-1}(U)},$$ and hence $\op{D}_r(\nu^*\sff{E}) \subset \nu^{-1}(U)$ is isomorphic to  $\op{D}_r(\sff{E})  \subset U$.
    By the discussion of Section \ref{virres} and the same argument as in Lemma \ref{lem:indep} the claim is proven. 
\end{proof}

Without the nonsingularity assumption in Proposition \ref{prop:red1st}, the left hand side of the equality that relies on a perfect obstruction theory is not defined. In the following proposition, we prove an exact analog of Theorem \ref{compGT} for the $r$-special complexes.

\begin{prop}\label{prop:comparisonSpecial}
    If $\sff E, \sff F\in\Perf(X,0,2)$ are $r$-special, and  $u \colon  \sff F|_U\to \sff{E}|_U$ is a map in the derived category, where $U\subset X$ is as in Condition \eqref{UDr}, then the parts (1) and (2) of Theorem \ref{compGT} are true after replacing $\big[\wt{\op{D}}_r(\sff{E})\big]^{\vir}$ and $\big[\wt{\op{D}}_r(\sff{F})\big]^{\vir}$ by $\llb{\sff E, r}$ and $\llb{\sff F, r}$, respectively. 
    \end{prop}

\begin{proof} Take an $(\sff E, \sff F)$-suitable blow up $\nu\colon Y\to X$ in which $Y$ is nonsingular. Let $V:=\nu^{-1}(U)\subset Y$. Then, by the assumption $$\nu^*\sff{E}|_{V} \cong \tau^{\le 1}( \nu^* \sff{E}) |_{V}, \qquad \nu^* \sff{F}|_{V} \cong \tau^{\le 1}( \nu^* \sff{F}) |_{V}$$ are in $\Perf(V,0,1)$. Thus, we may apply Theorem \ref{compGT} on the nonsingular $V$. The claim then follows from Proposition  \ref{prop:red1st}, the projection formula and Lemma \ref{lem:indep}.
\end{proof}

\begin{remark} One could give a more direct proof Proposition \ref{prop:comparisonSpecial} that does not require a perfect obstruction theory and the resolution of singularities (cf. Appendix \ref{quasiinvariance}).
\end{remark}

\subsection{Deformation invariance} \label{sec:defInvar}
For $\sff E\in \Perf(X,0,2)$ the integrals against the classes $\llb{\sff{E},r}_k$ and $\llbb{\sff E,r}$ are invariant under \emph{certain deformations} of $\sff E$ and $X$ that we will make more precise in this section. One should compare the situation with the integrals against the classes obtained from reduced perfect obstruction theories, whose deformation invariance is rather delicate (cf. \cite[Remark 3.1]{Kool_2014}). For  more general deformations, we will work out explicit \emph{correction terms} required in order to retain the deformation invariance. We only discuss the first classes $\llb{\sff E,r}_k$ in this section. The second class can be treated similarly.

Suppose $\cc X \to C$ is a flat projective morphism of relative dimension $d$ and with integral fibers  to a complex nonsingular quasi-projective pointed curve $(C,0)$, such that $\cc X_0=X$.   Let   
$$\mathscr{E_\bu}\simeq \{\mathscr E_0\xr{\s_0} \mathscr E_1 \xr{\s_1} \mathscr E_2\}$$ be a complex of vector bundles over $\cc X$, such that $\scr E_\bu|_{\cc X_0}\simeq \sff E$. We  apply Lemma \ref{lem:blhd1} to the presentation $$\mathscr{E}_1\to \mathscr{E}_2\to \sh^2(\mathscr{E}_\bu)\to 0,$$ to construct an $\scr E_\bu$-suitable blow up $\pi \colon \cc Y \to \cc X$, which is also fibered over $C$. By construction $\cc Y$ is an irreducible component of $$\Gr(\scr C, c) \subset \Gr(\scr E^*_1,c)$$ dominating $\cc X$, where $\scr C:=\coker(\s_1^*)$ is of generic rank $c$. Let $\cc U\subset \cc X$ be the open subset over which $\scr C$ is locally free of rank $c$. We assume that $\cc U_0:=\cc U\cap \cc X_0\neq \emptyset$, or equivalently $\sff h^2_{\scr E_\bu}=\sff h^2_{\sff E}$. Thus, $\cc C:=\scr C|_{\cc X_0}$ is locally free of rank $c$ over $\cc U_0$.
Since $\cc Y$ is irreducible, after shrinking $C$ (i.e. discarding finitely many points), we may assume that all the fibers of $\cc Y$ are integral, except for possibly the central fiber $\cc Y_0$.

Let $Y\subset \cc Y_0$ be the irreducible component dominating $X=\cc X_0$. 
By our assumption $\cc U_0\neq \emptyset$, the induced morphism $\nu\colon Y\to X$ is an $\sff E$-suitable blow up. The main source of complications in the deformation invariance that we are formulating in this section is the possibility for $\cc Y_0$ to be reducible, that is $\cc Y_0\neq Y$. 
This can happen if some components of the exceptional divisor on $ \cc Y$ lie over $0\in C$. Let $\pi_0:=\pi|_{\cc Y_0}\colon \cc Y_0\to \cc X_0$.
\begin{definition}
 In the set up above,  we say the class $\llb{\sff E,r}_k$ obeys deformation invariance on $C$ if $\llb{\scr E_\bu,r}_k|_{\cc X_0}=\llb{\sff E,r}_k$.
\end{definition}

The following proposition, if applied to 2-term complexes of vector bundles (cf. Section \ref{virres}) in case $\cc X\to C$ is smooth,  gives the deformation invariance for the classes arising from perfect obstruction theories.
\begin{prop} \label{prop:hd<=1central}
In either of the following two cases the class $\llb{\sff E,r}_k$ obeys deformation invariance on $C$ for each $ k\geq0 $.
\begin{enumerate}
    \item  $\hd(\sh^2(\scr E_\bu))\leq 1$.
    \item $\scr E_\bu$ is $r$-special.
\end{enumerate}
\end{prop}
\begin{proof}
(1) Let $E_i:=\scr E_i|_{\cc X_0}$ and $K:=\ker(\s_1|_{\cc X_0})$. The assumption on the homological dimension implies that $\cc Y=\cc X$, and also $\hd(\sh^2(\sff E))\leq 1$. Therefore, $K$ and $\scr K:=\ker(\s_1)$ are locally free and $K\cong \scr K|_{\cc X_0}$.  As a result,  $$\tau^{\le 1}(\scr E_\bu) |_{\cc X_0}\simeq \{E_0\to K\}\simeq \tau^{\le 1}(\sff E).$$
Consider the fiber diagram
$$\xymatrix{\wt{\op{D}}_r(\sff E) \ar@{^(->}[r] \ar[d]& \Gr(E_0^*,r)\ar[d]^-s\\
\Gr(E_0^*,r)\ar[r]^-{0} & K\otimes Q_{\Gr}}$$ in which $s$ is the section induced by the composition $Q^*_\Gr\to E_0\to K$. We can use the projection formula and the analogous diagram for $\scr E_\bu$ to write
\begin{align*}
{\llb{\scr E_\bu}_k|}_{\cc X_0}
&=\left(s_k\big(\sh^2(\cc E_\bu)\otimes Q_{\wt{\op{D}}_r}\big)\cap 0^![\Gr(\scr E_0^*,r)]\right)|_{\cc X_0}\\
&=s_k\big(\sh^2(\sff E)\otimes Q_{\wt{\op{D}}_r}\big)\cap 0^![\Gr(E_0^*,r)]=\llb{\sff E}_k.  
\end{align*}
(2) Let $\cc V$ be an open neighborhood of $\op{D}_r(\scr E_\bu)$ in $\cc X$, such that $\scr E_\bu|_{\cc V}\in\Perf(\cc V,0,1)$. Let $\cc V_0:=\cc V\cap \cc X_0$ and $E_i:=\scr E_i|_{\cc X_0}$. Then, as in part (1),
\begin{align*}
{\llb{\scr E_\bu}_k|}_{\cc X_0}
={\llb{\scr E_\bu|_{\cc V}}_k|}_{\cc V_0}&=s_k\big(\sh^2(\sff E)|_{\cc V_0}\otimes Q_{\wt{\op{D}}_r}\big)\cap 0^![\Gr(E_0^*|_{\cc V_0},r)]\\
&=\llb{\sff E|_{\cc V_0}}_k=\llb{\sff E}_k.  
\end{align*}

\end{proof}

We now formulate a general formula that indicates the failure of deformation invariance by means of some corrections terms arising from the extra components of $\cc Y_0$. Let $S$ be the set of irreducible components of $\overline{\cc Y_0\setminus Y}$. For any $W\in S$ with its reduced induced structure, denote by $m(W)$ the multiplicity of $W$ in the Cartier divisor $\cc Y_0\subset \cc Y$. Let  $\wt{\op{D}}_{r,W}:=\wt{\op{D}}_{r}(\tau^{\le 1}(\pi^*\scr E_{\bu})|_W)$.
\begin{theorem} \label{thm:defInv} For any $k\ge 0$,
    we have \begin{align*}&\llb{\scr E_\bu,r}_k|_{\cc X_0}=\llb{\sff E,r}_k\\&+\sum_{W\in S}m(W)\pi_{0*}\left(s_k\big(\sh^2(\pi^*\scr E_{\bu})\otimes Q_{\wt {\op{D}}_{r,W}}\big)\cap \llb{\tau^{\le 1}(\pi^*\scr E_{\bu})|_W,r}\right).\end{align*}
    In particular, if for every $W\in S$,  $\op D_r(\pi_0^* \sff E)\cap W=\emptyset$, then $\llb{\sff E,r}_k$ obeys deformation invariance on $C$. This is the case for example if  $Y=\cc Y_0$ that is, when $S=\emptyset$.
\end{theorem}
\begin{proof} For simplicity, we only consider the case $k=0$ here, and leave the straightforward modifications for $k>0$ to the reader. 
    If we define $\scr K:=\ker(\pi^*\s_1)$ and $K=\ker((\pi^*\s_1)|_Y)$ we have that $\scr K|_{\cc Y_0}\subset \ker(( \pi^*\s_1)|_{\cc Y_0})$, and only on further restriction to the component $Y\subset \cc Y_0$, we get the equality $\scr K|_{Y}=K$. Using the fiber diagram
$$\xymatrix{\wt{\op{D}}_r(\pi_0^*\sff E) \ar@{^(->}[r] \ar[d] & \Gr(\pi_0^*E_0^*,r)\ar[d]^-{s}\\
\Gr(\pi_0^*E_0^*,r)\ar[r]^-0 & \scr K|_{\cc Y_0}\otimes Q_{\Gr},}$$ 
we see that
    \begin{align*}
        \llb{\pi^*\scr E_\bu}|_{\cc Y_0}&=0^![\Gr(\pi_0^* E_0^*,r)]\\&=0^!\left([\Gr(\pi_0^* E_0^*|_Y,r)]+\sum_{W\in S}m(W)[\Gr(\pi_0^* E_0^*|_W,r)]\right).
    \end{align*}
The first claim now follows by 
pushing both sides down via $$\pi_0\colon \wt {\op D}_r(\pi_0^*\sff E)\to  \wt {\op D}_r(\sff E).$$ The second claim also follows, because in that case $ \llb{\tau^{\le 1}(\pi^*\scr E_\bu)|_{W},r}$ vanishes for any $W\in S$, and so no terms in the summation over $S$ will contribute.
\end{proof}

\begin{remark}
    The second part of Theorem \ref{thm:defInv} says that in order to guarantee deformation invariance for $\llb{\sff E,r}_k$, one only needs to ensure that the extra components of the central fiber $\cc Y_0$ of the blow up are kept away from the degeneracy locus of interest in $\cc Y$.
\end{remark}

\noindent \textbf{Example}. Let  $\mathcal X\To C$ be as above and $D\subset\cc X$ be a Cartier divisor meeting $X=\cc X_0$ transversely, with $Z:=X\cap D$. The following 3-term Koszul complex denoted by $\scr E_\bu$ gives a locally free resolution
$$0\To\cO_{\cc X}(-X-D)\xr{{-s_D\choose  s_X}} \cO_{\cc X}(-X)\oplus \cO_{\cc X}(-D)\xr{(s_X,s_D)}\cO_{\cc X}\To\cO_Z\To 0$$
with $s_X,s_D$ the natural inclusions.  The complex $\scr E_\bu$ is exact off $Z$ and $\sh^0(\scr E_\bu^\vee)\cong \cO_Z(D+X)$, so that
$$\op{D_0}(\scr E_\bu)=\cc X,\quad \op{D_1}(\scr E_\bu)= Z,\quad\op{D_2}(\scr E_\bu)=\emptyset.$$
The blow up $\pi:\cc Y:=\op{Bl}_Z\cc X\to \cc X$ is $\scr E_\bu$-suitable; denote by $W$ its exceptional divisor, and by $\wt X,\wt D$ the proper transforms of $X$ and $D$. The complex $\pi^*\scr E_\bu$ is given by
$$\cO_{\cc Y}(-\wt X-\wt D-2W)\xr{{-s_{\wt D+W}\choose s_{\wt X+W}}} \cO_{\cc Y}(-\wt X-W)\oplus \cO_{\cc Y}(-\wt{D}-W)\xr{(s_{\wt X+W},s_{\wt D+W})}\cO_{\cc Y}.$$
The kernel of the right arrow is $\cO_{\cc Y}(-\wt X-\wt D-W)$, so 
$$\tau^{\leq1}(\scr E_\bu)\cong \cO_{\cc Y}(-\wt X-\wt D- 2W)\xr{s_W}\cO_{\cc Y}(-\wt X-\wt D-W)$$
from which we find that
$$\llb{\pi^*\scr E_\bu,1}=[W]\in A_{d}(W).$$
and $\llb{\scr E_\bu,1}=0$ because $\dim Z<d$. Also, for $p\in C\setminus\{0\}$
$$\llb{\scr E_\bu|_{\cc X_p},1}=\pi_*\llb{\pi^*\scr E_\bu|_{\cc Y_p},1}=\pi_*(W\cdot \cc Y_p)=0.$$
At $p=0$, we have $\cc Y_0=\wt X+W$ and  $W\cdot \cc Y_0=0$. Therefore,
\begin{align*}
\pi_*\llb{\tau^{\leq1}(\pi^*\scr E_\bu)|_{W}}=&\pi_*(W\cdot W)=-[Z]\in A_{d-1}(Z),\\
    \llb{\scr E_\bu|_{X}}=&\pi_*\llb{\pi^*\scr E_\bu|_{\wt X}}=[Z]\in A_{d-1}(Z),
\end{align*} whose sum is $0=\llb{\scr E_\bu,1}|_{\cc X_0}$ as in Theorem \ref{thm:defInv}.

\begin{corollary}
The classes $\llb{\scr E_\bu|_p,r}_k$ for $p\in C\setminus 0$ obey deformation invariance.
\end{corollary}
\begin{proof}
In our set up, for each $p\in C\setminus 0$, $\cc Y_p$ is irreducible and an $\scr E_\bu|_p$-suitable blow up.   
\end{proof}

\begin{remark}
    Without shrinking the base curve $C$ in our setting, the corollary above says that the classes $\llb{\scr E_\bu|_p,r}_k$ obey deformation invariance for all $p\in C$ except for possibly finitely many points.
\end{remark}

\begin{corollary}\label{cor:irredDefInv}
If $\Gr(\cc C,c)$ has a unique irreducible component of maximal dimension $d$, then $\llb{\sff E,r}_k$ obeys deformation invariance on $C$ (cf.  Corollary \ref{cor:1irred} for sufficient condition).    
\end{corollary}\begin{proof}
    In this case, $Y=\cc Y_0$, so the claim follows from the second part of Theorem \ref{thm:defInv}.
\end{proof}

For the next corollary, let $s:=\rk(\scr E_2)-\rk(\scr E_1)+c$, and $\op{D}_t:=\op D_{t}
    ((\scr E_\bu)^\vee)\subset \cc X$. If $\cc I_t\subset \cc O_{\cc X}$ is the ideal sheaf of $\op D_t$, for any integer $n>0$, denote by $\op D^n_t\subset \cc X$ the closed subscheme defined by the ideal sheaf $\cc I_t^n$. Also, denote the normal cone of $\op{D}_t$ in $\cc X$ by $$\op{C}_{\op D_{t}
/\cc X}:=\op{Spec} (\bigoplus_{n>0}\cc I^{n-1}_t/\cc I^n_t).$$ 

\begin{corollary} \label{cor:flatDefInv}
    Suppose that either of the equivalent conditions below is satisfied.
    \begin{enumerate}
        \item For any $n>0$ and any $t>s$, each associated point of $\op D^n_{t}$ specializing to a point of $\op D_r(\scr E_\bu)$ is mapped to the generic point of $C$.
        \item For any $t>s$, each associated point of the normal cone $\op{C}_{\op D_{t}
/\cc X}$ whose image in $\cc X$ specializes to a point of $\op D_r(\scr E_\bu)$ is mapped to the generic point of $C$. 
    \end{enumerate}
    Then, $\llb{\sff E,r}_k$ obeys deformation invariance on $C$.  
 \end{corollary}

\begin{proof}
We first show the equivalence of (1) and (2).  Condition (1) (respectively, (2)) implies that $\op D^n_t$ for each $n>0$ (respectively, $\op{C}_{\op D_{t}/\cc X}$) is flat over $C$ in an open neighborhood of $\cc V$ of $\op D_r(\scr E_\bu)$. Without loss of generality, for simplicity, we assume $\cc X=\cc V$. Let $\tau$ be a uniformizing parameter for the discrete valuation ring $\cc O_{C,0}$. 
The flatness of $\op D^n_t$ is equivalent to $\tau$ being a nonzero divisor for $\cO_{\cc X}/\cc I^n_t$ and hence for $\cc I^{n-1}_t/\cc I^n_t$ for each $n>0$. Thus $\oplus_{n>0} \cc I^{n-1}_t/\cc I^n_t$ is flat  proving that (1) implies (2). Conversely, if $\oplus_{n>0} \cc I^{n-1}_t/\cc I^n_t$ is flat, so is each direct summand. Now by induction on $n$ and the short exact sequence $$0\to \cc I^{n-1}_t/\cc I^n_t\to \cO_{\cc X}/\cc I^n_t\to \cO_{\cc X}/\cc I^{n-1}_t\to 0,$$ we see that for each $n>0$, $\cO_{\cc X}/\cc I^n_t$ and hence $\op D^n_t$ is flat over $C$.

To see the claim about deformation invariance, by either of the conditions, the projective cone $\bb P(\op{C}_{\op D_{t}/\cc X})$ is flat over $C$, and moreover, $$\bb P(\op{C}_{\op D_{t}
    /\cc X})_0\cong \bb P(\op{C}_{\op D_{t}
    (\sff E^\vee)/\cc X_0}).$$ By Remark \ref{rmk:LiHu}, an alternative $\scr E_\bu$-suitable blow up (cf. Lemma \ref{lem:indep}), also denoted by $\pi:\cc Y\to \cc X$, can be obtained   by blowing up $\cc X$ at the degeneracy loci $\op D_t=\op D_{t}
    ((\scr E_\bu)^\vee)$ for $t>s$, starting with the deepest one. Thus, the exceptional divisors of the blow ups given by the projective cones above form flat families over $C$. In particular, in the notation of Theorem \ref{thm:defInv} $W\cap\op D_r(\pi_0^*\sff E)=\emptyset$ for any $W\in S$, so the claim follows from the second part of that theorem.
\end{proof}

\begin{remark}
\begin{enumerate}[i)]
\item A special case of Condition (2) in Corollary \ref{cor:flatDefInv} occurs, when each degeneracy locus $\op D_{t>s}
    ((\scr E_\bu)^\vee)$ is reqularly embedded and flat over $C$ in an open  neighborhood of $\op D_r(\scr E_\bu)$. Then, the normal cone $\op{C}_{\op D_{t}
    ((\scr E_\bu)^\vee)/\cc X}$ is a vector bundle over $\op{D}_{t}
    ((\scr E_\bu)^\vee)$, and hence also flat over $C$ in that neighborhood.   
    \item 
    In \cite[Theorem 4.1]{colley2001detecting}, another sufficient condition for ensuring the flatness of the normal cones needed in the proof of  Corollary \ref{cor:flatDefInv} is proven that uses a Segre class test and the notion of internal flatness.
    \item Condition (2) in Corollary \ref{cor:flatDefInv} is similar to the notion of  the normal flatness introduced by Hironaka in \cite[Chapter II]{Hir}.   
    \end{enumerate}\end{remark}

\noindent \textbf{Example}. We give an example of a 1-parameter flat family of subschemes, whose normal cone  in the ambient space does not give a flat family of normal cones. This justifies imposing  Condition (1) or (2) in Corollary \ref{cor:flatDefInv}.
Let $$\cc X:=\bb A^3_\bb C=\op{Spec} \bb C[\tau, x,y]\to C:=\op{Spec} \bb C[\tau]$$ be the projection, and consider the ideal $\cc I:=\langle x^2,xy,y^2-\tau x\rangle$. Then, $\op D:=\op{Spec}(\cc O_X/\cc I)$ is flat over $C$, whereas,  $\op D^2:=\op{Spec}(\cc O_X/\cc I^2)$ is not. In fact, $$\tau x^3=-x^2(y^2-\tau x)+x^2y^2\in \cc I^2$$ showing that $\tau$ is a zero divisor for $\cc O_X/\cc I^2$. 
We use Macaulay2 to calculate the normal cone  $$\op C_{\op D/\cc X}=\op{Spec} \bb C[\tau, x,y,u,v, w]/(\cc I+ \cc J),$$
where $\cc J:=\langle yu-\tau v+xw, xu-yv\rangle$. The radical of $\cc I+\cc J$ is calculated to be $\langle \tau v, x,y \rangle$. Thus, $(\op{C}_{\op D/\cc X})_{\op{red}}$ has two 3-dimensional components: one is flat over $C$, and the other one is contracted to the closed point $0=\langle \tau \rangle \in C$.

\subsection{Comparison theorems}
\label{sec:moddual}
 For a given $\sff E\in \Perf(X,0,2)$, fix a locally free resolution
 $$\sff E\simeq \{E_0\xrightarrow[]{\sigma_0}E_1 \xrightarrow[]{\sigma_{1}}E_2\}.$$
    Let $\nu:Y\to X$ be an $(\sff E, \sff E^\vee)$-suitable blow up i.e.  $$\hd(\sh^2(\nu^*\sff{E}))\leq1,\qquad \hd(\sh^{0}(\sff{\nu^*E}^\vee))\leq1.$$ Then $K:=\ker(\nu^*\sigma_1)$ and $C:=\ker(\nu^*\sigma^*_{0})$ as well as  $\nu^*E_1/K$  and $\nu^*E^*_1/C$ are locally free.  Since $\nu^*\sigma_1^*$ factors as $\nu^*E^*_{2}\to C\to \nu^*E^*_1$,  $\nu^*\sigma_1$ factors as $\nu^*E_{1}\to C^*\to \nu^*E_{2}$, so there is an induced surjection $C^*\twoheadrightarrow \nu^*E_1/K$. Similarly, there is an induced surjection $K^*\twoheadrightarrow \nu ^*E_1^*/C$, where after dualizing  give an injection of vector bundles $\nu ^*E_1/C^*\hookrightarrow K$. These maps fit into the following commutative diagram with exact rows:
\[\begin{tikzcd}
	0 & {(\nu^*E_1^*/C)^*} & {E_1} & {C^*} & 0 \\
	0 & K & {E_1} & {\nu^*E_1/K} & 0
	\arrow[from=1-1, to=1-2]
	\arrow[from=1-2, to=1-3]
	\arrow[hook, from=1-2, to=2-2]
	\arrow[from=1-3, to=1-4]
	\arrow[shift right, no head, from=1-3, to=2-3]
	\arrow[no head, from=1-3, to=2-3]
	\arrow[from=1-4, to=1-5]
	\arrow[two heads, from=1-4, to=2-4]
	\arrow[from=2-1, to=2-2]
	\arrow[from=2-2, to=2-3]
	\arrow[from=2-3, to=2-4]
	\arrow[from=2-4, to=2-5]
\end{tikzcd}\]
Using the snake lemma we find a locally free sheaf on $Y$ \begin{equation}\label{def:T}
    T:=\ker(C^*\twoheadrightarrow \nu^* E_1/K)\cong \coker(\nu^* E_1/C^*\hookrightarrow K)
\end{equation}
of rank
$\op{rk}(T)=k-e_1-c=\op{rk}(\ker\sigma_1)-\op{rk}(\im \sigma_{0})=\rkh{\sff{E}}^1.$

\begin{lemma}\label{lem:hd1surjections}
    Given a surjection $s:\cF\twoheadrightarrow \cG$ of coherent sheaves $\cF, \cG$ on $X$ of $\hd\leq1$, we have that $\hd(\ker(s))\le 1$.
\end{lemma}
\begin{proof}
     Let $V$ be any locally free sheaf surjecting onto $\cF$.     We have a commutative diagram with exact rows
    \[\begin{tikzcd}[cramped,row sep=scriptsize]
	0 & {K_{\sff \cF}} & V & {\cF} & 0 \\
	0 & {K_{\sff \cG}} & V & {\cG} & 0
	\arrow[from=1-1, to=1-2]
	\arrow[from=1-2, to=1-3]
	\arrow["t",from=1-2, to=2-2]
	\arrow[from=1-3, to=1-4]
	\arrow[shift right, no head, from=1-3, to=2-3]
	\arrow[no head, from=1-3, to=2-3]
	\arrow[from=1-4, to=1-5]
	\arrow["s", two heads, from=1-4, to=2-4]
	\arrow[from=2-1, to=2-2]
	\arrow[from=2-2, to=2-3]
	\arrow[from=2-3, to=2-4]
	\arrow[from=2-4, to=2-5]
\end{tikzcd}\]
    in which  $K_{\cF}$ and $K_{ \cG}$ are locally free by the homological dimension assumption. By the snake Lemma, $t\colon K_{\sff \cF}\to K_{\sff \cG}$ is injective, and
    $\ker(s)\cong\coker(t)$
    has homological dimension $\leq 1$. 
\end{proof}

\begin{lemma}\label{lem:truncatecone}
    Let $\sff E,\sff F\in\Perf(X,0,2)$ and $u:\sff F\to \sff E$ be map in the derived category with $\sff G:=\Cone(u)$. If $\sff G\in\Perf(X,1,2)$ and $\nu:Y\to X$ is any $(\sff E, \sff F, \sff G)$-suitable blow up, there is an exact triangle
    $$\tau^{\leq1}(\nu^*\sff F)\To\tau^{\leq1}(\nu^*\sff E)\To A[-1]$$
    where $A:=\ker(\sh^1(\nu^*\sff G)\to\sh^2(\nu^*\sff F))$ is locally free.
\end{lemma}
\begin{proof}
    Truncating
    $\nu^*\sff{F}\To\nu^*\sff{E}\To \nu^*\sff G$, we 
    get the exact triangle    
        \beq{eq:2triang}\tau^{\leq1}(\nu^*\sff{F})\To\tau^{\leq1}(\nu^*\sff{E})\To  A.\eeq
        By assumptions, we also have an exact sequence
        $$ 0\To A\To \sh^1(\nu^*\sff G)\To\sh^2(\nu^*\sff F)\xr{\sh^2(\nu^* u)}\sh^2(\nu^*\sff E)\To \sh^2(\nu^*\sff G)\To0$$ 
        By Lemma \ref{lem:hd1surjections} for $\cK:=\ker(\sh^2(\nu^*\sff E)\to\sh^2(\nu^*\sff G))$, $\hd(\cK)\leq1$. In turn, $\cK$ fits in a short exact sequence
    $$0\To \ker(\sh^2(\nu^*u))\To \sh^2(\nu^*\sff F)\To \cK\To 0,$$
    and so by Lemma \ref{lem:hd1surjections} again,  $\hd(\ker(\sh^2(\nu^*u)))\le 1$.

Since $\hd(\sh^2(\nu^*\sff G))\leq1$, fixing any resolution 
    $\sff{G}\simeq\{G_1\to G_2\}$, from the exact sequence $$0\To\sh^1(\nu^*\sff G)\To G_1\To G_2\To\sh^2(\sff \nu^*G)\To 0,$$
    we see that $\sh^1(\nu^*\sff G)$ must be locally free. 
    The cohomology sequence  of \eqref{eq:2triang}, can be truncated to the short exact sequence
    $$0\To A\To \sh^1(\nu^*\sff G)\To\ker(\sh^2(\nu^*u))\To0.$$
    As $\sh^1(\nu^*\sff G)$ is locally free and $\hd(\ker(\sh^2(\nu^*u)))\leq1$, $A$ is locally free.
\end{proof}

\begin{theorem}\label{thm:compdoub}
Let $\sff E,\sff F\in\Perf(X,0,2)$ and $u:\sff F\to \sff E$ be a map in the derived category with $\sff G:=\Cone(u)$. Suppose that $\nu:Y\to X$ is any $(\sff E,\sff F)$-suitable blow up and $r\ge 1$ in an integer.
\begin{enumerate}
    \item If $\sff{G}$ is represented by a locally free sheaf $G$ in degree 0, then there is a canonical closed immersion $\iota:\widetilde{\op{D}}_r(\sff{F})\hookrightarrow\widetilde{\op{D}}_r(\sff{E})$. For any $k\ge 0$
    $$\iota_*\llb{\sff{F},r}_k=c_{r(\sff e-\sff f)}({G}\otimes Q_{\widetilde{\op{D}}_r(\sff{E})})\cap \llb{\sff{E},r}_k.$$
    Moreover, $\widetilde{\op{D}}_r(\sff{F}^\vee)\cong\widetilde{\op{D}}_r(\sff{E}^\vee)$ with 
    $$\llbb{\sff F,r}=\llbb{\sff E,r}.$$
    \item If $\sff{G}\in\Perf(X,1,2)$, there is a canonical isomorphism $\widetilde{\op{D}}_r(\sff{F})\cong \widetilde{\op{D}}_r(\sff{E})$. If $\nu$ is in addition $\sff G$-suitable, then
$$\llb{\sff{E},r}=\nu_*\left(c_{r(\sff f-\sff e+\rkh{\sff E}^2-\rkh{\sff F}^2)}(A\otimes Q_{\widetilde{\op{D}}_r(\nu^*\sff{E})})\cap \llb{\nu^*\sff{F},r}\right),$$ where 
       $A:=\ker( \sh^1(\nu^*\sff G)\to \sh^2(\nu^*\sff{F}))$ is locally free of rank 
       $\sff f-\sff e+\rkh{\sff E}^2-\rkh{\sff F}^2$. 
       \begin{itemize}
           \item[(i)] If $\sh^2(\nu^*u)$ is an isomorphism, then $ \nu^*\sff G\simeq A[-1]$ and 
$$\llb{\sff{E},r}=s_{r(\sff f-\sff e)}(\sff G\otimes Q_{\widetilde{\op{D}}_r(\sff{E})})\cap \llb{\sff{F},r}.$$
           \item[(ii)] If $\sh^2(\sff E)=0$, then $\sff G$ is represented by a locally free sheaf $G$ in degree 1 and
           \begin{align*}
\llb{\sff{E},r}&=\sum_{i=0}^{r(\sff f-\sff e-\rkh{\sff F}^2)}c_{{r(\sff f-\sff e-\rkh{\sff F}^2)-i}}(G\otimes Q_{\wt{\op{D}}_r(\sff F)})\cap\llb{\sff F,r}_i
           \end{align*}
       \end{itemize}
    \item If $\sff G$ is represented by a locally free sheaf $G$ in degree $2$, then 
    $$\widetilde{\op{D}}_r(\sff{F})\cong \widetilde{\op{D}}_r(\sff{E})\text{ and }\llb{ \sff{F},r}=\llb{ \sff{E},r}.$$
\end{enumerate}
    \end{theorem}

    \begin{proof}
        For (1), the first claim follows from the projection formula and Proposition \ref{prop:comparisonSpecial} applied to  the exact triangle
        $$\tau^{\leq1}(\nu^*\sff{F})\to\tau^{\leq1}(\nu^*\sff{E})\to \nu^*\sff G$$ 
        by noting that $\sh^2(\sff F)\cong\sh^2(\sff E)$.
        The second claim follows easily, because  $\tau^{\geq2}(\nu^*\sff{F})\simeq\tau^{\geq2}(\nu^*\sff{E})$.

        For (2), we are in the setup of Lemma \ref{lem:truncatecone}. Working with the exact triangle as in that lemma, the claim is proven by an application of  Proposition \ref{prop:comparisonSpecial}. 
        In the case $\sh^2(\nu^*u)$ is an isomorphism, $\sh^2(\sff G)=0$ and $$A=\ker(\sh^1(\nu^*\sff G)\xr{0}\sh^2(\nu^*\sff F))=\sff \nu^*G[1],$$ and the other claim follows from the projection formula. In the case $\sh^2(\sff E)=0$, from the cohomology exact sequence we see that $\sh^1(\sff G)\xr{}\sh^2(\sff F)$ is surjective and $\sh^2(\sff G)=0$. Therefore, $\sff G\simeq G[-1]$ for some locally free $G$. Since $\hd(\sh^2(\nu^*\sff F))\leq 1$, in $K$-theory
        $$A=\ker(G\twoheadrightarrow \sh^2(\nu^*\sff F))=G- \sh^2(\nu^*\sff F).$$
        Thus, we have 
                   \begin{align*}
\llb{\sff{E}}&=\nu_*\left(c_{r\left(\sff f-\sff e-\rkh{\sff F}^2\right)}((\nu^*G-\sh^2(\nu^*\sff E))\otimes Q_{\widetilde{\op{D}}_r(\nu^*\sff{E})})\cap \llb{\nu^*\sff{F}}\right),
           \end{align*}
which give the claimed formula after expanding the Chern class of the difference and applying the projection formula.
        
        For $(3)$, we have $\tau^{\leq1}(\sff{E})\simeq\tau^{\leq1}(\sff F)$ from which the claim easily follows. 

        \end{proof}

The next proposition demonstrates how the first and the second classes are related in general. 

\begin{prop} \label{prop:aboutT}
 If $\nu\colon Y\to X$ is an $(\sff E, \sff E^\vee)$-suitable blow up, then  \begin{align*}&\llb{\sff{E},r}=\nu_*\left(c_{r\rkh{\sff{E}}^1}(T\otimes Q_{\widetilde{\op{D}}_r(\sff{E})})\cap\llbb{\nu^*\sff{E}^\vee, r}\right),\\
 &\llb{\sff E^\vee,r}=\nu_*\left(c_{r\rkh{\sff{E}}^1}(T^*\otimes Q_{\widetilde{\op{D}}_r(\sff E^\vee )})\cap\llbb{\nu^*\sff{E},r}\right),
 \end{align*}
    where $T$, defined in \eqref{def:T}, is locally free of rank $\rkh{\sff E}^1$ with $K$-theory class 
    \beq{eq:ktheoryT}\sh^2(\nu^*\sff E)+\sh^0(\nu^*\sff E^\vee)^\vee-\nu^*\sff E,\eeq which is independent of the choice of a locally free resolution for $\sff E$.
        
\end{prop}
\begin{proof}
    In the notation of \eqref{def:T}, there is an induced map of complexes $$u:\{\nu^*E_0\to\nu^*E_1/C^*\}\to\{\nu^*E_0\to K\}$$ with $\Cone(u)\simeq T[-1]$. The first formula now follows by applying part (2) of Theorem \ref{thm:compdoub}. The second formula is similarly proven by switching the roles of $\sff{E}$ and $\sff E^\vee$. The $K$-theory class of $T$ is $C^*+K-\nu^*E_1$ from which the last claim easily follows. 
\end{proof}

        

Suppose 
that we are in the setting of part (2) of Theorem \ref{thm:compdoub} and suppose that  
$\sff{E}$ is $r$-special, that is, it satisfies Condition \eqref{UDr} with $\sff E|_U\in\Perf(U,0,1)$ for some open neighborhood of $U$ of $\op D_r(\sff E)$. Then, $\sff G|_U=\Cone(u)|_U$ is quasi- isomorphic to a locally free sheaf $G$  of rank $g=\sff f-\sff e$ sitting in degree $1$. The following result extends part (2-ii) of Theorem \ref{thm:compdoub}.

\begin{prop}\label{prop:deeprefined}
Let $\sff E,\sff F\in\Perf(X,0,2)$ and $u:\sff F\to \sff E$ be map in the derived category with $\sff G:=\Cone(u)$. Suppose that $\sff{E}$ is $r$-special for some $r\ge 1$ and $\nu:Y\to X$ is any $\sff F$-suitable blow up. Then, 
   \begin{align*}
       \llb{\sff{E},r}&=\nu_*\left(c_{r(g-\rkh{\sff{F}}^2)}\big((\nu^* G-\sh^2(\nu^*\sff{F}))\otimes Q_{\wt{\op{D}}_r(\sff{F})}\big)\cap\llb{\nu^*\sff{F},r}\right).\\
       &=\sum_{i=0}^{r(\sff f-\sff e-\rkh{\sff F}^2)} c_{r(g-\rkh{\sff{F}}^2)-i}(G\otimes Q_{\wt{\op D}_r(\nu^*\sff{F})})\cap\llb{\sff{F},r}_i,
   \end{align*}
where $G$ is as defined before the proposition. If $\nu$ is $(\sff F,\sff F^\vee)$-suitable, then 
    \begin{align*}\llb{\sff{E},r}&=\nu_*\left(c_{r(g-\rkh{\sff{F}}^2+\rkh{\sff{F}}^1)}\big((T+\nu^* G-\sh^2(\nu^*\sff{F}))\otimes Q_{\wt{\op{D}}_r(\sff{F})}\big)\cap\llbb{\nu^*\sff{F}^\vee,r}\right),
   \end{align*} where the vector bundle $T$ is as defined in \eqref{def:T}.
\end{prop}
\begin{proof}
    Since ${G}$  is locally free and $\hd(\sh^2(\nu^*\sff{F}))\le 1$, as before $$A:=\ker( \nu^*G\twoheadrightarrow \sh^2(\nu^*\sff{F}|_U))$$ is  locally free, and we get the exact triangle
        $$\tau^{\leq1}(\nu^*\sff{F}|_U)\to\tau^{\leq1}(\nu^*\sff{E}|_U)\to A[-1].$$
    The claim for the first class now follows from part (2-ii) of Theorem \ref{thm:compdoub}. 
    
    To see the second assertion, note that by Proposition \ref{prop:aboutT} and the first part, we have 
$$\llb{\sff{E}}=\nu_*\left(c_{r\sff h^1_{\sff E}}(T\otimes Q_{\wt{\op{D}}_r(\sff{F})})\,c_{r(g-\rkh{\sff{F}}^2)}\big(A\otimes Q_{\wt{\op{D}}_r(\sff{F})}\big)\cap\llbb{\nu^*\sff{F}^\vee}\right).$$
    Since $A$ is locally free of rank $g-\rkh{\sff{F}}^2$ over the open set $\nu^{-1}(U)$, which contains $\wt{\op{D}}_r(\nu^*\sff{F})$, we can combine the two top Chern classes to finish the proof.
\end{proof}

The last comparison formula we prove in this section is based on the following set up.  Let $\sff{E},\sff{F},\sff{G},\sff{H}\in\Perf(X,0,2)$ sitting in a commutative diagram in the derived category
\begin{equation}\label{eq:squarecompare}
    \begin{tikzcd}
	{\sff{G}} & {\sff{E}} \\
	{\sff{F}} & {\sff{H}}
	\arrow["\alpha", from=1-1, to=1-2]
	\arrow["\beta"', from=1-1, to=2-1]
	\arrow["\gamma", from=1-2, to=2-2]
	\arrow["\delta", from=2-1, to=2-2]
\end{tikzcd}
\end{equation}
Suppose that  $\Cone(\a)$ is a locally free sheaf in degree 0, $\Cone(\gamma)\in\Perf(X,1,2)$ and 
$\Cone(\delta),\Cone(\beta)\in\Perf(X,0,2)$. Let $\nu:Y\to X$ be any $(\sff E,\sff F,\sff G, \sff H, \Cone(\gamma))$-suitable blow up.

\begin{theorem}
   \label{thm:comparesquare}
    In the set up of the last paragraph, $$G :=\tau^{\leq1}(\nu^*\sff E)-\tau^{\leq1}(\nu^*\sff F)$$ is represented in $K$-theory by a vector bundle of rank $\sff e-\sff f+\rkh{\sff E}^2-\rkh{\sff F}^2$. There is a canonical closed immersion $i:\wt{\op{D}}_r(\sff F)\hookrightarrow \wt{\op{D}}_r(\sff E)$ and for any $r\ge 1$,
    $$i_*\llb{\sff F,r }=\nu_*\left(c_{\sff e-\sff f+\rkh{\sff E}^2-\rkh{\sff F}^2}(G\otimes Q_{\op{\wt D_r(\sff E)}})\cap\llb{\nu^*\sff E, r}\right).$$
 
\end{theorem}
\begin{proof}
    After pulling back \eqref{eq:squarecompare} to $Y$ and truncating, we get the commuting diagram
    \[\begin{tikzcd}
	{\tau^{\leq1}(\nu^*\sff{G}}) & {\tau^{\leq1}(\nu^*\sff{E})} \\
	{\tau^{\leq1}(\nu^*\sff{F})} & {\tau^{\leq1}(\nu^*\sff{H}).}
	\arrow["\nt\alpha", from=1-1, to=1-2]
	\arrow["\nt\beta"', from=1-1, to=2-1]
	\arrow["\nt\gamma", from=1-2, to=2-2]
	\arrow["\nt\delta", from=2-1, to=2-2]
\end{tikzcd}\]
By the assumption and Lemma \ref{lem:truncatecone}, $$\Cone(\nt\a)\simeq B,\qquad \Cone(\nt\gamma)=A[-1]$$ for some locally free sheaves  $A$ and $B$. We then get the following diagram 
\begin{equation}\label{eq:diagutrunc}
    \begin{tikzcd}[cramped,]
	\tau^{\leq1}(\nu^* \sff G)& {\tau^{\leq1}(\nu^*\sff{E})}  & B\\
	{\tau^{\leq1}(\nu^*\sff{F})} & {\tau^{\leq1}(\nu^*\sff{H})}\\
	& {A[-1]}
    \arrow[ from=1-1, to=1-2]
     \arrow[ from=1-1, to=2-1]
	\arrow["{\nt\gamma}"', from=1-2, to=2-2]
	\arrow["u", dotted, from=2-1, to=1-2]
	\arrow["{\nt\delta}", from=2-1, to=2-2]
    \arrow[from=1-2, to=1-3]
	\arrow["v"',curve={height=10pt}, dotted, from=2-2, to=1-2]
	\arrow[from=2-2, to=3-2]
	\arrow[curve={height=10pt}, dotted, from=3-2, to=2-2]
\end{tikzcd}
\end{equation}
in which the top row and the middle column are exact triangles.
We may use Jouanolou trick: after pulling back the diagram to an affine bundle over $Y$ which is an affine variety, the connecting map $A[-1]\to \tau^{\leq1}(\nu^*\sff E)[1]$ will be zero. Then, we may choose the splittings marked with the dotted arrows in the middle column. As a result, we get a map $$u:=v\circ \nt \de :\tau^{\leq1}(\nu^*\sff F)\to \tau^{\leq1}(\nu^*\sff E).$$
We show that $\Cone(u)$ is isomorphic to a vector bundle in degree 0. At a closed point $y\in Y$, taking $\sh^0$  and using the assumptions, we get a commutative diagram as follows
\[\begin{tikzcd}[cramped,row sep=scriptsize]
	& {\sh^0(\nu^*\sff{E}|_y)} \\
	{\sh^0(\nu^*\sff{F}|_y)} & {\sh^0(\nu^*\sff{H}|_y)}
	\arrow[no head, from=1-2, to=2-2]
	\arrow[shift left, no head, from=1-2, to=2-2]
	\arrow["{\sh^0(u|_y)}", from=2-1, to=1-2]
	\arrow[hook, from=2-1, to=2-2]
\end{tikzcd}\]
in which the horizontal arrow is injective. From this, ${\sh^0(u|_y)}$ is injective, and hence $\sh^{<0}(\Cone(u)|_y)=0$. On the other hand,   taking $\sh^1$ of \eqref{eq:diagutrunc} restricted to $y$ gives the diagram with commutative squares and with exact rows
\[\begin{tikzcd}[cramped]
	{\sh^1(\tau^{\leq1}(\nu^*\sff{G})|_y)} & {\sh^1(\tau^{\leq1}(\nu^*\sff{E})|_y)} & 0 \\
	{\sh^1(\tau^{\leq1}(\nu^*\sff{F})|_y)} & {\sh^1(\tau^{\leq1}(\nu^*\sff{H})|_y)} & {\sh^1(\Cone(\nt\delta)|_y).}
	\arrow[from=1-1, to=1-2]
	\arrow[from=1-1, to=2-1]
	\arrow[from=1-2, to=1-3]
	\arrow["r",hook', from=1-2, to=2-2]
	\arrow[from=1-3, to=2-3]
    \arrow["{\sh^1(u|_y)}"{description}, from=2-1, to=1-2]
    \arrow["\sh^1(\nt\delta|_y)", from=2-1, to=2-2]
	\arrow["q",from=2-2, to=2-3]
\end{tikzcd}\]
By the exactness of the top row, the middle arrow $r$ maps into $\ker(q)=\im(\sh^1(\nt\delta|_y))$. Since $\sh^1(u|_y)=\sh^1(v|_y)\circ \sh^1(\nt\delta|_y)$ and $\sh^1(v|_y) \circ r=\id$, $\im(r)\subset\im (\sh^1(\nt\delta|_y))$ implies that $\sh^1(u|_y)$ is surjective. Therefore,  $\sh^{\geq1}(\Cone(u)|_y)=0$. Since $y$ is an arbitrary closed point and $\Cone(u)$ is a perfect complex, by the basechange theorem $\Cone(u)$ is locally free in degree $0$. The claim now follows from using part (1) of Theorem \ref{thm:compdoub} .\end{proof}

\subsection{Deepest degeneracy loci}\label{sec:deepest}
Let $\sff{E}\in \Perf(X,0,2)$ with $\sff{E}\simeq\{ E_0\to E_1\to E_2\}$ a locally free resolution, and $\nu:Y\to X$ any $\sff E$-suitable blow up. Suppose we have a bundle $B$, whose dual fits into Diagram  \eqref{fig:defsurjB}. One can take $B=E_0$, but in the applications we have in mind there are more choices of $B$ that independent of our choice of resolution. 

From the map $u:q^*\sff E\to q^*B/Q^*_{\Gr}$ obtained from \eqref{fig:defsurjB} on  $q:\Gr(B^*,r^*)\to X$, we get the perfect complex 
\begin{equation}\label{def:complexdeepest}
    \sff{H}:=\Cone(u):=\{q^*E_0\xr{\phi_1} q^*B/Q^*_\Gr\oplus q^*E_1\xr{\phi_2} q^*E_2\},
\end{equation}
such that $\wt{\op{D}}_r(\sff{E})\cong \op{D}_r(\sff{H})\subset\Gr(B^*,r)$. From the fiber squares
\[\begin{tikzcd}
	{ \op{D}_r(\nu^*\sff{H})} & {\Gr(\nu^*B^*,r)} & Y\\
	{ \op{D}_r(\sff{H})} & {\Gr(B^*,r)} & X
    \arrow[draw=none, from=1-2, to=2-3]
    \arrow[ draw=none, from=1-1, to=2-2]
	\arrow["{\tilde \iota}", hook, from=1-1, to=1-2]
	\arrow["\nu", from=1-1, to=2-1]
	\arrow["{\tilde q}", from=1-2, to=1-3]
	\arrow["\nu", from=1-2, to=2-2]
	\arrow["\nu", from=1-3, to=2-3]
	\arrow["\iota", hook, from=2-1, to=2-2]
	\arrow["q", from=2-2, to=2-3]
\end{tikzcd}\]
since $q$ is flat, $q^*\sh^2(\sff{E})=\sh^2(\sff{H})$ and $\sh^2(\nu^*\sff{H})=\tilde q^* \sh^2(\nu^*\sff{E})$, we see that $\nu:\Gr(\nu^*B^*,r)\to \Gr(B^*,r)$ is an $\sff H$-suitable blow up. Applying the construction of Section \ref{sec:setupB} over a resolution of singularities of $Y$\footnote{Alternatively, applying Theorem \ref{thm:A8}, one does not have to use the resolution of singularities or a perfect obstruction theory.}, using Proposition \ref{prop:red1st} and Lemma \ref{lem:indep}, and pushing down, we get 

\beq{eq:HandE}\llb{\sff{E},r}=\llb{\sff{H},r}\in A_{d+r(\sff e-\sff h^2_{\sff E}-r)}(\wt{\op{D}}_r(\sff E)).\eeq

Similarly, if $\nu$ is in addition $\sff E^\vee$-suitable and we define $T,C,K$ as in \eqref{def:T},  since
$$\llbb{\sff E^\vee,r}=\nu_* \llb{\{\nu^*E_0\to (\nu^*E_1^*/C)^*\},r},$$
 for $\sff{L}:=\{\nu^*E_0\to q^*\nu^*B/Q^*_{\Gr}\oplus q^*\nu^*E_1/C^*\}$, we get
\beq{eq:LandE}\llbb{\sff E^\vee, r}=\nu_*\llb{\sff{L},r}\in A_{d-r(r-\rkh{\sff E}^0)}(\wt{\op{D}}_r(\sff E)).\eeq
 In the following theorem, we use the equalities \eqref{eq:HandE} and \eqref{eq:LandE} to find Thom-Porteous formulas over $\Gr(B^*,r)$.

\begin{theorem}\label{thm:T-Pformulas}
Using the notation above and omitting the pullbacks $q^*$, for any $(\sff E,\sff E^\vee)$-suitable blow up $\nu:Y\to X$, we have   
\begin{align}\label{eq:modB}
\iota_*\llb{\sff{E},r}=\nu_*\left(c_{r(b-\sff e+\rkh{\sff{E}}^2)}\big((\nu^*B-\tau^{\leq1}(\nu^*\sff{E}))\otimes Q_{\Gr}\big)\right),\\
\label{eq:2ndmodB}\iota_*\llbb{\sff{E}^\vee,r}=\nu_*\left(c_{r(b-\rkh{\sff E}^0)}\big((\nu^*B-\nu^*E_0+\nu^*E_1-C^*)\big)\otimes Q_{\Gr}^*)\right).
\end{align}
\end{theorem}
\begin{proof}
The Thom-Porteous formula from Theorem 14.4 in \cite{fu} gives
\begin{align*}\iota_*\llb{\sff{E}}&=\nu_*\left(\Delta^{(r)}_{b-\sff e+\rkh{\sff{E}}^2}c\big(\nu^*B/Q^*_{\Gr}-\tau^{\leq1}(\nu^*\sff{E})\big)\right)\\
&=\nu_*\left(c_{r(b-\sff e+\rkh{\sff{E}}^2)}\big((\nu^*B-\tau^{\leq1}(\nu^*\sff{E}))\otimes Q_{\Gr}\big)\right)\end{align*}
(cf.  \eqref{eq:TP-Gr(B)}). Similarly,
\begin{align*}
\iota_*\llbb{\sff{E}^\vee}&=\nu_*\left(\Delta^{(r)}_{b-\rkh{\sff{E}}^0}c\big(\nu^*B/Q^*_{\Gr}-\nu^*E_0+\nu^*E_1/C^*\big)\right)
    \\
    &=\nu_*\left(c_{r(b-\rkh{\sff E}^0)}\big((\nu^*B-\nu^*E_0+\nu^*E_1-C^*)\otimes Q_{\Gr}^*\big)\right).
\end{align*}\end{proof}

\begin{remark}\label{rmk:doubledeep}
    If $B^*$ also surjects onto $\sh^0(E^\vee)$ (e.g.  $B=E_0\oplus E_2^*$ has this property) we have the embedding
$$\wt{\op{D}}_r(\sff{E}^\vee)\subset\Gr(B^*,r),$$
    which allows us to work with $\llb{\sff E^\vee,r}$ and $\llbb{\sff E,r}$ over the same ambient space $\Gr(B^*,r)$.
\end{remark}

In the special case  $r=1$, let $q:\prj(B)\to X$ and  $\sff H$ be as defined in \eqref{def:complexdeepest}. We have the following result. 

\begin{prop} \label{prop:refined} Suppose that for some positive integer $\ell$, there is a map $v:\cO_{\prj(B)}^{\oplus \ell}(-1)[-2]\to \sff H$ in the derived category, such that $\sff V:=\Cone(v)$ is 1-special.
    Then,
$$\llb{\sff{V},1}=\llb{\sff{E},1}_{g-\rkh{\sff{E}}^2}.$$
\end{prop}
\begin{proof}
     By  \eqref{eq:HandE} $\llb{\sff H}=\llb{\sff E}$. Since $\sh^2(\sff H)\cong\sh^2(q^*\sff{E})$, 
     $$\llb{\sff{H}}_k=\llb{\sff{E}}_k\quad\text{ for }k\geq0.$$ An application of Proposition \ref{prop:deeprefined} to the exact triangle $$\sff H \to \sff V\to \cO^{\oplus \ell}_{\prj(B)}(-1)[-1],$$  and  the vanishing  
     $c_{i>0}\big(\cO^{\oplus \ell}_{\prj(B)}(-1)\otimes Q_{\op D_r(\sff H)}\big)=0$, resulting from $Q_{\op D_r(\sff H)}=\cO_{\wt{\op D}_r}(1)$,  proves the claim.
\end{proof}

\subsection{Wall-crossing and duality}\label{sec:wallform}
In this section, we take $r=1$. Let $$\sff{E}, \;\hat{\sff{E}}:=\sff{E}^\vee[-2]\in\Perf(X,0,2).$$  Let $B$ be a vector bundle over $X$ as in Remark \ref{rmk:doubledeep}  and  $q:\prj(B)\to X$ be the projection. Denote the inclusions
$$\iota:\wt {\op{D}}_1(\sff E)\hookrightarrow \prj(B),\qquad \jmath:\wt {\op{D}}_1(\hat{\sff{E}})\hookrightarrow\prj(B).$$
and let $p= q\circ\iota$, $\hat p=q\circ\jmath$.
Suppose that there exists an integer $\ell \ge 0$ and maps in the derived category
$$\cO_{\prj(B)}^{\oplus \ell}(-1)[-2]\to q^*\sff E,\qquad  \cO_{\prj(B)}^{\oplus \ell}(-1)[-2]\to q^*\hat{\sff{E}},$$
such that both cones are 1-special, i.e. they satisfy Condition \eqref{UDr} for $r=1$. 

Let $\nu:Y\to X$ be any $(\sff E,\hat{\sff{E}})$-blow up. As in Section \ref{sec:deepest}, denote by $\nu:\prj(\nu^*B)\to \prj(B)$ the induced blow up and  by $\tilde q:\prj(\nu^* B)\to Y$ the basechange of $q$. Fix a locally free resolution 
$$\sff{E}\simeq \{E_0\xr{\sigma_0}E_1\xr{\sigma_1} E_2\},$$
and define $C,K, T$ over $Y$ as in \eqref{def:T}.  In the following two subsections we will distinguish two cases, when $\ell =0$ and when $\ell >0$.

\subsubsection{Case $\ell=0$} In this case, we are simply asking that both $\sff{E}$ and $\hat{\sff{E}}=\sff E^\vee[-2]$ satisfy Condition \eqref{UDr} for $r=1$. By the virtue of Remark \ref{rem:UDr}, we phrase the main result of this subsection as follows.
\begin{theorem}    
\label{thm:gendual}
    Given $\sff E\in\Perf(X,0,2)$ with ${\op{D}}_1(\sff{E})\cap {\op{D}}_1(\hat{\sff{E}})=\emptyset$, 
    $$p_*\llb{\sff{E},1}+(-1)^{\sff e}\hat p_*\llb{\hat{\sff{E}},1}=c_{1-\sff e}(-\sff{E})\in  A_{d+\sff e-1}(X),$$ where $p, \hat p$ were introduced at the beginning of this section. 
\end{theorem}
 \begin{proof}
 Since $\nu$ maps ${\op{D}}_1(\nu^* \sff{E})$ into ${\op{D}}_1(\sff{E})$ and  $ {\op{D}}_1(\nu^*\hat{\sff{E}})$ into ${\op{D}}_1(\hat{\sff{E}})$, we conclude that $${\op{D}}_1(\nu^*\sff{E})\cap {\op{D}}_1(\nu^*\hat{\sff{E}})=\emptyset.$$ 
As sets, ${\op{D}}_1(\nu^*\sff{E})= \op{Supp}(\sh^2(\nu^*\hat{\sff{E}}))$ and ${\op{D}}_1(\nu^*\hat{\sff{E}})= \op{Supp}(\sh^2(\nu^*\sff{E}))$, so $$\sh^2(\nu^*\sff{E})|_{{\op{D}}_1(\nu^*\sff{E})}=0,\qquad \sh^2(\nu^*\hat{\sff{E}})|_{{\op{D}}_1(\nu^*\hat{\sff{E}})}=0.$$  Both vanishings are open  conditions and remain true on some open neighborhoods of $\op D_1(\nu^*\sff E)$ and $\op D_1(\nu^* \hat{\sff{E}})$ in $Y$. We have an exact sequence
\beq{eq:sesh2}0\to K\to \nu^*E_1\to \nu^*E_2\to \sh^2(\nu^*\sff{E})\to 0, \eeq
 Pulling back along the flat morphism $\tilde q:\prj(\nu^*B)\to Y$,  \eqref{eq:sesh2} remains a locally free resolution of $\tilde q^*\sh^2(\nu^* \hat{\sff{E}})$. This sheaf vanishes on an open subset of $\prj(\nu^*B)$ containing $\widetilde {\op{D}}_1(\nu^*\sff{E})$.  Therefore, for any   $\alpha\in  A_*(\wt{\op{D}}_1(\nu^*\sff{E}))$ 
 \beq{eq:vanci}c_{i> 0}(\nt q^*(\nu^* E_1-\nu^*E_2-K)(1))\cap\alpha=0.\eeq
 Proposition \ref{prop:aboutT} gives
$$\llb{\nu^* \sff{E}}=c_{\rkh{\sff{E}}^1}\big(\nt q^*T\otimes \cO_{\wt{\op D}_r(\nu^*\sff E)}(1)\big)\cap\llbb{\nu^*\hat{\sff{E}}}.$$
Together with \eqref{eq:vanci}, we see that 
$$\llb{\nu^*\sff{E}}=c_{\rkh{\sff{E}}^1}\big(\nt q^*(T-K+\nu^*E_1-\nu^*E_2)(1)\big)\cap\llbb{\nu^*\hat{\sff{E}}}.$$
We push forward by $\iota':\widetilde {\op{D}}_1(\nu^*\sff{E})\hookrightarrow\prj(\nu^*B)$ and use \eqref{eq:2ndmodB} and \eqref{eq:ktheoryT} 
\begin{align*}
     \iota_*'\llb{\nu^*\sff{E}}&=c_{\rkh{\sff{E}}^1}\big(\nt q^*(T-K+\nu^*E_1-\nu^*E_2)(1)\big)\\ &\qquad \cap c_{b-\sff e-c}\big(\nt q^*(\nu^*B+\nu^* E_1/C^*-\nu^*E_0)(1)\big)\\
     &=c_{\rkh{\sff{E}}^1}\big(\nt q^*(C^*-\nu^*E_2)(1)\big)\\ &\qquad \cap c_{b+(e_1-e_0-c)}\big(\nt q^*(\nu^* B+\nu^*E_1-\nu^*E_0-C^*)(1)\big).
\end{align*}
Since $e_1-e_0-c=0$, Lemma \ref{lem:pushpicn0} gives 
$$\tilde q_*\iota_*'\llb{\nu^*\sff{E}}=\sum_{j=0}^{\rkh{\sff{E}}^1}c_{\rkh{\sff{E}}^1-j}(C^*-\nu^*E_2)\cap c_{1+j}(\nu^*E_1-\nu^*E_0-C^*).$$
On the other hand, since $\op{rk}(\tau^{\leq1}(\hat{\sff{E}}))=\op{rk}(\sff{E})=-\rkh{\sff{E}}^1$, Theorem \ref{thm:T-Pformulas} gives 
$$\jmath'_*\llb{\nu^*\hat{\sff{E}}}=c_{b+\rkh{\sff E}^1}\big((\nt q^*B+\nt q^*C-\nt q^*\nu ^*E_2^*)(1)\big).$$
for the inclusion $\jmath':\wt {\op{D}}_1(\nu^*\hat{\sff E})\hookrightarrow \prj(\nu^* B^*)$. By the special case of Lemma \ref{lem:pushpicn0}
$$\tilde q_*\jmath'_*\llb{\nu^*\hat{\sff{E}}}=c_{\rkh{\sff{E}}^1+1}(C-\nu ^*E_2^*)=(-1)^{\rkh{\sff{E}}^1+1}c_{\rkh{\sff{E}}^1+1}(C^*-\nu^*E_2).$$
 The two expressions combine to give
$$\tilde q_*\iota_*'\llb{\nu^*\sff{E}}+(-1)^{\rkh{\sff{E}}^1+1}\tilde q_*\jmath'_*\llb{\nu^*\hat{\sff{E}}}=c_{\rkh{\sff{E}}^1+1}(-\nu^*\sff{E}).$$
We now apply $\nu_*$ to both sides to conclude the proof of the theorem. 
\end{proof}


\subsubsection{Case $\ell >0$} In this case, we define the complexes in $\Perf(\prj(B),0,2)$
$$\sff{H}:=\Cone(q^*\sff E\to q^*B/\cO(1)),\qquad \hat{\sff{H}}:=\Cone(q^*\hat{\sff{E}}\to q^*B/\cO(1)),\ $$
as in Section \ref{sec:deepest}, so that
$$\llb{\sff{E},1}=\llb{\sff{H},1},\qquad \llb{\hat{\sff{E}},1}=\llb{\hat{\sff{H}},1}.$$
By our assumption there are maps
$$v:\cO_{\prj(B)}^{\oplus \ell}(-1)[-2]\to q^*\sff H,\qquad \hat v:\cO_{\prj(B)}^{\oplus \ell}(-1)[-2]\to q^*\hat{\sff{H}},$$
such that both $\Cone(v)$ and $\Cone(\hat v)$ satisfy \eqref{UDr}. By Proposition \ref{prop:refined}, 
$$\llb{\Cone (v),1}=\llb{\sff{H},1}_{\ell-\rkh{\sff E}^2},\qquad \llb{\Cone(\hat v),1}=\llb{\hat{\sff{H}},1}_{\ell-\rkh{\hat{\sff E}}^2}.$$

\begin{theorem}\label{thm:dualdeep}
    We have 
    $$(-1)^{\sff e}p_*\llb{\sff E,1}_{\ell-\rkh{\sff E}^2}=\hat p_*\llb{\sff E^\vee,1}_{\ell-\rkh{\sff E}^0}\in A_{d+\sff e-\ell-1}(X),$$ where $p, \hat p$ were introduced at the beginning of this section.
\end{theorem}
\begin{proof}
We denote by $\eta:=\ell-\rkh{\sff E}^2$ and $\zeta:=\ell-\sff h^0_{\sff E}=\ell-\rkh{\hat{\sff E}}^2$. 
Recall that we defined $C,K$ in Section \ref{sec:moddual} which give the locally free resolutions
\[\begin{tikzcd}[cramped,row sep=tiny]
	0 & K & {\nu^*E_1} & {\nu^*E_2} & {\sh^2(\nu^*\sff E)} & 0, \\
	0 & C & {\nu^*E_1^*} & {\nu^*E_0^*} & {\sh^2(\nu^*\hat{\sff E})} & 0.
	\arrow[from=1-1, to=1-2]
	\arrow[from=1-2, to=1-3]
	\arrow[from=1-3, to=1-4]
	\arrow[from=1-4, to=1-5]
	\arrow[from=1-5, to=1-6]
	\arrow[from=2-1, to=2-2]
	\arrow[from=2-2, to=2-3]
	\arrow[from=2-3, to=2-4]
	\arrow[from=2-4, to=2-5]
	\arrow[from=2-5, to=2-6]
\end{tikzcd}\]

Pulling the complexes back to $\prj(\nu^*B)$ and using \eqref{eq:2ndmodB} in the 3rd equality below, we can write 
\begin{align*}
    \iota_*\llb{ \sff E}_\eta&=\iota_*\nu_*\left(s_\eta\big(\sh^2(\nu^*\sff H)(1)\big)\cap\llb{\nu^*\sff H}\right)\\
    &=\nu_*\left(s_\eta\big(\wt q^*\sh^2(\nu^*\sff E)(1)\big)\cap\iota_*'\llb{\nu^*\sff E}\right)\\
    &=\nu_*\left(s_\eta\big(\wt q^*\sh^2(\nu^*\sff E)(1)\big)c_{\rkh{\sff E}^1}\big(\wt q^*T(1)\big)\cap\iota'_*\llbb{\nu^*\sff E^\vee}\right)\\
    &=\nu_*\left(c_{\rkh{\sff E}^1+\eta}\big(\wt q^*(T-\sh^2(\nu^*\sff E))(1)\big)c_{b-\rkh{\sff E}^0}\big(\wt q^*(B-\sh^2(\nu^*\hat{\sff{E}})^\vee)(1)\big)\right)\\
    &=\nu_*\left(c_{\ell+\rkh{\sff E}^1-\rkh{\hat{\sff{E}}}^0}\big(\wt q^*(C^*-E_2)(1)\big)c_{b-\rkh{\hat{\sff{E}}}^2}(\wt q^*(B-\sh^2(\nu^*\hat{\sff{E}})^\vee)(1)\big)\right)\\
    &=\nu_*\left(c_{\ell+\sff e-\rkh{\hat{\sff{E}}}^2}\big(\wt q^*(\tau^{\leq1}(\sff \nu^*\hat{\sff{E}}))^\vee(1)\big)c_{b-\rkh{\hat{\sff{E}}}^2}\big(\wt q^*(B-\sh^2(\nu^*\hat{\sff{E}})^\vee)(1)\big)\right),
\end{align*}
where $s_\eta\big(\wt q^*\sh^2(\nu^*\sff E)(1)\big)c_{\rkh{\sff E}^1}\big(\wt q^*T(1)\big)=c_{{\rkh{\sff E}^1}+\eta}\big(\wt q^*(T-\sh^2(\nu^*\sff E))(1)\big)$ as in the proof of Proposition \ref{prop:deeprefined}, since by the assumption there is an open $\wt U\subset\prj(\nu^*B)$ containing $\wt{\op{D}}_r(\sff E)$, such that $\cO_U^{\oplus \ell}-\wt q^*\sh^2(\nu^*\sff{ E})(1)|_{\wt U}$ in $K$-theory is represented by a locally free sheaf of rank $\eta$.

We get a similar formula for $\hat{\sff{E}}$:
\begin{align*}
    \jmath_*\llb{\hat{\sff{E}}}_\zeta&=\jmath_*\nu_*\left(s_\zeta\big(\sh^2(\nu^*\hat{\sff{H}})(1)\big)\cap\llb{\nu^*\hat{\sff{H}}}\right)\\
    &=\nu_*\left(s_\zeta \big(\wt q^*\sh^2(\nu^*\hat{\sff{E}})(1)\big)\cap\jmath'_*\llb{\nu^*\hat{\sff{E}}}\right)\\
     &=\nu_*\left(c_{\ell-{\rkh{\hat{\sff E}}^2}}\big(-\wt q^*\sh^2(\nu^*\hat{\sff{E}})(1)\big)c_{b-\sff e+\rkh{\hat{\sff{E}}}^2}\big(\wt q^*(B-\tau^{\leq1}(\nu^*\hat{\sff E}))(1)\big)\right).
\end{align*}
Both of these classes are the pushforward under $\nu$ of some classes over $\prj(\nu^*B)$. We then use  Lemma \ref{lem:pushpic} and the following fiber square to push the classes further down to $X$
\[\begin{tikzcd}
	{\prj(\nu^*B)} & {\prj(B)} \\
	Y & {X.}
	\arrow[" \nu", from=1-1, to=1-2]
	\arrow["{\tilde q}"', from=1-1, to=2-1]
	\arrow["q", from=1-2, to=2-2]
	\arrow["\nu"', from=2-1, to=2-2]
\end{tikzcd}\]
\begin{align*}
    p_*\llb{\sff E}_\eta=&
    \nu_*\sum_{j=0}^{\ell}(-1)^j{\ell-1\choose j}c_{1+j-\rkh{\hat{\sff{E}}}^2}\big(-\sh^2(\nu^*\hat{\sff{E}})^\vee\big)c_{\ell-\sff e+\rkh{\hat{\sff{E}}}^2-j}\big(\tau^{\leq1}(\nu^*\hat{\sff{E}})^\vee\big)\\
    =&\nu_*\sum_{j=0}^{\ell}(-1)^{j+1+\ell-\sff e}{\ell-1\choose j}c_{1+j-\rkh{\hat{\sff{E}}}^2}\big(-\sh^2(\nu^*\hat{\sff{E}})\big)c_{\ell-\sff e+\rkh{\hat{\sff{E}}}^2-j}\big(\tau^{\leq1}(\nu^*\hat{\sff{E}})\big),  
\end{align*}
      \begin{align*}
    \hat p_*\llb{\hat{\sff{E}}}_\zeta=&\nu_*\sum_{j=0}^{\ell}(-1)^j{\ell-1\choose j}c_{1+j-\sff e+\rkh{ \hat{\sff{E}}}^2}\big(\tau^{\leq1}(\nu^*\hat{\sff{E}})\big)c_{\ell-\rkh{\hat{\sff{E}}}^2-j}\big(-\sh^2(\nu^*\hat{\sff{E}})\big)\\
    =&\nu_*\sum_{k=0}^{\ell}(-1)^{\ell-1-k}{\ell-1\choose k}c_{\ell-k-\sff e+\rkh{\hat{\sff{E}}}^2}\big(\tau^{\leq1}(\nu^*\hat{\sff{E}})\big)c_{1+k-\rkh{\hat{\sff{E}}}^2}\big(-\sh^2(\nu^*\hat{\sff{E}})\big).  
\end{align*}
Comparing these two expressions, we get 
$$(-1)^{\sff e+\ell+1}p_*\llb{\sff E}_\eta=(-1)^{\ell-1}\hat p_*\llb{\hat{\sff{E}}}_\zeta$$ from which the claim follows.

\end{proof}

\subsection{Split 3-term complexes}\label{sec:splitting}
In this section, we assume that the complex $\sff{E}\in\Perf(X,0,2)$ splits as follows   
$$\sff{E}= \sff{A}\oplus \sff{G}[-1]\simeq\{A_0\xr{(\sigma_0,0)}A_1\oplus G_0\xr{(0,\sigma_1)} G_1\}$$ for some $\sff A, \sff G\in \Perf(X,0,1)$, and for which we have also fixed a locally free resolution accordingly.  
We further assume that $\sh^2(\sff E)=\coker(\s_1)$ is a torsion sheaf,  $\cC:=\coker(\s_1^*)$ has rank $c>0$ and the Grassmannian $\Gr(\cC,c)$ is irreducible. (Proposition \ref{prop:irred} provides a sufficient condition for the irreducibility).

Choose a vector bundle $B$ as in Section \ref{sec:deepest}, and define the complex  $$\sff{H}:=\{q^*A_0\to q^*(A_1\oplus G_0)\oplus q^*B/Q^*_\Gr\to q^*G_1\},$$  where $q:\Gr(B^*,r)\to X$. Since $q$ is a smooth morphism, the basechanged Grassmannian  $\Gr(q^*\cC,c)$ is also irreducible. 
By Corollary \ref{cor:irred}, $$\nu\colon Y:=\Gr(\cC,c)\to X, \qquad \nu\colon \Gr(\nu^*B^*,r)\cong \Gr(q^*\cC,c)\to X$$ are respectively $\sff E$- and $q^*\sff E$-suitable blow ups. They fit into the following commutative diagram
\[\begin{tikzcd}
	& {\Gr(q^*G_0^*,c)} & {\Gr(G_0^*,c)} \\
	{ \wt{\op{D}}_r(\nu^* \sff{E})} & {\Gr(\nu^*B^*,r)} & Y \\
	{ \wt{\op{D}}_r(\sff{E})} & {\Gr(B^*,r)} & X.
	\arrow[, from=1-2, to=1-3]
	\arrow["{\tilde p}"{pos=0.25}, shift left, curve={height=-6pt}, from=1-2, to=3-2]
	\arrow["p", shift left, curve={height=-8pt}, from=1-3, to=3-3]
	\arrow["{{\tilde \iota}}", hook, from=2-1, to=2-2]
	\arrow["\nu"', from=2-1, to=3-1]
	\arrow["j", hook, from=2-2, to=1-2]
	\arrow["{{\tilde q}}", from=2-2, to=2-3]
	\arrow["\nu"', from=2-2, to=3-2]
	\arrow["i", hook, from=2-3, to=1-3]
	\arrow["\nu"', from=2-3, to=3-3]
	\arrow["\iota", hook, from=3-1, to=3-2]
	\arrow["q", from=3-2, to=3-3]
\end{tikzcd}\]

For $K:=\ker(q^*\sigma_1)=(\nu^*q^*\cC)^*$ and $\sff e:=\rank{\tau^{\leq1}(\sff{E})}=\rank(\sff{E})$ (as $\sh^2(\sff{E})$ is a torsion sheaf), the Thom-Porteous formula \eqref{eq:modB} gives
$$\tilde\iota_*\llb{\nu^*\sff{E}}=c_{r(b-\sff e)}\big((\nu^*q^*B-\nu^*q^*\sff A+\nt q^* Q^*_{\Gr(G_0^*)}|_Y)\otimes Q_{\Gr(\nu^*B^*)}\big).$$
 Using Corollary\ref{cor:irred} and the identification $K=Q_{\Gr(q^*G_0^*,c)}^*|_{\Gr(\nu^*B^*,r)}$, after omitting the obvious pullback symbols, we get
\begin{align}&\label{eq:ps1}j_*\tilde\iota_*\llb{\nu^*\sff{E}}\\ \nonumber &=
    c_{r(b-{\sff e})}\big((B-\sff A+Q^*_{\Gr(G_0^*)})\otimes Q_{\Gr(\nu^*B^*)})\big) \Delta_c^{(g_1)}c\big(Q_{\Gr(G_0^*)}-G_1^*\big),
\end{align}
so that
$$\iota_*\llb{\sff{E}}=\tilde p_*\left(c_{r(b-{\sff e})}\big((B-\sff A+Q^*_{\Gr(G_0^*)})\otimes Q_{\Gr(\nu^*B^*)}\big)\Delta_c^{(g_1)}c\big(Q_{\Gr(G_0^*)}-G_1^*\big)\right).$$

In the case of $r=1$, a more explicit formula can be obtained for the first classes, which can also be extended to include all the sequence of classes defined in \eqref{def:refinedE}. We start by identifying $\Gr(G^*_0,c)\cong\Gr(G_0^*(-1),c)$, where $\cO(1)=\cO_{\prj(B)}(1)$. Under this isomorphism, $Q_{\Gr(G_0^*)}(-1)$ corresponds to $Q_{\Gr}:=Q_{\Gr(G_0^*(-1))}$. We then have the identity 
$$\Delta_c^{(g_1)}c\big({Q}_{\Gr(G_0^*)}-G_1^*\big)=c_{rg_1}\big({Q}_{\Gr(G_0^*)}\otimes G_1\big)=\Delta_c^{(g_1)}c\big({Q}_{\Gr(G_0^*)}(-1)-G_1^*(-1)\big),$$
and since $\sh^2(\nu^*q^*\sff{E})=q^*K-\nu^*q^*\sff{G}=Q^*_{\Gr}(-1)-\nu^*q^*\sff{G}$ in $K$-theory over $\Gr(q^*\cC,c)$, 
\begin{align*}\label{ps3}
     \iota_*&\llb{\sff{E},1}_m=
    \tilde p_*\left(c_m\big (\sff{G}(1)-Q_{\Gr}^*\big)c_{b-\sff e}\big((B-\sff A)(1)+Q^*_{\Gr}\big)\Delta_c^{(g_1)}c\big(\wt{Q}\big)\right),
\end{align*}
where  $\widetilde {Q}:= Q_{\Gr}-G_1^*(-1)$.

\begin{theorem}\label{thm:refinedforsplit}
    Under the assumptions of this subsection, for $$\sff E=A\oplus \sff G[-1]\in \Perf(X,0,2),$$ and omitting the pullback symbols,  for any integer $m\ge 0$, 
    $$\iota_*\llb{\sff{E},1}_m=c_{b-{\sff e}-c}\big(B(1)-\sff{A}(1)\big )c_{c+m}\big(\sff{G}(1)\big)\in A_{\sff n-1+\sff e-m}(\prj(B)).$$
\end{theorem}
\begin{proof}
Denote by $U_{\Gr}$  the universal subbundle of $\Gr(q^*G_0^*(-1),c)$, so that $$Q_{\Gr}\cong q^*G_0^*(-1)/U_{\Gr},$$ and let $\widetilde U:=U_{\Gr}-G_1^*(-1)$. For a sequence of nonnegative integers $I=(i_1,\ldots, i_n)$, denote the determinant of Segre classes by $$s_I(-):=|s_{i_p-p+q}(-)|_{1\leq p,q\leq n}.$$  We can now express $\iota_*\llb{\sff E,1}_m$ as the image under $\tilde p_*$ of 
    \begin{align*}     
&c_m\big (U^*_{\Gr}-G_1(1)\big)c_{b-{\sff e}}\big(G_0(1)-U_{\Gr}^*- \sff{A}(1)+B(1)\big)(-1)^{cg_1}s_{(g_1)^{c}}\big(\widetilde {Q}\big)\\
    =&\sum_{i=0}^{b-{\sff e}}(-1)^{m+i+cg_1}c_m\big(\wt U\big)c_{b-{\sff e}-i}\big((\sff{G}- \sff A+B)(1)\big) s_i\big(U_{\Gr}-G_1^*(-1)\big)s_{(g_1)^{c}}\big(\widetilde {Q}\big)\\
     =&\sum_{i=0}^{b-{\sff e}}(-1)^{m+i+cg_1}c_{b-{\sff e}-i}\big((\sff{G}- \sff A+B)(1)\big)(-1)^ms_{(1)^m}\big(\widetilde U\big)s_i\big(\widetilde U\big)s_{(g_1)^{c}}\big(\widetilde {Q}\big)\\
 =&\sum_{i=0}^{b-{\sff e}}(-1)^{i+cg_1}c_{b-{\sff e}-i}\big((\sff{G}- \sff A+B)(1)\big)\left( s_{i,(1)^m}(\widetilde U)+s_{i+1,(1)^{m-1}}(\widetilde U)\right)s_{(g_1)^{c}}(\widetilde {Q}).
     \end{align*}
    For the last equality we used  \cite[Lemma 14.5.2]{fu}. Here, $s_{i+1,(1)^{m-1}}(\widetilde U)=0$ if $m=0$. By  \cite[Proposition  2.2]{Pragacz1988}, for a sequence of integers $I$ and $J=(g_1)^c$,
    $$\nt p_*\left(s_I\big(U-G_1^*(1)\big)s_{J}(Q-G_1^*(1))\right)=(-1)^{cg_1}s_{(0)^{c},I}\big(G_0^*(-1)-G_1^*(-1)\big).$$
    After this simplification, the projection formula allows us to get rid of the pushforward $\tilde p_*$. For $m\geq 1$, $\iota_*\llb{\sff{E},1}_m=$
    \begin{align*}
     =&\sum_{i=0}^{b-{\sff e}}(-1)^{i}c_{b-{\sff e}-i}\big((\sff{G}- \sff{A}+B)(1)\big)\Big( s_{(0)^{c},i,(1)^m}\big(\sff{G}^\vee(-1)\big)\\
     &+s_{(0)^{c},i+1,(1)^{m-1}}\big(\sff{G}^\vee(-1)\big)\Big)\\
    =&\sum_{i=c}^{b-{\sff e}}(-1)^ic_{b-{\sff e}-i}\big((\sff{G}- \sff{A}+B)(1)\big)(-1)^{c}\Big(s_{i-c,(1)^{c+m}}\big(\sff{G}^\vee(-1)\big)\\
    &+s_{i+1-c,(1)^{c+m-1}}\big(\sff{G}^\vee(-1)\big)\Big)\\
    =&\sum_{i=c}^{b-{\sff e}}(-1)^{m+i}c_{b-{\sff e}-i}\big((\sff{G}- \sff{A}+B)(1)\big) s_{i-c}\big(\sff{G}^\vee(-1)\big)c_{c+m}\big(\sff{G}^*(-1)\big)\\
    =&c_{b-{\sff e}-c}\big(B(1)- \sff{A}(1)\big)c_{c+m}\big(\sff{G}(1)\big).
    \end{align*}
    For $m=0$, we get instead  
    \begin{equation}\label{eqn:casem0}  
    \iota_*\llb{\sff{E},1}=\sum_{i=0}^{b-{\sff e}}(-1)^{c}c_{b-{\sff e}-i}\big((\sff{G}- \sff{A}+B)(1)\big) s_{i-c,(1)^{c}}\big(\sff{G}(1)\big).
    \end{equation}
    But after expanding the determinant, one can show that  
\begin{align*}
        &s_{i-c,(1)^{c}}\big(\sff{G}(1)\big)=(-1)^c \sum_{j=0}^{c}s_{i-c+j}\big(\sff{G}(1)\big)c_{c-j}\big(\sff{G}(1)\big).
    \end{align*}
    Note that for $i\leq c$, the right hand is $(-1)^cc_i\big(\sff{G}(1)-\sff{G}(1)\big)=0$. Substituting back into \eqref{eqn:casem0}, we get
\begin{align*}\iota_*\llb{\sff{E},1}
    &=\sum_{i=0}^{b-{\sff e}}c_{b-{\sff e}-i}\big((\sff{G}- \sff{A}+B)(1)\big) \sum_{j=0}^{c}s_{i-c+j}\big(\sff{G}(1)\big)c_{c-j}\big(\sff{G}(1)\big)\\
    &=\sum_{j=0}^{c}c_{c-j}\big(\sff{G}(1)\big) \sum_{i=c-j}^{b-{\sff e}}c_{b-{\sff e}-i}\big((\sff{G}- \sff{A}+B)(1)\big) s_{i-c+j}\big(\sff{G}(1)\big)\\
    &=\sum_{j=0}^{c}c_{c-j}\big(\sff{G}(1)\big)c_{b-{\sff e}-c+j}(B(1)- \sff{A}(1)\big).
\end{align*}
Now, by the same reasoning as in Section \eqref{sec:setupB}, $B-\sff{A}$ is represented by a vector bundle of rank $b-{\sff e}-c$ in $K$-theory, so for $j>0$
$$c_{b-{\sff e}-c+j}\big(B(1)- \sff{A}(1)
\big)=0,$$
and hence the last summation above involves only one possibly non-vanishing term $c_{c}\big(\sff{G}(1)\big)c_{b-{\sff e}-c}\big(B(1)-\sff{A}(1)\big)$. This completes the proof. 
\end{proof}

\bigskip

\section{Nested Hilbert schemes of surfaces}\label{sec:NHS}

Let $S$ be a nonsingular complex projective surface and $S^{[n_1,n_2]}_\b$ be the Nested Hilbert scheme as introduced in Section \ref{sec:introNested}. It carries a perfect obstruction theory leading to a virtual cycle denoted by $$[S^{[n_1,n_2]}_\b]^{\vir}\in A_{n_1+n_2+\vd_\b}(S^{[n_1,n_2]}_\b),$$ where $\vd_\b=\b(\b-K_S)/2$ (cf. \cite{gholampour2020nested, G1, G2}). This perfect obstruction theory is the fixed part of the Donaldson-Thomas obstruction theory on the total space of a line bundle on $S$ (cf. \cite{gholampour2020localized}), and is also the monopole part of the Vafa-Witten obstruction theory of $S$ (cf. \cite{tanaka2017vafa, tanaka2018vafa}). 
We refer to this as the DT virtual cycle.

When $H^2(L)=0$ for every effective  $L\in \Pic_\beta(S)$ and $H^2(\cO_S)\neq 0$, the virtual cycle above vanishes. In this case, $S^{[n_1,n_2]}_\b$ carries a reduced perfect obstruction theory leading to a reduced cycle denoted by 
$$[S^{[n_1,n_2]}_\b]^{\op{red}}\in A_{n_1+n_2+\vd_\b+p_g}(S^{[n_1,n_2]}_\b),$$ where $p_g=h^2(\cO_S)$ is the geometric genus of $S$ (cf. \cite{gholampour2020nested, G1, G2}). This class has been related to curve counting on $S$ (cf. \cite{Kool_2014,Kool_2014_2}). 

In the rest of the paper, we recover these two classes by the construction of Section \ref{Sec:vir3-term}, and study some of their properties. To simplify the notation we will sometimes omit the obvious pullback symbols in what follows.

\subsection{Recovering the reduced cycle}\label{sec:setupHilb}
    Let 
    $$X:=S^{[n_1]}\times S^{[n_2]}\times \Pic_\beta(S).$$
    Letting $\pi:S\times X\to X$ to be the projection, we define     \beq{complexE}\sff{E}:=R\sHom_\pi(\cI_1,\cI_2(\cP_\beta))\in\Perf(X,0,2).\eeq Here, $\cI_1, \cI_2, \cP_\b$ are the universal objects defined in Section \ref{sec:introNested}. The construction of Section \ref{Sec:vir3-term} applies and enables us to define 
\begin{equation}\label{def:firsthilb}
\llhb{S_\beta^{[n_1,n_2]}}:=\llb{\sff{E},1}\in A_{n_1+n_2+\vd_\beta+p_g-\rkh{\sff E}^2}(S_\beta^{[n_1,n_2]})
\end{equation}
by the virtue of the canonical isomorphism 
$S_\beta^{[n_1,n_2]}\cong\widetilde {\op{D}}_1(\sff{E})$ established in \cite{G2}. If $\sff h^2_{\sff E}=0$, this class has the same dimension as that of the reduced virtual cycle. In fact, according to the following proposition, the two classes coincide when the reduced cycle is defined. 

\begin{prop} \label{proph2l=0} If $H^2(L)=0$ for all effective $L \in \Pic_\b(S)$, we have
    $$\llhb{S_\beta^{[n_1,n_2]}}=[S_\beta^{[n_1,n_2]}]^{\op{red}}.$$  
    
\end{prop}
\begin{proof}
Under the assumption of the proposition, \cite{G2} proves that
$$[S_\beta^{[n_1,n_2]}]^{\op{red}}:=[\wt{\op{D}}_1(\sff{E})]^{\vir}\in A_{n_1+n_2+\vd_\beta+p_g}(S_\beta^{[n_1,n_2]}).$$
But the condition in the proposition is equivalent to Condition \eqref{UDr} for $r=1$, where $U$ can be taken to be  an open neighborhood of the image of the Abel-Jacobi map $\op{AJ}(S^{[n_1]}\times S^{[n_1]}\times S_\beta)\subset X$. Therefore, the claim is an immediate consequence of Proposition \ref{prop:red1st}.

\end{proof}

\subsection{Ample divisor} \label{sec:embvia}

Fix a sufficiently ample divisor $A\subset S$ so that for all $L\in\Pic_\beta(S)$, $H^{\geq1}(L(A))=0$. Define the locally free sheaf 
\begin{equation}\label{def:B}B:=R\pi_*(\cP_\beta(A))=\pi_*(\cP_\beta(A))
\end{equation}
with the map $s_A:R\pi_*\cP_\beta\to B$ induced by the canonical section $\cO_S\to \cO_S(A)$. As in (4.12) of \cite{G2}, we  have the surjection $B^*\twoheadrightarrow \sh^0(\sff E^\vee)$ defined using Serre duality by the composition
\begin{align}
    B^*\cong R^2\pi_*(\cP_\beta^*(K_S-A))
    &\onto{s^*_A}R^2\pi_*(\cP^*_\beta\otimes K_S)\cong R^2\pi_*(\cI_1\otimes\cP^*_\beta\otimes K_S)\nonumber\\
&\onto{}\sExt^2_\pi(\cI_2,\cI_1\otimes\cP^*_\beta\otimes K_S)\cong \sh^0(\sff E^\vee).\label{fig:surjB}
\end{align} 

Following Section \ref{sec:setupB}, we get an embedding $\iota: S_\beta^{[n_1,n_2]}\hookrightarrow \prj(B)$  with $\cO_{\prj(B)}(-1)$ restricting to the dual of the quotient line bundle of $S_\beta^{[n_1,n_2]}\cong \wt {\op{D}}_1(\sff{E})$, and in the diagram
\begin{equation}\label{fig:BtoE0}\begin{tikzcd}
	&& {B^*} \\
	{E_{1}^*} & {E_0^*} & {\sh^0(\sff{E}^\vee)} & 0
	\arrow["\psi"', dashed, from=1-3, to=2-2]
	\arrow[two heads, from=1-3, to=2-3]
	\arrow[from=2-1, to=2-2]
	\arrow[from=2-2, to=2-3]
	\arrow[from=2-3, to=2-4]
\end{tikzcd}\end{equation}
we can find a lifting $\psi$ after replacing $\prj(B)$ by an affine bundle over it, which is an affine variety.

  We also have the canonical embeddings $i$ and $\jmath$ as below resulted from the surjections in \eqref{fig:surjB}.
    \begin{equation}\label{fig:contentionshilb}
    \begin{tikzcd}[cramped]
	{S_\beta^{[n_1,n_2]}} & {S^{[n_1]}\times S^{[n_2]}\times S_\beta} & {\prj(B)}
	\arrow["i", hook, from=1-1, to=1-2]
	\arrow["\iota"', hook, curve={height=12pt}, from=1-1, to=1-3]
	\arrow["\jmath", hook, from=1-2, to=1-3]
\end{tikzcd}
    \end{equation} Pulling back to $\prj(B)$, the dual of $\psi$ in \eqref{fig:BtoE0} gives the map $u:\sff{E}\to {B}/{\cO(-1)}$ as in Section \ref{sec:deepest}. For $\sff{H}:=\Cone(u)$, we have the identifications $$S_\beta^{[n_1,n_2]}\cong\widetilde {\op{D}}_1(\sff{H})\cong{\op{D}}_1(\sff{H})$$ from which we may conclude as in Section \ref{sec:deepest} that 
\begin{equation}\label{eq:hilbdeepest1}
    \llhb{S_\beta^{[n_1,n_2]}}=\llb{\sff{H},1}.
\end{equation}

\subsection{Recovering DT virtual cycle I} \label{sec:recvircyc} Let $\widetilde X:=S^{[n_1]}\times S^{[n_2]}\times S_\beta,$
and $\pi: \widetilde X\times S\to \widetilde X$ be the projection. Denote by $\cD_\b$ the universal divisor of $S_\b$.  By \cite[Section 4.2]{G2}, we have the following commutative diagram of complexes on $\widetilde X$
\begin{equation}\label{fig:h2diag}
    \begin{tikzcd}[column sep=tiny]
	& {R^2\pi_*\cO[-2]} &  \\
	{R\pi_*\cI_2(\cD_\beta)} && {R\pi_*\cO(\cD_\beta)}   \\
	{R\sHom_\pi(\cI_1,\cI_2(\cD_\beta))} && {R\sHom_\pi(\cI_1,\cO(\cD_\beta)),}
	\arrow[dotted,from=1-2, to=2-1]
	\arrow["s_
    \beta",from=1-2, to=2-3]
	\arrow[from=2-1, to=2-3]
	\arrow[from=2-1, to=3-1]
	\arrow[from=2-3, to=3-3]
	\arrow["f",from=3-1, to=3-3]
\end{tikzcd}
\end{equation} 
where  the left diagonal arrow is obtained after pulling back to an affine bundle, which is an affine variety. On the second cohomology, we get the surjections over a closed point $(I_1, I_2, D)\in \wt X$ 
\begin{equation}\label{fig:h2surj}
    \begin{tikzcd}
	& {H^2(\cO_S)} \\
	{H^2(I_2(D))} && {H^2(\cO(D))} \\
	{\text{Ext}^2(I_1,I_2(D))} && {\text{Ext}^2(I_1,\cO(D)).}
	\arrow[two heads, from=1-2, to=2-1]
	\arrow[two heads, from=1-2, to=2-3]
	\arrow[no head, from=2-1, to=2-3]
	\arrow[shift left, no head, from=2-1, to=2-3]
	\arrow[two heads, from=2-1, to=3-1]
	\arrow[two heads, from=2-3, to=3-3]
	\arrow[two heads, from=3-1, to=3-3]
\end{tikzcd}
\end{equation}
Denote by $\rho(\sff A)=\Cone(R^2\pi_*\cO[-2]\to\sff A)$ for a complex $\sff A$ in the diagram \eqref{fig:h2diag}. By \eqref{fig:h2surj}, $$\sh^{i\geq2}(\rho(\sff A))=0.$$
For the embedding $\jmath: \wt X\hookrightarrow \prj(B)$ in \eqref{fig:contentionshilb}, $\cO_{\wt X}(D_\b)\cong \cP_\b(1)|_{\wt X}$. One can find an open $U\subset\prj(B)$ containing $\wt X$ over which there is a diagram as in \ref{fig:h2diag} with $\cO(\cD_\b)$ replaced by $\cP_\b(1)$, such that the fiberwise surjectivity on the second cohomology as in \eqref{fig:h2surj} is satisfied. 

If  $u:\sff E\to B/\cO(-1)$ is the map over $\prj(B)$ as in Section \ref{sec:embvia}, after applying $\rho(-)$ as above, we get $$w:\rho(\sff E(1))(-1)\to B/\cO(-1),$$ which also satisfies the conditions from Section \ref{sec:deepest}. Let $\sff V:=\Cone(w)$.

\begin{prop}\label{prop:hilbvir}
    We have $[S_\beta^{[n_1,n_2]}]^{\vir}=\llb{\sff V,1}.$
\end{prop}
\begin{proof}
   By the discussion above, $\sff V$  satisfies Condition \eqref{UDr}, so from Proposition \ref{prop:red1st},  $\llb{\sff V,1}=[{\op D}_r(\sff V)]^{\vir}$. The claim follows from \cite[Theorem 4.34]{G1}.
\end{proof}


\subsection{Recovering DT virtual cycle II}
Recall the complexes $\sff H=\Cone(u)$ from Section \ref{sec:embvia} and $\sff V=\Cone(w)$ from Section \ref{sec:recvircyc} defined over some open $U\subset \prj(B)$. We have shown that $$\llhb{S_\beta^{[n_1,n_2]}}=\llb{\sff{H},1},\qquad [S_\beta^{[n_1,n_2]}]^{\vir}=\llb{\sff{V},1}.$$
Consider the following commutative diagram with exact columns and rows in which the map $\alpha$ is induced from the rest of the diagram. 
\begin{equation}\label{vircred}
\begin{tikzcd}[column sep=small]
	{R^2\pi_*\cO(-1)[-2]} & {R\sHom_\pi(\cI_1,\cI_2(\cP_\beta))} & {\rho \big(R\sHom_\pi(\cI_1,\cI_2(\cP_\beta)(1))\big)(-1)} \\
	0 & {B/\cO(-1)} & {B/\cO(-1)} \\
	{R^2\pi_*\cO(-1)[-1]} & {\sff{H}[1]} & {\sff{V}[1]}
	\arrow[from=1-1, to=1-2]
	\arrow[from=1-1, to=2-1]
	\arrow[from=1-2, to=1-3]
	\arrow["u",from=1-2, to=2-2]
	\arrow["w",from=1-3, to=2-3]
	\arrow[from=2-1, to=2-2]
	\arrow[from=2-1, to=3-1]
	\arrow[from=2-2, to=2-3]
	\arrow[from=2-2, to=3-2]
	\arrow[from=2-3, to=3-3]
	\arrow[from=3-1, to=3-2]
	\arrow["\alpha", dashed, from=3-2, to=3-3]
\end{tikzcd}
\end{equation}
From this, we see that $\Cone(\sff{H}\to \sff{V})$ is represented by the vector bundle $\cO(-1)^{\oplus p_g}$ sitting in degree $1$ and $\sh^2(\sff{V})=0$. An application of Proposition \ref{prop:refined} now gives the following result.
\begin{theorem}\label{redtovir}
The virtual cycle of $S_\beta^{[n_1,n_2]}$ can be obtained as
$$[S_\beta^{[n_1,n_2]}]^{\vir}=\llb{\sff E,1}_{p_g-\rkh{\sff E}^2}.$$ \qed
\end{theorem}

This theorem motivates the following definition (cf. Definition \ref{def:refinedE}).

\begin{definition} \label{def:refinedNested}
\emph{Define} the $k$-th refined class \emph{on the nested Hilbert scheme as}
\begin{equation}\label{def:refpoin}
    \llhb{S_\beta^{[n_1,n_2]}}_k:=\llb{\sff{E},1}_k\in A_{n_1+n_2+\vd_\beta+p_g-\rkh{\sff E}^2-k}(S_{\beta}^{[n_1,n_2]})
\end{equation}
\emph{for each} $0\leq k \leq p_g-\rkh{\sff E}^2$. 
\end{definition}
In general, this gives us classes in decreasing degrees with the two extremes
$$\llhb{S_\beta^{[n_1,n_2]}}_0=\llhb{S_\beta^{[n_1,n_2]}},\qquad \llhb{S_\beta^{[n_1,n_2]}}_{p_g-\rkh{\sff{E}}^2}=[S_\beta^{[n_1,n_2]}]^{\vir}.$$
In the situation of Proposition \ref{proph2l=0}, $\sh^2(\sff E)|_{S_\beta^{[n_1,n_2]}}=0$, and hence
$$\llhb{S_\beta^{[n_1,n_2]}}_{k>0}=0.$$

\begin{remark} \label{rmk:diffvirdim} Let $0=\tau_0, \tau _1, \dots, \tau_s\in H^2(S,\bb Z)_{\op{tor}}$ be all the distinct torsion classes. Denote by $\sff E_{\tau_i}$ the analog of the complex \eqref{complexE} for the cohomology class $\b+\tau_i$. Then, for any $\min(\sff{h}^2_{\sff E_{\tau_0}},\dots,  \sff{h}^2_{\sff E_{\tau_s}})\le m \le p_g$, the refined cycles  
$$\llhb{S_{\beta+\tau_i}^{[n_1,n_2]}}_{m-\sff{h}^2_{\sff E_{\tau_i}}}\qquad i=0,\dots, s$$ have the same virtual dimension $n_1+n_2+\vd_\b+p_g-m$. 
    \end{remark}

The following proposition gives a sufficient condition for deformation invariance of the refined classes defined above in the sense of Section \ref{sec:defInvar}. It should be compared to the deformation invariance property of the reduced cycles  (cf. \cite[Remark 3.1]{Kool_2014} and Remark \ref{rmk:prop4.5}).

\begin{prop} \label{prop:defInvmink}
Let $\cc S\to C$ be a smooth family of projective surfaces over a nonsingular curve $C$, $\b$ be a class in the cohomology of the fibers, and $k$, $l$ be fixed nonnegative integers. Let $\Pic_\beta(\cc S/C)$ be the relative Picard variety and $\op{BN}_t\subset \Pic_{\b}(\cc S/C)$ be the Brill-Noether locus, where $h^2(L)\ge t$. Suppose that for each fiber $\cc S_p$, \begin{enumerate}[(i)]
\item $\b$ is of type $(1,1)$, \item  $l=\min \{h^2(L):L\in\Pic_\beta(\cc S_p)\}$,
\item for any $t>l$, each associated point of the normal cone $\op{C}_{\op{BN}_t/\Pic}$, whose image in $\Pic_\b(\cc S/C)$ specializes to a point of $\op{AJ}(\cc S_{p,\b})\subset \Pic_\b(\cc S/C)$ is mapped to the generic point of $C$.
\end{enumerate}    
Then, the classes 
    $\llhb{\cc S_{p,\beta}^{[0,n]}}_k$ obey deformation invariance on $C$. 

\end{prop}
\begin{proof}
     For $(\cc S/C)^{[n]}$ the relative Hilbert scheme of points, let $$\cc X:=(\cc S/C)^{[n]}\times_C\Pic_\beta(\cc S/C)$$ with the universal ideal $\scr I$ and a  Poincar\'e line bundle $\scr P_\b$ corresponding to the two factors. For the projection $\pi\colon \cc S\times_C \cc X\to \cc X$, let $\scr E_\bu:=R\pi_*\scr I(\scr P_\b)$. By the item (i),  $\cc X$ is smooth over $C$ of relative dimension $2n+g$, where the Hodge number $g:=h^1(\cO_{\cc S_p})$ remains constant for the family of the surfaces. The item (ii) is equivalent to $\sff h^2_{\scr E_\bu}=\sff h^2_{\scr E_\bu|_p}$ for any $p\in C$, which was part of the set up of Section \ref{sec:defInvar}. One can check that the item (iii) translates to the requirement of  Corollary \ref{cor:flatDefInv} by noting that $(\scr E_\bu)^\vee$ and $(R\pi_* \scr P_\b)^\vee$ have identical degeneracy loci (cf. Remark \ref{rmk:n2=0 stronger}). This completes the proof.    


\end{proof}

\begin{remark} \label{rmk:prop4.5}
\begin{enumerate}[i)]
    \item A special case of the item (iii) in Proposition \ref{prop:defInvmink} occurs when $h^2(L)=l$  for all effective line bundles $L \in \Pic_\b(\cc S_p)$. In the case $l=0$, by Proposition \ref{proph2l=0}, this gives the deformation invariance of the reduced cycles $[\cc S_{p,\beta}^{[0,n]}]^{\op{red}}$ (cf. \cite[Remark 3.1]{Kool_2014}).
    \item In the case $g=0$, the condition in the item (iii) in Proposition \ref{prop:defInvmink} is vacuous. When $g=1$ the item (iii) requires only the  flatness of all $\op{BN}_{t>s}$ over $C$.

    \item To ensure deformation invariance for the classes $\llhb{\cc S_{p,\b}^{[n_1,n_2]}}_k$, the item (iii) in Proposition \ref{prop:defInvmink} needs to be modified to  a more involved requirement in which the normal cone of $\op{BN}_t\subset\Pic_\b(\cS/C)$ is replaced with the normal cone of the degeneracy locus 
    $$\op{D}_t((\scr E_\bu)^\vee)\subset (\cS/C)^{[n_1]}\times_C (\cS/C)^{[n_2]}\times_C \Pic_\b(\cS/C),$$ where $\scr E_\bu=R\sHom_\pi(\scr I_1,\scr I_2(\scr P_\b))$ for the universal ideal sheaves $\scr I_1$ and $\scr I_2$ corresponding to the first two factors above. This is because when $n_1\neq 0$, $(\scr E_\bu)^\vee$ and $(R\pi_* \scr P_\b)^\vee$ may have different degeneracy loci.  
\item For $g>0$, instead of the item (iii) in Proposition \ref{prop:defInvmink}, one may formulate a sufficient condition for deformation invariance in terms of the dimension of $\op{BN}_{t>s}$ (or dimension of $\op{D}_{t>s}((\scr E_\bu)^\vee)$  in Remark iii above when $n_1\neq 0$)  by using Corollary \ref{cor:irredDefInv} (cf. Proposition \ref{prop:defInvell}).  
\end{enumerate}
    
\end{remark}

\subsection{Comparison formulas} In this section, we prove some formulas comparing the classes $\llhb{S^{[n_1,n_2]}_\b}_k$ and $\llhb{S_\b}_k$. Consider the commutative diagram over $X=S^{[n_1]}\times S^{[n_2]}\times \Pic_\beta(S)$
\begin{equation}\label{eq:hilbcompsquare}
    \begin{tikzcd}
	{R\pi_*\cI_2(\cP_\beta)} & {R\pi_*\cP_\beta}   \\
	{R\sHom_\pi(\cI_1,\cI_2(\cP_\beta))} & {R\sHom_\pi(\cI_1,\cP_\beta).}
	\arrow["\i_2",from=1-1, to=1-2]
    \arrow[from=1-1, to=2-1]
    \arrow["\i_1",from=1-2, to=2-2]
    \arrow[from=2-1, to=2-2]
\end{tikzcd}\end{equation}
As before, we denote by $\sff E=R\sHom_\pi(\cI_1,\cI_2(\cP_\beta))$ and let $\sff G:=R\pi_*\cP_\beta$. Define the $K$-theory class class $$\sff{CO}_\beta^{[n_1,n_2]}:=\sff G(1)-\sff E(1)$$ on $\bb P(B)$. It was shown in \cite[Corollary 8.11]{G1} that for $i>0$,
\beq{COvan}c_{n_1+n_2+i}(\sff{CO}_\beta^{[n_1,n_2]})=0 \quad \text{ on } \quad \wt X=S^{[n_1]}\times S^{[n_2]}\times S_\b.\eeq

The last two formulas in the following theorem were proven before in \cite{G2}.

\begin{theorem}\label{thm:comphilbCO}
    Given a blow up $\nu:Y\to X$ suitable for all the four complexes in \eqref{eq:hilbcompsquare}, then $$G:=\tau^{\leq1}(\nu^*\sff G)-\tau^{\leq1}(\nu^*\sff E)$$ is represented in $K$-theory by a vector bundle of rank $n_1+n_2+\rkh{\sff E}^2-\rkh{\sff G}^2$, and $$i_*\llhb{S_\beta^{[n_1,n_2]}}=\nu_*\left(c_{n_1+n_2+\rkh{\sff E}^2-\rkh{\sff G}^2}\big(G(1)\big)\cap\llb{ \nu^*\sff G,1}\right),$$
where $i$ is the inclusion $S_\beta^{[n_1,n_2]}\Into{} \wt X$. Moreover, as results,
$$i_*[S_\beta^{[n_1,n_2]}]^{\op{vir}}=c_{n_1+n_2}\big(\sff{CO}^{[n_1,n_2]}_\beta\big)\cap[S^{[n_1]}\times S^{[n_2]}]\times[ S_\beta]^{\op{vir}},$$
and whenever the reduced cycle is defined (cf. Proposition \ref{proph2l=0}),
$$i_*[S_\beta^{[n_1,n_2]}]^{\op{red}}=c_{n_1+n_2}\big(\sff{CO}^{[n_1,n_2]}_\beta\big)\cap[S^{[n_1]}\times S^{[n_2]}]\times[ S_\beta]^{\op{red}}.$$
\end{theorem}
\begin{proof}
By the cohomology and basechange theorem, for the labeled maps in \eqref{eq:hilbcompsquare}, $\Cone(\i_2)\simeq\cP_\beta^{[n_2]}$ is a locally free sheaf sitting in degree 0, and $\Cone(\i_1)\simeq\sExt^2_\pi(\cO/\cI_1,\cP_\beta)[-1]$ is a locally free sheaf sitting in degree 1. In particular, Diagram \eqref{eq:hilbcompsquare} satisfies the same conditions as Diagram \eqref{eq:squarecompare}, which lets us apply Theorem \ref{thm:comparesquare} and prove the first formula. The claim about the reduced cycles follows immediately from this, Proposition \ref{proph2l=0} and the projection formula.

To deduce the claim about the virtual cycles from the first formula in the theorem, we work over $\prj(B)$. By Theorem \ref{redtovir}, the virtual cycles on both sides can be obtained as refined cycles of $\sff E$ and $\sff G$. By the surjectivity in Diagram \eqref{fig:h2surj} and the same argument given before Proposition  \ref{prop:deeprefined}, we know that there are locally free sheaves $K_{\sff E}$ of rank $p_g-\rkh{\sff E}^2$ and $K_{\sff G}$  of rank $p_g-\rkh{\sff G}^2$  representing  $\cO^{\oplus p_g}-\sh^2(\nu^*\sff E(1))$ and $\cO^{\oplus p_g}-\sh^2(\nu^*\sff G(1))$, respectively. Both locally free sheaves are defined over the same open set $\nu^{-1}(U)$, where $U\subset \prj(B)$ contains $\wt X$ and was defined after Diagram \eqref{fig:h2surj}.
The right hand side of the first formula can be written as 
\begin{align*}
    &c_{n_1+n_2+\rkh{\sff E}^2-\rkh{\sff G}^2}\big((\nu^*\sff G-\nu^*\sff E+\sh^2(\nu^*\sff E)-\sh^2(\nu^*\sff G))(1)\big)\cap\llb{ \nu^*\sff G}=\\
    &c_{n_1+n_2+\rkh{\sff E}^2-\rkh{\sff G}^2}\big((\nu^*\sff G-\nu^*\sff E)(1)-K_{\sff E}+K_{\sff G}\big)\cap\llb{ \nu^*\sff G}.
\end{align*}
Now, intersecting both sides of the first formula by the top Chern class $c_{p_g-\rkh{\sff E}^2}(K_{\sff E})$, applying Theorem \ref{redtovir} and pushing down by $\nu_*$, we get
$$i_*[S_\beta^{[n_1,n_2]}]^{\op{vir}}=\nu_*\left(c_{n_1+n_2+p_g-\rkh{\sff G}^2}\big(\nu^*\sff{CO}^{[n_1,n_2]}_\beta+K_{\sff G}\big)\cap\llb{ \nu^*\sff G}\right).$$
 By the vanishing of the higher Chern classes \eqref{COvan}, we can separate the top class of $K_{\sff G}$ and then apply Theorem \ref{redtovir} again to get
\begin{align*}
 i_*[S_\beta^{[n_1,n_2]}]^{\op{vir}}&=\nu_*\left(c_{n_1+n_2}\big((\nu^*\sff G-\nu^*\sff E)(1)\big)c_{p_g-\rkh{\sff G}^2}(K_{\sff G})\cap\llb{ \nu^*\sff G}\right)\\
 &=c_{n_1+n_2}(\sff{CO}^{[n_1,n_2]}_\beta)\cap[S^{[n_1]}\times S^{[n_2]}]\times[ S_\beta]^{\op{vir}}.
\end{align*}

\end{proof}

In the case when $n_1=0$ we get a stronger result involving all the refined classes of $S_\beta$ and $S_\beta^{[0,n_2]}$ (cf. Definition \ref{def:refinedNested}). This nested Hilbert scheme is identified with the moduli space of stable pairs on $S$ (cf. \cite[Appendix B.2]{pandharipande2010stable}). Replace $n_2$ by $n$ and $\cI_2$ by $\cI$ and consider  the exact triangle 
\beq{eq:coneI2}R\pi_*\cI(\cc P_\b) \xr{\i}  R\pi_* \cc P_\b \to \cc P_\b^{[n]} \quad  \text{over} \quad S^{[n]}\times \Pic_\b(S).\eeq
Here, $S_\beta^{[n]}:=S_\beta^{[0,n]}$ for simplicity. We have the following theorem.

\begin{theorem}\label{thm:compn1n2}
For any $0\le k\le p_g-\sff h^2_{\sff G}$,
$$i_*\llhb{S_\beta^{[n]}}_k=c_{n}\big(R\pi_*\cI(\cD_\beta)-R\pi_*\cO(\cD_\beta)\big)\cap[ S^{[n]}]\times\llhb{ S_\beta}_k,$$
where $i:S_\beta^{[n]}\Into  S^{[n]} \times S_\beta$.   
\end{theorem}
\begin{proof}
The claim follows from part (1) of Theorem \ref{thm:compdoub} applied to the map $\i$ in the exact triangle \eqref{eq:coneI2}.\end{proof}


\begin{remark}\label{rmk:n2=0 stronger}
The reason that we get stronger result in Theorem \ref{thm:compn1n2} is that by the exact triangle \eqref{eq:coneI2}, $R^2\pi_*\cI(\cc P_\b)\cong R^2\pi_* \cc P_\b$. In particular, a blow up is suitable for the left complex in \eqref{eq:coneI2} if an only if it is suitable for the middle one.
\end{remark}

\subsection{No curves} In this subsection, suppose $\b=0$ and denote the nested Hilbert scheme of points by $S^{[n_1,n_2]}$. In \cite{G1}, it was shown that this is the only degeneracy locus (and hence the deepest) of the complex $R\sHom(\cI_1,\cI_2)$ on $S^{[n_1]}\times S^{[n_2]}$ and the pushforward of its DT virtual cycle is $$c_{n_1+n_2}(-R\sHom(\cI_1,\cI_2)).$$
By a similar argument, one can see that $S^{[n_1,n_2]}$ is the only degeneracy locus of the complex $$\sff E=R\sHom(\cI_1,\cI_2(\cc P_0))\quad \text{over} \quad  X=S^{[n_1]}\times S^{[n_2]}\times \op{Jac}(S),$$ and we get the refined cycles 
$$\llhb{S^{[n_1,n_2]}}_k\in A_{n_1+n_2+p_g-\sff h^2_{\sff E}-k}(S^{[n_1,n_2]}).$$
If $\nu\colon Y\to X$ is an $\sff E$-suitable blow up, and $\iota\colon S^{[n_1,n_2]}\hookrightarrow X$ denotes the inclusion an application of Thom-Porteous formula (cf. Theorem \ref{thm:T-Pformulas} with $B=\cO_{X}$) gives 
$$\iota_*\llhb{S^{[n_1,n_2]}}=\nu_*\left( c_{n_1+n_2+g-p_g+\sff h^2_{\sff E}}\big (-\nu^*\sff E+\fr h^2(\nu^*\sff E)\big)\right).$$
More generally, $\cO^{\oplus p_g
}-\sh^2(\nu^*\sff E)(1)$ behaves like a vector bundle of rank $p_g-\rkh{\sff E}^2$ by the argument as in the proof of Theorem \ref{thm:comphilbCO}. For a suitable bundle $B$, Lemma \ref{lem:pushpicn0}  shows
\begin{align*} &\iota_*\llhb{S^{[n_1,n_2]}}_k=\\&\nu_*\left(s_k\big(\sh^2(\nu^*\sff E)(1)\big) \,c_{b-1+n_1+n_2+g-p_g+\sff h^2_{\sff E}}\big((B-\nu^*\sff E+\fr h^2(\nu^*\sff E))(1)\big)\right)=\\&\sum_{j=0}^k{p_g-\rkh{\sff E}^2-k+j\choose j}\nu_*\left(s_{k-j}\big(\fr h^2(\nu^*\sff E)\big)\cap c_{n_1+n_2+g-p_g+\sff h^2_{\sff E}+j}\big (-\nu^*\sff E+\fr h^2(\nu^*\sff E)\big)\right).\end{align*}
In the case that $k=p_g-\sff h^2_{\sff E}$, one recovers the Chern class formula of \cite{G1} mentioned above.  


\subsection{Wall crossing and duality}
We apply the results of Section \ref{sec:wallform} to nested Hilbert schemes. Let $$p:S_\beta^{[n_1,n_2]}\to X,\qquad \hat p:S_{\hat{\beta}}^{[n_2,n_1]}\to X$$ the natural projections, where $\hat{\beta}:=K_S-\b$. For the perfect complex $\sff E=R\sHom(\cI_1,\cI_2(\cP_\beta))$ on $X$ as before, 
$$\sff e=\rank(\sff E)=\vd_ \beta+\chi(\cO_S)-n_1-n_2.$$

\begin{theorem}\label{thm:hilbwall}
    If $\sff E$ as above satisfies Condition \eqref{UDr}, then the reduced cycles are defined and satisfy 
    $$p_*[S_\beta^{[n_1,n_2]}]^{\op{red}}-(-1)^{\sff e}\hat p_*[S_{\hat{\beta}}^{[n_2,n_1]}]^{\op{red}}=c_{1-\sff e}(- \sff E).$$
    In particular, when $p_g=0$ since $[-]^{\op{red}}=[-]^{\vir}$, we see that
    $$p_*[S_\beta^{[n_1,n_2]}]^{\op{vir}}-(-1)^{\sff e}\hat p_*[S_\beta^{[n_2,n_1]}]^{\op{vir}}=c_{1-\sff e}(-\sff E).$$
\end{theorem}
\begin{proof}
    This is a direct corollary of Theorem \ref{thm:gendual}.
\end{proof}

When $p_g>0$, we work over $q\colon \prj(B)\to X$ as before to obtain a map 
$$v:\cO^{\oplus p_g}_{\prj(B)}(-1)[-2]\To \sff E,$$
 whose cone $\rho(\sff E(1))(-1)$ satisfies Condition \eqref{UDr} because the induced map $\sh^2(u)|_{\wt X}$ is surjective. After replacing $A$ in $B=R\pi_*\cP_\beta(A)$ by a more positive divisor and identifying $\Pic_\beta(S)\cong\Pic_{\hat{\beta}}(S)$, we also get embeddings 
 $$S_{\hat{\beta}}^{[n_2,n_1]}\Into S_{\hat{\beta}}\times S^{[n_1]}\times S^{[n_2]}\Into \prj(B),$$
and for $\hat{\sff E}=\sff E^\vee[-2]$, we have an analogous map
$$\hat v:\cO_{\prj(B)}^{\oplus p_g}(-1)[-2]\To \hat{\sff E},$$
whose cone  satisfies Condition \eqref{UDr} as well. 

\begin{theorem}\label{thm:wallhilb} When $p_g>0$, 
    $$p_*[S_\beta^{[n_1,n_2]}]^{\vir}=(-1)^{\sff \chi(\cO_S)+\vd\beta+n_1+n_2}\hat p_*[S^{[n_2,n_1]}_{\hat{\beta}}]^{\vir}\in A_{n_1+n_2+\vd_\beta}(X).$$
\end{theorem}
\begin{proof}
 The complex $\sff E$ satisfies the hypotheses for Theorem \ref{thm:dualdeep}.  The complexes $\sff H$ and $\hat{\sff H}$ in the proposition correspond to $\Cone(\sff E\to B/\cO(-1))$ and $\Cone(\hat{\sff E}\to B/\cO(-1))$, and as in Theorem \ref{redtovir}, the refined classes in  Theorem \ref{thm:dualdeep} correspond to the virtual cycles in our claim.
\end{proof}



\section{Hilbert scheme of divisors} \label{sec:calcRefined}
In this section, we work with the Hilbert scheme of divisors  $S_\beta$  over a nonsingular complex projective surface $S$ with irregularity $g:=q(S)$ and geometric genus $p_g$. Let $$\sff{E}:=R\pi_*\cP_\beta,\qquad\sff{H}:=\Cone(\sff{E}\to B/\cO(-1)),$$ where $q:\prj(B)\to\Pic_\beta(S)$ is as defined in Section \ref{sec:embvia}. Let $\iota:S_{\beta}\hookrightarrow \prj(B)$. For surfaces with irregularity $g=0$, we find an explicit formula  for $\llhb{S_{\beta}}_k$  for any $\beta$ and $0\le k \le p_g-\sff h^2_{\sff E}$ in Section \ref{sec:g0}. When the irregularity $g\geq 1$, we find an explicit formula for $\iota_*\llhb{S_{K_S}}_k$ for any $0\le k \le p_g-\sff h^2_{\sff E}$ in Section \ref{sec:g1}, which specializes to the famous Seiberg-Witten theory formula $\deg [S_{K_S}]^{\vir}=(-1)^{p_g}$ proven before in \cite{chang2012}.  We also work out blow up formulas for the refined classes analogous to the ones in Seiberg-Witten theory. 

\subsection{Irregularity zero}\label{sec:g0}
When $g=0$ for any curve class $\b$, the Picard variety consists of a single line bundle, $\Pic_\beta(S)=\{L\}$. In this case,
$$\sff{E}\simeq\left\{\cO^{\oplus h^0(L)}\xr{0}\cO^{\oplus h^1(L)}\xr{0}\cO^{\oplus h^2(L)}\right\},$$
and hence no blow up is required in order to define its refined classes $\llhb{S_\b}_k$. These classes obey deformation invariance under the requirements of Proposition \ref{prop:defInvmink} (cf. part ii of Remark \ref{rmk:prop4.5}). 
 Also, $S_\beta=\prj(\sh^0(\sff E))=\prj^{h^0(L)-1}$. Let $H$ denote the hyperplane class of $S_\b$.

\begin{theorem}\label{thm:g=0total}
    We have 
    $$\sum_{k=0}^{p_g-h^2(L)}\llhb{S_\beta}_k=s\big(\cO^{\oplus h^2(L)}_{S_\beta}(1)\big)H^{h^1(L)}.$$
    Thus, for each $0\leq k\leq p_g-h^2(L)=h^0(L)-h^1(L)-1-\vd_\b$,
    $$\llhb{S_\beta}_k=(-1)^k{h^2(L)+k-1\choose k}H^{h^1(L)+k}\in A_{h^0(L)-h^1(L)-1-k}(S_\beta).$$
\end{theorem}
\begin{proof}
   The class $\llhb{S_\beta}$ in $\wt{\op D}_1(\sff E)=\prj^{h^0(L)-1}=S_\beta$ 
    is defined using the composition
    $$\cO(-1)\Into\cO^{\oplus h^0(L)}\overset{0}{\To}\cO^{\oplus h^1(L)}.$$
    $\llhb{S_\beta}$ is then the localized top Chern class of the zero section of $\cO^{\oplus h^1(L)}\otimes\cO(1)$. Therefore,
    $$\llhb{S_\beta}=c_{h^1(L)}(\cO_{S_\beta}^{\oplus h^1(L)}(1))=H^{h^1(L)}.$$

    The claim about other classes follows since $\sh^2(\sff{E})=\cO_{S_\beta}^{\oplus h^2(L)}$.
\end{proof}

\begin{corollary}\label{cor:52}
    Suppose that $p_g>0$. If $h^2(L)=0$, then  
    $$\llhb{S_{\b}}_k=0\qquad \text{ for any } 0<k\leq p_g.$$
    In particular, 
$[S_\beta]^{\vir}=\llhb{S_\b}_{p_g}=0$. If $h^2(L)>0$ and $\vd_\b \ge 0$ then $\vd_\b=0$.
\end{corollary}

\begin{proof}
If $h^2(L)=0$, the vanishing claims follow directly from Theorem \ref{thm:g=0total}. Now, suppose that  $h^2(L)>0$ and $\vd_\b\geq0$.
 Recall that in the proof of Proposition \ref{prop:deeprefined},  we showed $\cO^{\oplus p_g}-q^*\sh^2(\sff E)(1)$ in $K$-theory is represented by  a locally free sheaf $G$ on $S_\b$ of rank $p_g-h^2(L)$. 
By the assumption $\vd_\b\ge 0$,
$$0\leq p_g-h^2(L)+h^1(L)=h^0(L)-1-\vd_\b\leq h^0(L)-1$$
Since $h^2(L)>0$, $K_S-L$ is effective, so $h^0(L)\leq h^0(K_S-L+L)=p_g$ and
$$h^2(L)-h^1(L)=\vd_\b+1+p_g-h^0(L)>0.$$
In our case, $\cO^{p_g}+\cO^{h^1(L)-h^2(L)}(1)=G+\cO^{h^1(L)}(1)$ in $K$-theory, so we have $$c_{j}\big(\cO^{h^1(L)-h^2(L)}(1)\big)=0 \qquad j>p_g-h^2(L)+h^1(L).$$
But since $h^1(L)-h^2(L)<0$, $c_{h^0(L)-1}(\cO^{h^1(L)-h^2(L)}(1))\neq 0$, so the only possibility is that 
$$p_g-h^2(L)+h^1(L)=h^0(L)-1,$$
which is equivalent to $\vd_\beta=0$. 
\end{proof}

\begin{corollary}
    When $p_g>0$, the only classes $\beta$ with $[S_\b]^{\vir}\neq0$ are those with $\vd_\beta=0$ and $h^2(L)>0$.
\end{corollary}
\begin{proof}
 This is a direct application of the Theorem \ref{thm:g=0total} and Corollary \ref{cor:52}. 
\end{proof}

\begin{corollary}\label{cor:degks}
    In the case of $\beta= K_S$,
    $$\sum_{k=0}^{p_g-1}\llhb{S_{K_S}}_k=s(\cO_{S_{K_S}}(1)).$$
    In particular, 
$\deg[S_\beta]^{\vir}=\deg\llhb{S_{K_S}}_{p_g-1}=(-1)^{\chi(\cO_S)}\in A_0(S_{K_S})$.
\end{corollary}
\begin{proof}
    Note that by Serre duality $h^2(K_S)=1$ and $h^1(K_S)=h^1(\cO_S)=0$ proving the first claim. For the second claim, note that the coefficient of $H^{p_g-1}$ in $\frac{1}{1+H}$ is $(-1)^{p_g-1}$, which is $(-1)^{\chi(\cO_S)}$ in this case.
\end{proof}

\subsection{Canonical class}\label{sec:g1}
In this section we assume that $g>0$ and $p_g>0$. These imply that the Hilbert scheme $S_{K_S}\neq \emptyset$, and the Picard variety $\Pic_{K_S}(S)$ is positive dimensional. Let $\sff E:=R\pi_*\cP_{K_S}$.
\begin{lemma} \label{lem:jac}
    Let $x=\{\cO(K_S)\}\in\Pic_{K_S}(S)$ be the closed point corresponding to the canonical line bundle, then
    $$\sh^2(\sff E)=\cO_{x}.$$
\end{lemma}
\begin{proof}
    If $L\in \Pic_{K_S}(S)$, then $L^*(K_S)$ is a degree 0 line bundle, by Serre duality $$h^2(L)=\begin{cases}\bb C & L=\cO(K_S),\\ 0 &L\neq \cO(K_S).\end{cases}$$ Therefore,  $\sh^2(\sff{E})$ is set theoretically supported on $x$, and  $\op{rk}(i^*\sh^2(\sff{E}))= 1$ where $i:\{x\} \hookrightarrow \Pic_{K_S}(S)$ is the inclusion. We claim that it is scheme theoretically supported on $x$. By Serre duality, the claim is equivalent to showing that the degeneracy locus  $${\op{D}}_1(R\pi_* \cP_0)\into{j} \Jac(S)$$ is the reduced point $\{\cO_S\}$.  By the discussion above, we certainly have ${\op{D}}_1(R\pi_* \cP_0)=\{\cO_S\}$ as sets, and also $\sh^0:=\sh^0(\bff Lj^* R\pi_* \cP_0)$ is a line bundle by \cite[Lemma 3.3]{G1}. But this line bundle must  be trivial  as ${\op{D}}_1(R\pi_* \cP_0)$ is Artinian. So there is a nowhere vanishing section of $\sh^0$ that gives a nowhere vanishing section of $\cP_0|_{S \times {\op{D}}_1}$. This line bundle is  trivial and therefore its classifying map ${\op{D}}_1(R\pi_* \cP_0)\to \Jac(S)$ factors through $\{\cO_S\}$. This proves  
${\op{D}}_1(R\pi_* \cP_0)=\{\cO_S\}$ as schemes.
\end{proof}

By Lemma \ref{lem:jac} and Remark \ref{remk:blsupp}, the blow up  of $\Pic_{K_S}(S)$ at the point $\{\cO_{K_S}\}$  is $\sff E$-suitable. 
\begin{prop}
    The refined classes $\llhb{S_{K_S}}_k$ for $k\ge 0$ obey deformation invariance for a smooth family of projective surfaces  over a nonsingular curve having the surface $S$ as a fiber. 
\end{prop} 
\begin{proof}
    By the description of the blow up given above, we have $\cc Y_0=Y$ and hence $S=\emptyset$ in the notation of Theorem \ref{thm:defInv} in which we take $\cc Y$ to be the blow up of the relative Picard variety of the family of surfaces at the section corresponding to the canonical line bundle of each fiber.
\end{proof}

Denote by $\nu:Y\to\Pic_{K_S}(S)$ the blow up of $\Pic_{K_S}(S)$ at the point $\{\cO_{K_S}\}$. Let $N\cong\prj^{g-1}$ be its exceptional divisor. For $$Z=q^{-1}(\{\cO_{K_S}\})\subset\prj(B),$$ the induced blow up $\nu:Y_B\to\prj(B)$ is given by blowing up $Z\subset\prj(B)$ and it is $\sff F$-suitable. Denote by $\tilde Z:=\nu^{-1}(Z)\subset Y_B$ the exceptional divisor, and by $H:=c_1(\cO_{\prj(B)}(1))$. We get the cartesian diagram (so $Y_B\cong\prj(\nu^*B)$) 
\[\begin{tikzcd}[cramped]
	{Y_B} & Y \\
	{\prj(B)} & {\Pic_{K_S}(S).}
	\arrow["{\tilde q}", from=1-1, to=1-2]
	\arrow["\nu", from=1-1, to=2-1]
	\arrow["\nu", from=1-2, to=2-2]
	\arrow["q", from=2-1, to=2-2]
\end{tikzcd}\]
Note that   $\cO(\tilde Z)\cong\tilde q^*\cO(N)$ so that $\nu_*(\tilde Z^j)=0$ if $j\neq g$ and $\nu_*(\tilde Z^g)=(-1)^{g-1}Z$. Also,
$$\tilde Z\cong N\times_{x}Z\cong\prj^{g-1}\times\prj^{b-1}.$$

\begin{prop} \label{prop:SKS}
    After pushing forward into $\prj(B)$, 
    $$\iota_*\llhb{S_{K_S}}=c_{b-1+g-p_g}(-\sff{H}(1))+(-1)^{g-1}Z\cdot H^{b-1-p_g}\in A_{p_g}(\prj(B)).$$
\end{prop}
\begin{proof}
    Denoting by $\{-\}_{j}$ the part of codimension $j$ of an expression and defining $\ell:=b-1+g-p_g$, we have 
    \begin{align*}
         \iota_*\llhb{S_{K_S}}&=\nu_*c_{b-(p_g-g+1)}\big(-\nu^*\sff{H}(1)+\cO_{\tilde Z}(1)\big)\\
    &=c_\ell(-\sff{H}(1))+\sum _{i=1}^{\ell}c_{\ell-i}\big(-\sff{H}(1)\big)\nu_*c_i(\cO_{\tilde Z}(1))\\
    &=c_\ell(-\sff{H}(1))+\sum _{i=1}^{\ell}c_{\ell-i}\big(-\sff{H}(1)\big)\nu_*\big((\tilde Z-H)^{i-1}\tilde Z\big).
    \end{align*}
Recall that  $\nu_*(\tilde Z^j)=0$ for $j\neq g$ and $\nu_*(\tilde Z^g)=(-1)^{g-1}Z$. 
Thus,
\begin{align*}
         \iota_*\llhb{S_{K_S}}&=c_\ell(-\sff{H}(1))+\sum _{i=1}^{\ell}c_{\ell-i}(-\sff{H}(1)){i-1\choose g-1}\nu_*\big(\tilde Z^g(-H)^{i-g}\big)\\
         &=c_\ell(-\sff{H}(1))+\sum _{i=1}^{\ell}c_{\ell-i}(-\sff{H}(1)){i-1\choose g-1}(-1)^{g-1}Z(-H)^{i-g}.
    \end{align*}
    Since $Z$ is the fiber over a point of $\Pic_{K_S}(S)$, $-\sff{H}{|_Z}=\cO_Z^{\oplus\ell}$ in $K$-theory, so
\begin{align*}
         \iota_*\llhb{S_{K_S}}&=c_\ell(-\sff{H}(1))+\sum _{i=1}^{\ell}c_{\ell-i}(\cO^{\oplus\ell}(1)){i-1\choose g-1}(-1)^{i-1}Z H^{i-g}\\
         &=c_\ell(-\sff{H}(1))+\sum _{i=1}^{\ell}(-1)^{i-1}{\ell\choose i}{i-1\choose g-1}Z\cdot H^{b-1-p_g}\\
         &=c_\ell(-\sff{H}(1))+(-1)^{g-1}Z\cdot H^{b-1-p_g}.
    \end{align*}
    where the last equality follows from the following lemma.
\end{proof}

\begin{lemma}\label{lem:altbinom}
    Let $a>c\geq0$ be integers. Then
    $$\sum_{k=c+1}^{a}(-1)^k{a\choose k}{k-1\choose c}=(-1)^{c+1}.$$
\end{lemma}
\begin{proof}
    The claimed identity is clearly true for  $a=c+1$, and also for $c=0$ in which case the left hand side consists of all the terms of $(1-1)^a$ except for $1^a(-1)^0$. We use  induction on $a$. For $a=1$ it holds true, as then $c=0$. Suppose the identity is true for $a-1$, and let $0\leq c<a$. If $c=a-1$ or $c=0$ we know the equality holds, otherwise, 
    \begin{align*}
        &\sum_{k=c+1}^{a}(-1)^k{a\choose k}{k-1\choose c}\\
        &=\sum_{k=c+1}^{a}(-1)^k{a-1\choose k}{k-1\choose c}+\sum_{k=c+1}^{a}(-1)^k{a-1\choose k-1}{k-1\choose c}\\
        &=(-1)^{c+1}+\sum_{k=c+2}^{a}(-1)^k{a-1\choose k-1}{k-2\choose c}+\sum_{k=c+1}^{a}(-1)^k{a-1\choose k-1}{k-2\choose c-1}\\
        &=(-1)^{c+1}+(-1)^{c+2}+(-1)^{c+1}=(-1)^{c+1}.
    \end{align*}
\end{proof}

\begin{prop} \label{propSKSk}
    For $k>0$
    $$\iota_*\llhb{S_{K_S}}_k=(-1)^{k+g-1}Z H^{b-1+k-p_g}\in A_{p_g-k}(Z).$$
\end{prop}

\begin{proof} Letting $\ell=b-1+g-p_g$ again, we have
\begin{align*}
\iota_*\llhb{S_{K_S}}_k&=\nu_*\left(s_{k}\big(\nu^*h^2(\sff{H})(1)\big)\cap\llb{\nu^*\sff{H}}\right)\\
    &=\nu_*\left(s_{k}(\cO_{\tilde Z}(1))c_{b-(p_g-g+1)}\big(-\nu^*\sff{H}(1)+\cO_{\tilde Z}(1)\big)\right)\\
    &=\nu_*\left(\left\{\frac{1+H-\tilde Z}{1+H}\right\}_{k}\sum _{i=0}^{\ell}c_{\ell-i}\big(-\nu^*\sff{H}(1)\big)c_i\big(\cO_{\tilde Z}(1)\big)\right)\\
    &=\sum _{i=0}^{\ell}c_{\ell-i}(-\sff{H}(1))\nu_*\left(c_i(\cO_{\tilde Z}(1))\left\{\frac{1+H-\tilde Z}{1+H}\right\}_{k}\right)\\
     &=\sum _{i=0}^{\ell}c_{\ell-i}(-\sff{H}(1))\nu_*\left(\left\{\frac{1+H}{1+H-\tilde Z}\right\}_{i}(-1)^{k}\tilde Z H^{k-1}\right).
\end{align*}
 For $i>0$, we have
\begin{align*}
    \left\{\frac{1+H}{1+H-\tilde Z}\right\}_{i}(-1)^{k}\tilde Z H^{k-1}&=(\tilde Z-H)^{i-1}\tilde Z(-1)^{k}\tilde Z H^{k-1}\\
    &=(-1)^{k} \tilde Z^2(\tilde Z-H)^{i-1}H^{k-1}.
\end{align*}
After applying $\nu_*$ to it, we get $$(-1)^{k}{i-1\choose g-2}(-1)^{g-1}Z H^{i-g+k}(-1)^{i-g+1}.$$ This is $0$, when $g=1$. For $i=0$, we have 
\begin{align*}
     \left\{\frac{1+H}{1+H-\tilde Z}\right\}_{0}(-1)^{k}\tilde Z H^{k-1}&=(-1)^{k}\tilde Z H^{k-1}.
\end{align*}
After applying $\nu_*$, this term vanishes when $g\geq2$. 

Recall that $-\sff{H}{|_Z}=\cO_Z^{\ell}$ in $K$-theory. We distinguish two cases: if $g=1$ only one term remains
    \begin{align*}
        \iota_*\llb {S_{K_S}}_{k}=&(-1)^{k}c_{\ell}(-\sff{H}(1))ZH|_Z^{k-1}\\
        =&(-1)^{k}c_{\ell}(\cO(1)^{\ell})H|_Z^{k-1}=(-1)^{k+g-1}H|_Z^{b-1+k-p_g},
    \end{align*}
and if $g\geq2$
\begin{align*}
    \iota_*\llhb{S_{K_S}}_k&=(-1)^{k}\sum _{i=1}^{\ell}(-1)^{i+k}{i-1\choose g-2}c_{\ell-i}(-\sff{H}(1))Z H^{i-g+k}\\
    &=\sum _{i=1}^{\ell}(-1)^{i+k}{i-1\choose g-2}c_{\ell-i}(\cO(1)^{\ell})H|_Z^{i-g+k}\\
    &=\sum _{i=1}^{\ell}(-1)^{i+k}{i-1\choose g-2}{{\ell\choose \ell-i}}H^{b-1+k-p_g}|_Z\\
    &=(-1)^kH^{b-1+k-p_g}|_Z\sum _{i=g-1}^{\ell}(-1)^{i}{i-1\choose g-2}{{\ell\choose i}}\\
    &=(-1)^{k+g-1}H|_Z^{b-1+k-p_g},
\end{align*}
where last equality follows from Lemma \ref{lem:altbinom}.
\end{proof}

Putting together Proposition \ref{prop:SKS} and \ref{propSKSk}, we prove the following theorem.

\begin{theorem}\label{thm:classKS}
    When $g, p_g >0$ for the surface $S$, we have
    $$\sum_{k=0}^{p_g}\iota_*\llhb {S_{K_S}}_k=c_{b-1+g-p_g}(-\sff{H}(1))+\sum_{k=0}^{p_g}(-1)^{k+g-1} H|_Z^{b-1+p_g-k}$$
    where $Z\cong\prj^{b-1}\subset\prj(B)$.
\end{theorem}
\begin{corollary} \label{cor:chang} 
    $\deg[S_{K_S}]^{\vir}=(-1)^{\chi(\cO_S)}$.
\end{corollary}
 \begin{proof}
    Using Theorems \ref{redtovir} and \ref{thm:classKS}, we can write
    \begin{align*}
\deg[S_{K_S}]^{\vir}&=\deg\llhb{S_{K_S}}_{p_g}=(-1)^{p_g+g-1}\deg (H|_Z^{b-1})=(-1)^{p_g+g-1}.
    \end{align*}
 \end{proof}

\subsection{Blow up formulas}
Let $f:\wh S\to S$ be the blow up of $S$ at a closed point $x \in S$ with $C$ the exceptional curve. Any curve class $\wh \b\in H^2(\wh S,\bb Q)$ can be uniquely written as $\wh \b=f^* \beta+ \ell C $ for some $\ell \in \bb Z$ and some curve class $\b \in H^2( S,\bb Q) $.

For the integer $\ell$ above, we have a morphism \beq{blmap} \mu_\ell:\wh S_{\wh \b}\to S_\b\eeq defined by pushing down divisors. 
If $\ell \ge 0$ then $\mu_\ell$ is an isomorphism, and if $\ell <0$ it is a closed immersion. We denote by $\wh {\sff  G}:=R\wh \pi_* \cc P_{\wh \b}$ and $\wh{\sff F}:=R\wh \pi_* \cc P_{f^*\b}$ over $$X=\Pic_{\b}(S)\cong \Pic_{\wh \b}(\wh S)$$ with the projection $\wh \pi \colon \wh S\times X\to X$. The following theorem is an analog of the blow up formula for the Seiberg-Witten invariants of algebraic surfaces proven in \cite[Section 3.3]{DKO}. Instead of the virtual cycles, it is about the refined and the reduced cycles of the Hilbert scheme of divisors. 

\begin{theorem} \label{thm:blow up} Let $\mu_\ell:\wh S_{\wh \b}\to S_\b$ be as in \eqref{blmap} and $\nu\colon Y\to X$ be any $(\wh{ \sff F},\wh{\sff E})$-suitable blow up. Then,
$$\mu_{\ell*} \llhb{\wh S_{\wh \b}}=\begin{cases} \llhb{S_\b} & \ell=0, \\ c_1\big(\cO(\cc D_\b\big)|_{\{x\}\times S_\b}\big)^{\ell(\ell+1)/2} \cap \llhb{S_\b} & \ell <0,
 \\  \nu_*\left (c_{{\ell\choose 2}-\rkh{\wh{\sff G} }^2+\rkh{\wh{\sff F}}^2}\big(\tau^{\leq1}\nu^*\wh{\sff G}(1)-\tau^{\leq1}\nu^*\wh{\sff F}(1)\big) \cap \llb{\nu^*\wh{\sff F},1} \right)& \ell >0. \end{cases}$$
 If the reduced cycle is defined (cf. Proposition \ref{proph2l=0}), then we can replace $\llhb{-}$ by $[-]^{\op{red}}$ in above formulas; in particular, the case of $\ell>1$ gives 
 $$\nu_{\ell*}[\wh S_{\wh\beta}]^{\op{red}}=c_1(\cO(\cc D_\b)|_{\{x\}\times S_\b})^{\ell(\ell-1)/2}  \cap [S_\beta]^{\op{red}}.$$
   \end{theorem}
  
\begin{proof} 
Consider the natural exact triangles
\begin{align*}& R\wh \pi_* \cc P_{f^* \b} \xr{s_{\ell  C}}  R\wh \pi_* \cc P_{\wh \b}\to  R\wh \pi_* (\cc P_{\wh \b}|_{\ell  C \times \Pic_{f^*\b}(\wh S)}) \qquad \ell \ge 0,\\
& R\wh \pi_* \cc P_{\wh \b} \xr{s^*_{-\ell  C}}  R\wh \pi_* \cc P_{f^* \b}\to  R\wh \pi_* (\cc P_{f^* \b}|_{-\ell  C \times \Pic_{f^*\b}(\wh S)}) \qquad \ell <0, \end{align*} over $\Pic_{f^*\b}(\wh S)$, where for $\ell=0$ the rightmost term is 0. By basechange and analyzing over closed points, it can be seen that $ R\wh \pi_* (\cc P_{\wh \b}|_{\ell  C \times \Pic_{f^*\b}(\wh S)})$ is a vector bundle in degree 1 for $\ell >0$, and $ R\wh \pi_* (\cc P_{f^* \b}|_{-\ell  C \times \Pic_{f^*\b}(\wh S)})$ is a vector bundle in degree 0 for $\ell <0$.
Using induction on $\ell$ and the identifications $\Pic_{f^*\b}(\wh S)\cong \Pic_{\b}(S)$ and $\wh S_{f^*\b}\cong S_\b$, shows that (see \cite[Section 3.3]{DKO} for a similar argument) 
\begin{align*}
&c\big( R^1\wh \pi_* (\cO(\cc D_{\wh \b})|_{\ell  C \times \wh S_{f^*\b}})  \big)=\left(1+c_1\big(\cO(\cc D_\b)|_{\{x\}\times S_\b}\big)\right)^{\ell(\ell-1)/2} \quad  \ell >0,\\
&c\big(R^1\pi_* (\cO(\cc D_{f^* \b})|_{-\ell  C \times \wh S_{f^*\b}})   \big)=\left(1+c_1(\cO(\cc D_\b)|_{\{x\}\times S_\b})\right)^{\ell(\ell+1)/2}  \qquad  \ell < 0.
 \end{align*} 
The formulas now follow  from Theorem \ref{thm:compdoub}. Part (1) of the theorem gives the formula for $\ell<0$ and part (2) for  $\ell>0$. 
\end{proof}

\section{Elliptic Surfaces} \label{sec:elliptic}
In this section, we let $f:S\to C$ be a relatively minimal elliptic surface over a nonsingular curve $C$ of genus $g:=g(C)$. Denote by $F_i$ the multiple fibers of $S$ and their multiplicities by $m_i$. Let $p_g=h^2(\cO_S)$ and $q:=q(S)=h^1(\cO_S)$. Such surfaces always have $\chi:=\chi(\cO_S)\geq0$. By the canonical bundle formula, 
\beq{eq:cann}\cO(K_S)\cong{f}^*(\cO_C(K_C)\otimes M^*)\otimes\cO_S\left(\sum (m_i-1)F_i\right),\eeq
where $M=R^1{f}_*\cO_S$ is a line bundle of degree $-\chi$ on $C$. As long as $M\neq\cO_C$, it is true that $q(S)=g(C)$, but when $M=\cO_C$, $q(S)=g(C)+1$. In this section, we only consider the elliptic surfaces of the first type. 

\begin{lemma}\label{lem:effectivedual}
    Any effective line bundle  $L\in \Pic(S)$ with $h^2(L)\neq0$ must be of the form
    $$L\cong f^*\cO_C(D)\otimes \cO_S(\sum  a_iF_i),$$
    where $D$ is an effective divisor on $C$ and $0\leq a_i<m_i$ for all $i$.
\end{lemma}
\begin{proof}
    If such an $L$ exists then $\cO(K_S)\cong (\cO(K_S)\otimes L^*)\otimes L$ is effective, so  $p_g>0$. Formula \eqref{eq:cann} shows that for $F$ the class of a general fiber, $K_S$ is numerically equivalent to $rF$ with $r\in\Q$. Since $K_S$ is effective, $r\geq 0$, and hence $K_S$ is nef.

    If $S$ is an abelian or K3 surface, then $r=0$ and $K_S\cong\cO_S$. Thus, $L\cong\cO_S$, which satisfies our claim. Otherwise, $S$ must be properly elliptic, and $r>0$. Writing $L=\cO(E)$ for some effective divisor $E$ on $S$,  since $K_S$ is nef and $K_S^2=0$, we have
    $$0\leq K_S\cdot(K_S-E)=K_S^2-K_S\cdot E=-K_S\cdot E\leq 0.$$
    Therefore, $K_S\cdot E=rF\cdot E=0$ and so $E$ must be a vertical divisor. Thus, $L$ can be written as claimed  with $\cO_C(D):=f_*L$.
\end{proof}

\subsection{Split complexes revisited}\label{sec:61}
We study effective divisors of the form given in Lemma \ref{lem:effectivedual}. Let \beq{verticalbeta} \b:=c_1(L)\quad \text{where} \quad L:={f}^*\cO(D)\otimes\cO_S(\sum  a_iF_i)\eeq with $D$ an effective divisor on $C$ of degree $d$ and $0\leq a_i<m_i$. By the assumption $q(S)=g(C)$, the morphism $$\Pic_d(C)\to \Pic_\beta(S)\qquad L_0\mapsto {f}^*L_0\otimes\cO_S(\sum a_i F_i)$$  is an isomorphism. We also have the following property.

\begin{lemma}\label{lem:projection}
    For integers $0\leq a_i<m_i$, 
    $$f_*\cO_S(\sum a_i F_i)=\cO_C.$$
\end{lemma}
\begin{proof}
    We use induction on $\sum a_i\geq0$. The claim is clear when all $a_i=0$. Suppose now that some $a_j>0$, we have an exact sequence
    $$0\To\cO_S(-F_j+\sum a_i F_i)\To \cO_S(\sum a_i F_i)\To \cO_{F_j}(\sum a_i F_i)\To0.$$
    Since $\cO_{F_j}(\sum a_i F_i)\cong \cO_{F_j}(a_j F_j)$ and $\cO_{F_j}(F_j)$ is a torsion sheaf of order $m_j$, then $f_*\cO_{F_j}(\sum a_iF_i)\cong0$. By the exact sequence and our inductive hypothesis   $$f_*\cO_S(\sum a_i F_i)\cong f_*\cO_S(-F_j+\sum a_i F_i)\cong\cO_C.$$
    \end{proof}

Denote by $\cP_d$ a Poincar\'e line bundle over $C\times \Pic_d(C)$, then $\cP_\beta:={f}^*\cP_d\otimes\op{pr}_1^*\cO(\sum a_iF_i)$ is a Poincar\'e line bundle over $S\times \Pic_d(C)$ with the property that ${f}_*\cP_\beta\cong\cP_d$ (cf. Lemma \ref{lem:projection}).  Here, we use the same symbol for $f$ and $f\times\op{id_{\Pic}}$. Denote by $\pi$ and $\psi$ the obvious projections commuting the diagram
\begin{equation} \label{pictriangle}
    \begin{tikzcd}
	{S\times \Pic_d(C)} && {C\times \Pic_d(C)} \\
	& {\Pic_d(C).}
	\arrow["{f}", from=1-1, to=1-3]
	\arrow["\pi"', from=1-1, to=2-2]
	\arrow["{\psi}", from=1-3, to=2-2]
\end{tikzcd}
\end{equation}
In this case, we also have the identification of the  Hilbert schemes $S_\beta\cong C_d$ via the pullback of the divisors. In particular $S_\b$ is a nonsingular variety of dimension $d$. Let  $\theta$ be the class of the theta divisor on $J(C)\cong \Pic_d(C)$. Denote by $\bar\theta:=\psi^*\theta$. Also, let $\cO_{C_d}(1)$ be the universal line bundle, and $x:=c_1(\cO_{C_d}(1))$. 

\begin{lemma} \label{lem:ellsplit} We have 
$$R\pi_* \cP_\beta\cong R\psi_*\cP_d\oplus R\psi_*\cP_d(M)[-1].$$
In particular,   $R^2\pi_{*}\cP_\beta\cong R^1\psi_{*}\cP_d(M)$ and $R^1\pi_{*}\cP_\beta\cong \psi_{*}\cP_d( M)\oplus  R^1\psi_{*}\cP_d$.
\end{lemma}
\begin{proof}
By \eqref{eq:cann}, the relative canonical bundle is $$\omega_{S/C}=f^*M^*\otimes \cO(\sum (m_i-1)F_i),$$ where as before, $M=R^1f_*\cO$. By the relative duality and Lemma \ref{lem:projection}, 
$$R^1{f}_*\cO(\sum a_iF_i)\cong(f_*\cO(\sum (m_i-1-a_i)F_i)\otimes M^*)^*\cong M.$$

We defined $\cP_\beta={f}^*\cP_d\otimes\op{pr}_1^*\cO(\sum a_i F_i)$, so by the projection formula 
$$R^j{f}_{*}\cP_\beta\cong \cP_d\otimes R^j{f}_*\cO_{S\times \Pic_d(C)}(\sum a_iF_i)\cong\begin{cases}
    \cP_d, &j=0,\\
    \cP_d(M),\;&j=1,\\
    0,&j\geq2.
\end{cases}$$
Since $R^{i\geq2}\psi_{*}(R^j{f}_{*}\cP_\beta)=0$ as $\psi$ has relative dimension $1$, the second page $E_2^{i,j}$ of the spectral sequence for $R\pi_{*}\cP_\beta=R(\psi\circ{f})_{*}\cP_\beta$ has only terms in degrees $(i,j)$ with $1\leq i, j\leq 2$, so it  converges and $E_2^{i,j}=E_\infty^{i,j}$. The claims of the lemma follow by writing down these terms. 
\end{proof}
Define the following complexes and fix locally free resolutions for them as $$\sff{E}:=R\pi_{*}\cP_\beta\simeq\{E_0\to E_1\to E_2\},$$  $$\sff{A}:=R\psi_{*}\cP_d\simeq\{A_0\to A_1\},\qquad  \sff{G}:=R\psi_{*}\cP_d(M)\simeq\{G_0\to G_1\}.$$ By Lemma \ref{lem:ellsplit}, we have
\begin{equation}\label{eq:splitell}
    \sff{E}\cong \sff{A}\oplus \sff{G}[-1].
\end{equation}

\begin{lemma}\label{lem:rankPoincare}
For any $d\in\Z$,
$$\rk(R^0\psi_*\cP_d)=\max\{0,d-g+1\},\quad \rk(R^1\psi_*\cP_d)=\max\{0,g-d-1\}.$$
\end{lemma}
\begin{proof}
    \cite[Lemma IV.3.5]{arbarello1} states that whenever $g-d+r\geq0$, no irreducible component of $W_d^r:=\op{D}_{r+1}(R\psi_*\cP_d)$ is contained in $W_d^{r+1}:=\op{D}_{r+2}(R\psi_*\cP_d)$. In particular, this shows that $W_d^r\setminus W_d^{r+1}\neq\emptyset$ when $W_d^r\neq \emptyset$.
    
    By Riemann-Roch, $W_d^{d-g}=\Pic_d(C)$.  
    Given $L_0\in W_d^{d-g}\setminus  W_d^{d-g+1}$, then $h^0(L)=d-g+1$, proving the first claim.
Since $\rk(R\psi_*\cP_d)=d-g+1$, the second claim follows.
\end{proof}

Our goal is to give an explicit evaluation of the classes $\llhb{S_\beta}_k=\llb{\sff E,1}_k$. We will distinguish three cases: \\ 1) $d\le p_g$,\\ 2) $p_g<d\le p_g+g-1$,\\ 3) $d>p_g+g-1$.  \\ Note that for $d\leq  p_g+g-1$ and $p_g>0$, the reduced cycle $[S_\beta]^{\op{red}}$ is not defined.
\subsubsection{Case $d\leq p_g$}

As a consequence of the following lemma, no blow up is needed in order to define the refined classes. 
\begin{lemma}\label{prop:ell1noblow}
    In this case, $\hd(R^2\pi_{*}\cP_\beta)\leq 1$ and $\rk(R^2\pi_{*}\cP_\beta)=p_g-d$. Moreover,
    $$\tau^{\leq1}(R\pi_{*}\cP_\beta)=R\psi_{*}\cP_d,$$
    so that $\llhb{S_\beta}=\llb{R\psi_*\cP_d,1}.$
\end{lemma}
\begin{proof}
We have $R^2\pi_{*}\cP_\beta\cong R^1\psi_{*}\cP_d(M)$ and $\cP_d(M)$ is a Poincar\'e line bundle over $C\times \Pic_{d-\chi}(C)$. By Lemma \ref{lem:rankPoincare}, since $d-\chi\leq g-1$, $\rk(R^0\psi_{*}\cP_d(M))=0$, so that $R^0\psi_{*}\cP_d(M)=0$ and $\rk(R^1\psi_*\cP_d(M))=g-1-d+\chi=p_g-d$. By our choice of the locally free resolution above, we have an exact sequence $$0\to G_0\to G_1\to R^1\psi_{*}\cP_d(M)\to0.$$ This proves that $\hd(R^2\pi_*\cP_\beta)\leq1$ and $\tau^{\leq1}(\sff{E})=\sff{A}$. \end{proof}

\begin{lemma}  $\llhb{S_\beta}=[S_\beta]$. 
\end{lemma}
\begin{proof} By Lemma \ref{prop:ell1noblow}, $\llhb{S_\beta}=[\wt{\op{D}}_1(\sff A)]^{\vir}\in A_{p_g-\sff h^2_{\sff E}}(S_\b)$ and $$p_g-\sff h^2_{\sff E}=p_g-(p_g-d)=d=\dim S_\b.$$ Since $S_\beta=\wt {\op{D}}_1(\sff{A})$ is smooth of the expected dimension, the claim follows.

\end{proof}

The following theorem generalizes \cite[Proposition 4.8]{DKO}.

\begin{theorem}\label{totellcase1} For any $0\le k\le d$, $\llhb{S_\beta}_k$ is given by
$$\sum_{i=0}^{k}{d-p_g-i\choose k-i}\frac{(-1)^i\theta^ix^{k-i}}{i!}=(-1)^k\sum_{i=0}^{k}{p_g-d-1+k\choose k-i}\frac{\theta^ix^{k-i}}{i!}.$$
In particular, when $k=d$, since $\deg(\theta^i x^{d-i})=\frac{g!}{(g-i)!}$ we have that
$$\deg[S_\beta]^{\vir}=\deg\llhb{ S_\beta}_d=(-1)^{d}{2g-2+\chi(\cO_S)\choose d}.$$
\end{theorem}
\begin{proof}
    By definition, in this case 
    $$\llhb{ S_\beta}_k=[S_\beta]\cap s_k\big(R^1\psi_{*}\cP_d(M)(1)\big)$$
    Since $R^0\psi_{*}\cP_d(M)=0$,$R^1\psi_{*}\cP_d(M)=-R\psi_{*}\cP_d(M)$. By Lemma \ref{chernpic}, 
    $$\llhb{ S_\beta}_k=c_k\big(R\psi_{*}\cP_d(M)(1)\big)=\{(1+x)^{N}e^{\frac{-\theta}{1+x}}\}_k,$$
   where $N=d-g+1-\chi=d-p_g$ Using Lemma \ref{combchern}, we see that
    \begin{align*}
        c_k\big(R\psi_{*}\cP_d(M)(1)\big)
        &=(-1)^{k}\sum_{i=0}^k{k-(d-p_g)-1\choose k-i}\frac{\theta ^ix^ {k-i}}{i!}.\\
    \end{align*}
Finally, when $k=d$, $\deg\llhb{S_\beta}_d=$
    $$(-1)^d\sum_{i=0}^{d}{p_g-1\choose d-i}{g\choose i}=(-1)^{d}{p_g+g-1\choose d}=(-1)^{d}{\chi+2g-2\choose d}.$$
\end{proof}

\begin{corollary}\label{cor:pushell1}
    For the inclusion $\iota:S_\beta\hookrightarrow\prj(B)$ as  Section \ref{sec:embvia}, we have
    $$\iota_*\llhb{ S_\beta}_k=c_{b-1+g-d}\big(B(1)-R\psi_{*}\cP_d(1)\big)c_{k}\big(R\psi_{*}\cP_d(M)(1)\big).$$
    Pushing down along $q\colon \bb P(B)\to \Pic_\beta(S)$ we get:
    \begin{align*}    
        q_*\iota_*\llhb{S_\beta}_k=(-1)^{k}\frac{\theta^{g-d+k}}{g-d+k!}{g-1+p_g-2d+2k\choose k}
    \end{align*}
\end{corollary}
\begin{proof}
    Since $\tau^{\leq1}(\sff{E})\simeq \sff{A}$, the first claim follows from \eqref{eq:TP-Gr(B)}  . By Lemma \ref{lem:pushpic}, 
    \begin{align*}    
        q_*\iota_*\llhb{S_\beta}_k&=\sum_{j=0}^{p_g-d+k-1}(-1)^j{p_g-d+k-1\choose j}c_{g-d+j}\big(-R\psi_{*}\cP_d\big)c_{k-j}\big(R\psi_{*}\cP_d(\cM)\big)\\
        &=(-1)^{k}\frac{\theta^{g-d+k}}{(g-d+k)!}\sum_{j=0}^{p_g-d+k-1}{p_g-d+k-1\choose j}{g-d+k\choose k-j}\\
        &=(-1)^{k}\frac{\theta^{g-d+k}}{(g-d+k)!}{g-1+p_g-2d+2k\choose k}.
    \end{align*}
\end{proof}

\subsubsection{Case $p_g< d\le p_g+g-1$}
In this section, we work under the assumption that the curve $C$ is general.  Under this assumption, we will show that the results of Section \ref{sec:splitting} can be applied in this case. 
By Lemma \ref{lem:ellsplit} $R^2\pi_{*}\cP_\beta\cong R^1\psi_{*}\cP_d(M)$, which is  a torsion sheaf by the assumption on $d$. There is an exact sequence
\[\begin{tikzcd}
	0 & K & {G_0} & {G_1} & R^2\pi_{*}\cP_\beta & 0
	\arrow[from=1-1, to=1-2]
	\arrow[from=1-2, to=1-3]
	\arrow[from=1-3, to=1-4]
	\arrow[from=1-4, to=1-5]
	\arrow[from=1-5, to=1-6]
\end{tikzcd}\]
Letting $\cP_m=\cP_d^*(K_C-M)$ and dualizing, we get the short exact sequence
\[\begin{tikzcd}
	0 & {G_1^*} & {G_0^*} & R^1\psi_*\cP_m & 0
	\arrow[from=1-1, to=1-2]
	\arrow["f",from=1-2, to=1-3]
	\arrow[from=1-3, to=1-4]
	\arrow[from=1-4, to=1-5]
\end{tikzcd}\]
 $\cP_m$ is a Poincar\'e line bundle on $C\times\Pic_m(C)$ with $m=\chi(\cO_S)+2g-2-d$. Since $R\psi_*\cP_{d}(M)=\sff G$, and by Riemann-Roch $\rk(R^1\psi_*\cP_m)=d-p_g$. Serre duality identifies the fibers of $R^1\psi_*\cP_m$ with duals of the fibers of $R^0\psi_*\cP_{d}(M)$, so the space of the linear systems of degree $d-\chi(\cO_S)$ and dimension $d-p_g-1$ for $C$ can be expressed as follows
 \begin{align*}
     G^{d-p_g-1}_{d-\chi(\cO_S)}(C)&=\Gr(R^1\psi_*\cP_m,d-p_g)\\
     &=\{(L,V):L\in\Pic_{d-\chi(\cO_S)}(C),\, V\subset h^0(L),\,\dim(V)=d-p_g\}.
 \end{align*}
 For this rank and degree, the Brill-Noether number is 
 $$\rho=g-(d-p_g)(g-(d-\chi(\cO_S))+(d-p_g-1))=g-(d-p_g)(0)=g.$$
 By the result of Gieseker (cf. \cite[Chapter 5]{arbarello1}), since we work over a general curve $C$, $\Gr(R^1\psi_*\cP_m,d-p_g)$ is smooth of dimension $g$. As a result, all the required  conditions of Section \ref{sec:splitting} are satisfied, and an application of Theorem \ref{thm:refinedforsplit} gives the following proposition.

\begin{prop} \label{prop66}
    For $\iota:S_\beta\hookrightarrow\prj(B)$ as in Section \ref{sec:embvia} and any $0\leq k\le p_g$, we have
    $$\iota_*\llhb{S_\beta}_k=c_{b-1+g-d}\big(B(1)-R\psi_{*}\cP_d(1)\big)c_{d-p_g+k}\big(R\psi_{*}\cP_d(M)(1)\big).$$
\end{prop}

Pushing down along $q:\prj(B)\to\Pic_\beta(S)$, we get the following theorem.
\begin{theorem} \label{thm:67} We have
$q_*\iota_*\llhb{S_\beta}=0$, and 
    for $ 0< k\le p_g$,
 \begin{align*}
q_*\iota_*\llhb{S_\beta}_k=(-1)^{d-p_g+k}\frac{\theta^{g-p_g+k}}{(g-p_g+k)!}{g-1-p_g+2k\choose d-p_g+k}.
    \end{align*}

 
\end{theorem}

\begin{proof} We apply $q_*$ to the identity in Proposition \ref{prop66}. For the first identity, apply Lemma \ref{lem:pushpicn0} and use \ref{chernpic} to get
\begin{align*} 
q_*\iota_*\llhb{S_\beta}&=\sum_{j=0}^{d-p_g}c_{g-d+j}(-\sff{A})c_{d-p_g-j}(\sff{G})\\
&=\sum_{j=0}^{d-p_g}\frac{\theta^{g-d+j}}{(g-d+j)!}\frac{(-\theta)^{d-p_g-j}}{(d-p_g-j)!}\\
        &=\frac{(-1)^{d-p_g}\theta^{g-p_g}}{(g-p_g)!}\sum_{j=0}^{d-p_g}(-1)^{j}{g-p_g\choose d-p_g-j}\\
        &=\frac{\theta^{g-p_g}}{(g-p_g)!}\sum_{j=0}^{g-p_g}(-1)^{j}{g-p_g\choose j}=0.
    \end{align*}
The last sum is taken to $g-p_g$ since $d>p_g$ implies $d\geq g$ since $\chi(\cO_S)\geq0$.

For the second identity, applying Lemma \ref{lem:pushpic} and \ref{chernpic}, we get
   \begin{align*} 
q_*\iota_*\llhb{S_\beta}_k&=\sum_{j=0}^{k-1}(-1)^j{k-1\choose j}c_{g-d+j}(-\sff{A})c_{d-p_g+k-j}(\sff{G})\\
        &=(-1)^{d-p_g+k}\frac{\theta^{g-p_g+k}}{(g-p_g+k)!}\sum_{j=0}^{k-1}{k-1\choose j}{g-p_g+k\choose d-p_g+k-j}.
    \end{align*}
    Now the claim follows from Vandermonde's identity.
\end{proof}

\begin{corollary}
    $\displaystyle \deg[ S_\beta]^{\vir}=(-1)^{d}{2g-2+\chi(\cO_S)\choose d}$.
\end{corollary}
\begin{proof}
    In Theorem \ref{thm:67} set $k=p_g$ and use $\deg(\theta^g)=g!$.
\end{proof}

\subsubsection{Case $d>p_g+g-1$}
In this case, $$R^2\pi_{*}\cP_\beta\cong R^1\psi_{*}\cP_d(M)=0,$$ because $d-\chi>2g-2$. There is only one non-vanishing refined cycle, which coincides with the reduced cycle (cf. Proposition \ref{proph2l=0})
$$\llhb{ S_\beta}=[S_\beta]^{\op{red}}\in A_{p_g}(S_\beta).$$

\begin{theorem}
    The reduced class is given by $$[S_\beta]^{\op{red}}=c_{p_g-d}\big(T_{S_\beta}-R\pi_{*}\cP_\beta(1)\big)=\sum_{j=0}^{d-p_g}\frac{(-1)^j\theta^{j}x^{d-p_g-j}}{j!}.$$ 
\end{theorem}

\begin{proof}
As we have seen before (cf. Section \ref{sec:embvia}), $[S_\beta]^{\op{red}}=[{{\op{D}}}_1(\sff{H})]^{\vir}$ for $$\sff{H}:=\{E_0\to E_1\oplus B/\cO(-1)\to E_2\}$$ over $\prj(B)$. As shown in \cite{G1}, this has a perfect obstruction theory given by 
$$\big\{ \sh^0(\sff H)|_{S_\beta}\otimes \sh^1(\sff{H})^*|_{S_\beta}\to\Omega_{\prj(B)}|_{S_\beta}\big\}\to L_{S_\b}.$$
Since  $\sh^{-1}(R\psi_{*}\cP_d^\vee )=0,$ then $\sh^0(\sff E)$ is locally free with $S_\b=\prj(\sh^0(\sff E))$. Therefore $\sh^0(\sff{H})|_{S_\beta}=\cO_{S_\b}(-1)\cap q^*\sh^0(\sff E)=\cO_{S_\b}(-1)$, and 
\begin{align*}
    \sh^1(\sff{H})|_{S_\beta}&=-\sff{H}+\sh^0(\sff{H})|_{S_\beta}\\
    &=B/\cO_{S_\beta}(-1)-R\pi_{*}\cP_\beta+\cO_{S_\beta}(-1)=B-R\pi_{*}\cP_\beta
\end{align*} in $K$-theory.
Since $\prj(B)$ is a projective bundle over the abelian variety $\Pic_\beta(S)$, $T_{\prj(B)}=T_{\prj(B)/\Pic_\beta(S)}=B(1)$. Since $S_\beta\cong C_d$ is smooth of dimension $d$, the reduced class is given by the degree $d-p_g$ Chern class of
\begin{align*}
    \sh^1(\sff{H})\otimes \sh^0(\sff{H})^*-T_{\prj(B)}|_{S_\beta}+T_{S_\beta}&=B(1)-R\pi_{*}\cP_\beta(1)-B(1)+T_{S_\beta}\\
    &=-R\pi_{*}\cP_\beta(1)+T_{S_\beta}
\end{align*}
proving the first equality. By Lemma \ref{chernpic}, $$c\big(R\pi_{*}\cP_\beta(1)\big)=c\big(R\psi\cP_d(1)-R\psi\cP_{d-\chi}(1)\big)=(1+x)^{\chi},$$ and by  \cite[IV.(5.4)]{arbarello1} $c(T_{C_d})=(1+x)^{d+1-g}e^{\frac{-\theta}{1+x}}$. Using Lemma \ref{combchern},  we find
\begin{align*}
    [S_\beta]^{\op{red}}&=\left\{(1+x)^{d-p_g}e^{\frac{-\theta}{1+x}}\right\}_{d-p_g}\\
    &=\sum_{j=0}^{d-p_g}(-1)^j{d-p_g-j\choose d-p_g-j}\frac{\theta^j x^{d-p_g-j}}{j!}\\
    &=\sum_{j=0}^{d-p_g}(-1)^j\frac{\theta^j x^{{d-p_g}-j}}{j!}.
\end{align*}
\end{proof}


\begin{corollary}
    For any $i\geq0$, the pushforward of $x^i\cap[S_\beta]^{\op{red}}$ along $$p\colon S_\b  \to \Pic_\beta(S)$$ is equal to zero, unless $i=p_g-g$ in which case it is $[\Pic_\b(S)].$
\end{corollary}
\begin{proof}
    Since $d>2g-2$,  $C_d\to \Pic_d(C)$ is the projective bundle $\prj(\psi_*\cP_d)$, where $\psi_*\cP_d$ is of rank $d-g+1$. Since $\theta$ is a class pulled back from $\Pic_d(C)$,  the pushforward of $x^{(d-p_g)+i-j}\theta^j$ is zero for $d-p_g+i-j<d-g$, or equivalently for $i-j<p_g-g$. For  $i-j\geq p_g-g$, by Lemma \ref{chernpic}

    \begin{align*}
        p_*\theta^j x^{{d-p_g}+i-j}&=\theta^js_{g-p_g+i-j}\big(\psi_{*}\cP_d\big)\\
        &=\theta^j\{e^{\theta}\}_{\theta^{g-p_g+i-j}}=\frac{\theta^{g-p_g+i}}{(g-p_g+i-j)!},
    \end{align*}
    so that
    \begin{align*}
        p_*\sum_{j=0}^{d-p_g}(-1)^j\frac{\theta^j x^{{d-p_g}+i-j}}{j!}=&\sum_{j=0}^{g-p_g+i}(-1)^j\frac{\theta^{{g-p_g}+i}}{j!(g-p_g+i-j)!}\\
        =&\frac{\theta^{g-p_g+i}}{(g-p_g+i)!}\sum_{j=0}^{g-p_g+i}(-1)^j{g-p_g+i\choose j}.
    \end{align*}
    If $g-p_g+i>0$, the last sum is the expansion of $(1-1)^{g-p_g+i}=0$, and when $g-p_g+i=0$ it is $1$.
\end{proof}

\subsection{Deformation invariance}
Using the results of the last subsection, we  formulate a stronger deformation invariance property for the refined cycles of the Hilbert schemes of divisors on elliptic fibrations (cf. Proposition \ref{prop:defInvmink}).

\begin{prop} \label{prop:defInvell}
Let $\cc S\xr{f} \cc C \to T$ be a smooth family of minimal elliptic surfaces over a nonsingular curve $T$ with $q(\cc S_t)=g(\cc C_t)$ for each $t\in T$. Let $\b$ be a vertical cohomology class of the form \eqref{verticalbeta} for each fiber $\cc S_t$ in which $d$ does not depend on $t$ and all $a_i=0$. In the case that \footnote{The Hodge numbers $q(\cc S_t)$ and $p_g(\cc S_t)$ remain unchanged for the family the surfaces.} $$p_g(\cc S_t) <d\leq p_g(\cc S_t)+q(\cc S_t)-1,$$ we assume each $\cc C_t$  is a general curve. Suppose that $$\min \{h^2(L):L\in\Pic_\beta(\cc S_t)\}$$ remains constant for the family of surfaces.
Then, the classes  $\llhb{\cc S_{t,\beta}}_k$ obey deformation invariance on $T$. 
\end{prop}
\begin{proof}
 Take $\cc X:=\Pic_\beta(\cc S/T)$ and let $\scr P_d$ be a Poincar\'e line bundle on $\cc C\times_T\Pic_d(\cC/T)$.  By the hypothesis that all $a_i=0$ at each fiber, the pullback of line bundles induces a morphism $$\Pic_d(\cc C/T)\xr{F}\Pic_\beta(\cc S/T).$$
 Both relative Picard schemes are smooth over $T$ of relative dimension $g$, and for any $t\in T$, $F_t:\Pic_d(\cc C_t)\to\Pic_\beta(\cc S_t)$ is an isomorphism (cf. Section \ref{sec:61}), so $F$ is an isomorphism. Also, $\scr P_\b:=(f^*\times F)\scr P_d$ is a Poincar\'e line bundle on $\cc S\times_T\Pic_{\b}(\cc S/T)$.  Define $\scr E_\bu:=R\pi_*\scr P_\b$ over $\cc X$ as in the proof of Proposition \ref{prop:defInvmink}.

 For any $p\in \cC$, let $t\in T$ be such that $p\in \cc C_t$, then  $$h^1((\cc S_t)_p,\cc O_{(\cc S_t)_p})=1,$$
since $\cc S_t\to \cc C_t$ is an elliptic fibration. By the Cohomology and Base Change theorem, it follows that $\cc M:=R^1f_*\cO_{\cc S}$ is a line bundle. Moreover, since $f:\cc S\to \cc C$ is a flat morphism between smooth schemes, by \cite[Lemma 0B91]{stacks-p} we know that the formation of $R^1f_*\cO_{\cc S}$ commutes with basechange. In particular, 
    $$(R^1f_*\cO_{\cc S})|_{\cc S_t}\cong R^1f_*\cO_{\cc S_t}.$$

For $\psi:\cc C\times_T\Pic_d(\cc C/T)\to  \Pic_d(\cc C/T)$ the projection, let $\scr A_\bu:=R\psi_*\scr P_d$ and $\scr G_\bu:=R\psi_*\scr P_d(\op{pr}_1^*\cc M)$. The argument in Lemma \ref{lem:ellsplit} generalizes to show
$$R\pi_*\scr P_\b\cong R\psi_*\scr P_d\oplus R\psi_*\scr P_d(\op{pr}_1^*\cc M)[-1].$$

When $d\leq p_g$, using the argument in Lemma \ref{prop:ell1noblow} at every fiber, we find that $\sh^0(R\psi_*\scr P_d(\op{pr}_1^*\cc M))=0$, giving a locally free presentation
$$0\To \scr G_0\To\scr G_1\To\sh^2(R\pi_*\scr P_\b)\To0.$$
Then, $\hd(\sh^2(\scr E_\bu))\leq1$ and we can apply Proposition \ref{prop:hd<=1central}. 

When $d>p_g+g-1$, the claim is about the deformation invariance property of the reduced cycle $[\cc S_{0,\b}]^{\op{red}}$ (all the other refined cycles vanish), which is known. 

When $p_g<d\leq p_g+g-1$, since 
$$\scr G_0\xr{\s} \scr G_1\to \sh^2(\scr E_\bu)\to 0$$
is a presentation, we can find an $\scr E_\bu$-suitable blowup $\nu:\cc Y\to \cc X$ with $\cc Y\subset\Gr(\cc B,b)$, for the sheaf $\cc B=\coker(\s^*)$ of rank $b$. Since $\cc C_0$ is a general curve, the discussion before Proposition \ref{prop66} shows that $\Gr(\cc B|_{\cc X_0},b)$ is irreducible, allowing us to use Corollary \ref{cor:irredDefInv} to finish the proof.
\end{proof}

\subsection{Duality} In this section, we prove a stronger version of duality that involves all the refined classes. 
We relate the cycles $\llhb{S_\beta}_k$ and $\llhb{S_{\hat{\beta}}}_{\hat k}$ for  $\hat{\beta}=K_S-\b$, for $\hat k$ chosen so that the two classes have the same dimension. Given $d=\deg(D)$, let $\hat d:=2g-2+\chi(\cO_S)-d$ be the degree of an effective divisor $\hat{D}$ on $C$ corresponding to $\hat{\beta}$. We assume that $0\leq d\leq p_g+g-1$ and $0\leq k\leq p_g-\rkh{\sff E}^2$. Then $0\leq \hat d\leq p_g+g-1$, and $0\leq \hat k\leq p_g-\rkh{\sff E}^0$.

When $g-1\leq d\leq p_g$, we allow $C$ to be any smooth curve. Otherwise, we work  under the assumption that the curve $C$ is general, since in this case either $d$ or $\hat d$ are in the range $(p_g,p_g+g-1]$.

\begin{theorem} \label{thm:dualityell}
    For $\b, \hat{\b}$ and $k, \hat k$ as above, 
    $$q_*\iota_*\llhb{ S_{\beta}}_{ k}=(-1)^{\chi(\cO_S)}q_*\iota_*\llb{ S_{\hat{\beta}}}_{\hat k}.$$
\end{theorem}
\begin{proof}
    We assume that $0\leq d\leq p_g$, since otherwise $0\leq \hat d\leq p_g$ and the argument is symmetric on exchanging $\beta\leftrightarrow\hat\beta$.
    
    If $d<g-1$ then $\hat d> p_g$. In this case, $\hat k= p_g-d+k$ and as in the proof of the second identity of Theorem \ref{thm:67}, we get 
    \begin{align*}
        q_*\iota_*\llhb{ S_{\hat{\beta}}}_{\hat k}&=\sum_{j=0}^{\hat k}(-1)^j{\hat k-1\choose j}c_{g-\bar d+j}(\sff{G}^\vee)c_{\hat d-p_g+\hat k-j}(-\sff{A}^\vee)\\
        &=\sum_{j=0}^{p_g- d+k-1}(-1)^j{p_g-d+k-1\choose j}c_{d-p_g+1+j}(\sff{G}^\vee)c_{\bar d-d+k-j}(-\sff{A}^\vee)\\
        &=\sum_{j=0}^{\hat d-d+k}(-1)^j{p_g-d+k-1\choose j}c_{d-p_g+1+j}(\sff{G}^\vee)c_{\bar d-d+k-j}(-\sff{A}^\vee).
        \end{align*}
    in the last equality we use that $\hat d-d+k=g+(p_g-d+k-1)-d\leq p_g-d+k-1$ to cut the sum at this term.

    If $g-1\leq d\leq p_g$ then $\hat d\leq p_g$. In this case,  $\hat k=\hat d-d+k$ and as in the proof of Corollary \ref{cor:pushell1} we have
    \begin{align*}
        q_*\iota_*\llhb{ S_{\hat{\beta}}}_{\hat k}&=\sum_{j=0}^{\hat k}(-1)^j{p_g-\hat d+\hat k-1\choose j}c_{g-\bar d+j}(\sff{G}^\vee) c_{\hat k-j}(-\sff{A}^\vee)\\
        &=\sum_{j=0}^{\hat d-d+ k}(-1)^j{p_g- d+ k-1\choose j}c_{d-p_g+1+j}(\sff{G}^\vee)c_{ \bar d-d+k-j}(-\sff{A}^\vee).
        \end{align*}
    Thus in either case, we get the same expression. Changing the variable $j\leftrightarrow (p_g-d+k-1)-j$ and using $\hat d-d+k=g+(p_g-d+k-1)-d$, we see that the last sum will now start from $\op{max}\{d-g,0\}$:
        \begin{align*}
        &=(-1)^{p_g-d+k-1}\sum_{j=\op{max}\{d-g,0\}}^{p_g- d+k-1}(-1)^j{p_g-d+k-1\choose j}c_{k-j}(\sff{G}^\vee)c_{g-d+j}(-\sff{A}^\vee)\\
        &=(-1)^{p_g-d+k-1}\sum_{j=0}^{p_g- d+k-1}(-1)^j{p_g-d+k-1\choose j}c_{k-j}(\sff{G}^\vee)c_{g-d+j}(-\sff{A}^\vee)\\
        &=(-1)^{p_g+g-1}\sum_{j=0}^{p_g- d+k-1}(-1)^j{p_g-d+k-1\choose j}c_{k-j}(\sff{G})c_{g-d+j}(-\sff{A})\\
        &=(-1)^{\chi(\cO(S))}q_*\iota_*\llb{ S_\beta}_{k}.
    \end{align*}
    For the first equality, when $\op{max}\{d-g,0\}=d-g$, we can start the sum at $0$ as  $c_{g-d+j}(\sff A^\vee)=0$ for $0\leq j< d-g$. This together with $d\leq p_g$ give the first equality. 
\end{proof}

\section{Stable pairs}
\label{sec:curvecounting}

In this section, we follow the work of Kool-Thomas in \cite{Kool_2014, Kool_2014_2}, and using our theory, we relax the condition in their work required for the existence of the reduced cycle on the moduli space of stable pairs. As before, let $S$ be a nonsingular complex projective surface. For simplicity we denote the nested Hilbert scheme $S^{[0,n]}_\b$ by $S^{[n]}_\b$.

\subsection{Setup}
 There is a natural isomorphism of schemes with (reduced) perfect obstruction theories (cf. \cite{gholampour2020nested, G1,G2})
 $$P_{1-h+n}(S,\beta) \cong S^{[n]}_\beta, $$ where the left hand side is the moduli space of stable pairs on $S$, and $h$ is the arithmetic genus of curves in class $\beta$, determined by adjunction $2h-2=\beta^2+\beta\cdot K_S.$
Let $\cc D_\b \subset S\times S^{[n]}_\b$ be the universal divisor, and $\cc D^{[n]}_\b$ be the associated rank $n$ tautological bundle over $S^{[n]}_\b$. Over $S\times S^{[n]}_\b$, we have $\cO(-\cD_\b)\subset \cI$. Then, for the projections $S \xleftarrow{p} S\times S_\beta^{[n]}\xr{\pi}  S_\beta^{[n]}$, \beq{def:sffT}\sff T:=R\sHom_\pi(\cI/\cO(-\cD_\b),\cI)[1]\in \Perf(S^{[n]}_\b, 0,1)\eeq is identified with the virtual tangent bundle of the moduli space of stable pairs.
 For any $\sigma_i\in H^*(S,\Z)$, let $\tau(\sigma_i):=\pi_*(c_1(\cD_\beta)\cap p^*(\sigma_i))$. 
 Define the intersection numbers
$$\llhb {S_\beta^{[n]}(\sigma_1,\dots,\sigma_\ell)}:=\int_{\llhb{S_\beta^{[n]}}}\prod^\ell_{i=1}\tau(\sigma_i)\quad\in \Z,$$
where it is defined to be zero if $\dim \prod^\ell_{i=1}\tau(\sigma_i)\neq 2(n+\vd_\b+p_g)$.

Given an integral basis of $\bar \ga=\{\gamma_1,\dots,\ga_{b_1(S)}\}$ of  $H_1(S,\Z)/$torsion, such that  $\int_{\Pic_\beta(S)}\tilde\gamma_1\wedge\dots\wedge\tilde\gamma_{b_1}=1$, we consider the intersection numbers of the form
$$\llhb{S_\beta^{[n]}(m,\bar\gamma;\bar\sigma)}:=\llhb{S_\beta^{[n]}([\pt]^{m}[\gamma_1]\cdots[\gamma_{b_1(S)}]\,\sigma_1\cdots\sigma_\ell)},$$
where  $\bar\sigma=(\sigma_1,\cdots,\sigma_\ell)$, $[\pt]$  is the class of a point on $S$ and $[\ga_i]$ is the Poincar\'e dual of $\ga_i$. As in \cite[Equation (57)]{Kool_2014}, the insertion of the basis $\bar \ga$ corresponds to the Gysin pullback $j^!$ for $j:\{L\}\hookrightarrow \Pic_\beta(S)$  the inclusion of a line bundle, and the term $\tau([\pt]^m)$ further cuts down by $H^m$ for $H$ the hyperplane section. Then
$$\llhb {S_\beta^{[n]}(m,\bar\gamma;\bar\sigma)}=\int_{S^{[n]}_{|L|}}j^!\llhb{S_\beta^{[n]}}\cap H^m \prod^\ell_{i=1}\tau(\sigma_i)$$
in which $S^{[n]}_{|L|}\subset S^{[n]}_\b$ is the locus where the divisorial part of the stable pair belongs to $|L|$ (cf. \eqref{equ:SnPa}).
 The following diagram relates the main spaces of interest in this section to each other. Here, $L$ is a fixed bundle in class $\beta$ and $\gamma=\beta+c_1(\cO(A))$ for a sufficiently positive divisor $A$, such that $h^{i\geq1}(M(A))=0$ for any $M\in\Pic_\beta(S)$. The Hilbert scheme $S_\gamma$ plays the role of $\prj(B)$ from Section \ref{sec:setupB}.
    \begin{equation}\label{fig:maindiagKT}
       \begin{tikzcd}[cramped]
	{S_{|L|}^{[n]}} & {S^{[n]}_\beta} & \\
	{S^{[n]}\times|L(A)|} & {S^{[n]}\times S_\gamma} & {S^{[n]}\times\prj(B)} \\
	{\{L(A)\}} & {\Pic_\gamma(S)} \\
	{\{L\}} & {\Pic_\beta(S)}
	\arrow[hook, from=1-1, to=1-2]
	\arrow["{{\iota_{L}}}"', hook, from=1-1, to=2-1]
	\arrow["\iota", hook, from=1-2, to=2-2]
	\arrow["\jmath", hook, from=2-1, to=2-2]
	\arrow[from=2-1, to=3-1]
	\arrow[equals, from=2-2, to=2-3]
	\arrow[from=2-2, to=3-2]
	\arrow[hook, from=3-1, to=3-2]
	\arrow[equals, from=3-1, to=4-1]
	\arrow[equals, from=3-2, to=4-2]
	\arrow["j", hook, from=4-1, to=4-2]
\end{tikzcd}
    \end{equation}
Defining the insertions $\tau(-)$ on $S^{[n]}\times S_\ga$ as 
$$\tau(\s_i):=\pi_*(c_1(\cD_\b)\cap (\id_{S^{[n]}},\iota)_*p^*(\s_i))\in A_*(S^{[n]}\times S_\ga),$$ we can push forward by the inclusion $\iota_{L}$ to get 
\beq{eq:insertions}\llhb{S_\beta^{[n]}(m,\bar\gamma,\bar\s)}=\int_{S^{[n]}\times|L(A)|}\iota{_L}_* j^!\llhb{S_\beta^{[n]}}\cap  H^m\prod_{i=1}^\ell\tau(\s_i).\eeq
Note that the class ${\iota_L}_*j^!\llhb{S_\beta^{[n]}}$ lives in ${A}_r(S^{[n]}\times|L(A)|)$ with
$$r=(p_g-h^2(\beta)+\vd_\beta+n)-g=\chi(L)-h^2(\b)-1+n,$$
where $h^2(\b):=\rk (R^2\pi_*\cP_\b)$, and as before, $p_g=h^{0,2}(S)$ and $g=h^{0,1}(S)$. The intersection number in \eqref{eq:insertions} vanishes if $\dim \prod^\ell_{i=1}\tau(\sigma_i)\neq 2(r-m)$.
\subsection{Evaluation} In this section, we study the intersection numbers of the form $\llhb{S_\beta^{[n]}(m,\bar\gamma,\bar\s)}$. There is a factorization (cf. \eqref{fig:contentionshilb}) of $\iota$  as $$\iota\colon S_\beta^{[n]}\overset{i}\hookrightarrow S^{[n]}\times S_\beta\overset{\varphi}\hookrightarrow S^{[n]}\times \prj(B).$$ 
For $\sff E:=R\pi_*\cI(\cP_\beta)$ and $q:\prj(B)\to\Pic_\beta(S)$ the projection, there is a map $\sff E\to B/\cO(-1)$ and we take $\sff H$ to be its cone. By \eqref{eq:hilbdeepest1},
$$\llhb{S_\beta^{[n]}}=\llb{\sff{H},1}.$$

In diagram \eqref{fig:maindiagKT} $j^!=\jmath^*$, and  from the functoriality of Gysin pullback    \begin{equation*}
    {\iota_L}_* j^!\llhb{S_\beta ^{[n]}}=j^!\iota_*\llhb{S_\beta ^{[n]}}=\jmath^*\varphi_*i_*\llhb{S_\beta^{[n]}}.\end{equation*}
    Using Theorem \ref{thm:compn1n2}, we see that
    \begin{align*}
i_*(\llhb{S_\beta^{[n]}})&=c_n\big({R}\pi_*\cO(\cD_\beta)-{R}\pi_*\cI(\cD_\beta)\big)\cap[S^{[n]}]\times\llhb{S_\beta}\\
&=c_n\big(\cD_\beta^{[n]}\big)\cap [S^{[n]}]\times\llhb{S_\beta}.
   \end{align*}
    For the inclusion $\bar\jmath:|L(A)|\hookrightarrow S_\gamma$, so that $\jmath=\op{id}\times\bar\jmath$, and for $\varphi=\op{id}\times\bar\varphi$, 
    \begin{equation}\label{eq:KTpushpull}
     {\iota_L}_* j^!\llhb{S_\beta ^{[n]}}=c_n\big(L^{[n]}(1)\big)\cap([S^{[n]}]\times\bar\jmath^*\bar\varphi_*\llhb{S_\beta}). 
    \end{equation}
    On restriction to the linear system $|L(A)|\subset S_\gamma$, $\cD_\beta\cong L\boxtimes\cO(1)$ and thus $\cD_\beta^{[n]}\cong L^{[n]}(1)$. Also, $\bar\jmath^*$ is the intersection product with $|L(A)|$ in $S_\gamma$, and the class $H$ on $S^{[n]}\times |L(A)|$ only acts on the second component. In the next theorem, we give an explicit evaluation of
 $\llhb{S_\beta^{[n]}(m,\bar\gamma,\bar\s)}$.


\begin{theorem}\label{thm:pointinsert}
    In the notation of Diagram \eqref{fig:maindiagKT},
$$\llhb{S_\beta^{[n]}(m,\bar\gamma,\bar\s)}=\int_{S^{[n]}\times|L(A)|}c_n\big(L^{[n]}(1)\big)H^{h^0(L(A))-\chi(L)+h^2(\beta)+m}\prod_{i=1}^\ell\tau(\s_i).$$
\end{theorem}
\begin{proof}
    The proof is mainly obtained by further unraveling the class $\iota_{L*}j^!\llhb{S_\beta^{[n]}}$ in the integrand of \eqref{eq:insertions} through Equation \eqref{eq:KTpushpull}. Let $U\subset \Pic_\beta(S)$ be the open set where $R^2\pi_*\cP_\beta$ is locally free. Over $U$, $\rk(R^2\pi_*\cP_\beta)=h^2(\beta)$. Denoting the restriction of $q$ by $q_U:\prj(B|_U)\to U$, we have that $\sh^2(\sff{H}(1))|_{\prj( B|_U)}=q_U^*(R^2\pi_*\cP_\beta)(1)$ is locally free. The $\sff H$-suitable blow up $\nu:Y\to \prj(B)$ of Lemma \ref{lem:blhd1}  is an isomorphism over $\prj(B|_U)$. 
    Let $Y_U:=\nu^{-1}(\prj(B|_U))$ and $\nu_{U}:=\nu|_{Y_U}$, and for the open subsets $\prj(B|_U)$ and $Y_U$, denote their open embeddings by $\psi_{\prj(B|_U)}$ and $\psi_{Y_U}$ respectively. 
    
    For $\eta:=b-\chi(L)+h^2(\beta)$. By the Thom-Porteous formula \eqref{eq:TP-Gr(B)}
    \begin{align*}
            \bar\varphi_*\llhb{S_\beta}& =\nu_* c_{\eta}\big(-\nu^*\sff{H}(1)+\sh^2(\nu^*\sff{H}(1))\big) \\
            &=\sum_{i=0}^{\eta}c_{\eta-i}\big(-\sff{H}(1)\big)\nu_* c_i\big(\sh^2(\nu^*\sff{H}(1))\big). 
    \end{align*}
    Recall that $\bar\jmath$ is the inclusion $\bar\jmath:|L(A)|\hookrightarrow \prj(B)$, so that in $K$-theory:
    \begin{align*}
        -\sff{H}(1)|_{|L(A)|}&=-{R}\pi_*\cP_\beta(1)|_{|L(A)|}+q^*B(1)|_{|L(A)|}\\
        &=-\cO(1)^{\oplus\chi(L)}+\cO(1)^{\oplus b},
    \end{align*}
    and $\bar\jmath^*$ is the intersection product by $[|L(A)|]=q^*[\pt]$ in $\prj(B)$. Therefore,
    \begin{align}\bar\jmath^*\bar\varphi_*\llhb{S_\beta}& =\sum_{i=0}^{\eta}\bar\jmath^*c_{\eta-i}\big(-\sff{H}(1)\big)\nu_*c_i\big(\sh^2(\nu^*\sff{H}(1))\big)\nonumber\\
            &=\sum_{i=0}^{\eta}c_{\eta-i}\big(\cO(1)^{b-\chi(L)}\big)\bar\jmath^*\nu_*c_i\big(\sh^2(\nu^*\sff{H}(1))\big)\nonumber\\
            &=\sum_{i=0}^{\eta}c_{\eta-i}\big(\cO(1)^{b-\chi(L)}\big)\nu_*c_i\big(\sh^2(\nu^*\sff{H}(1))\big) \cap q^*([\pt]).\label{eq:kt1}
    \end{align}
    By Leray-Hirsch theorem 
    $$A_{k}(\prj(B))\cong \bigoplus_{j=0}^{b-1} H^{j}A_{k+j-b+1}(\Pic_\beta(S)),$$
    so each term $\nu_*c_i\big(\sh^2(\nu^*\sff{H}(1))\big)\in A_{b-1+g-i}(\prj(B))$ can be written in the form
    \begin{equation}\label{eq:decomppic}
\nu_*c_i\big (\sh^2(\nu^*\sff{H}(1))\big)=\sum_{j=0}^{b-1}H^jq^*(\a_{i,j})\end{equation}
    for some $\a_{i,j}\in A_{g-i+j}(\Pic_\beta(S))$. Thus, for any $i$,
    \begin{equation*}
    \nu_*c_i\big(\sh^2(\nu^*\sff{H}(1))\big) \cap q^*([\pt]\big)= \sum_{j=0}^{b-1}H^jq^*(\a_{i,j}\cdot[\pt])=H^{i}q^*(\alpha_{i,i}\cdot[\pt]).
    \end{equation*}
    The last equality holds because $\alpha_{i,j}\cdot[\pt]=0$, unless $\alpha_{i,j}\in A_g(\Pic_\beta(S))$ meaning that $j=i$. Let $\lambda_i\in\Z_{\geq0}$, such that $\alpha_{i,i}=\lambda_i[\Pic_\beta(S)]$. Now, $\bar\jmath^*\nu_*c_i\big(\sh^2(\nu^*\sff{H}(1))\big)=\lambda_iH^iq^*([\pt])$. From \eqref{eq:kt1}, since $q^*[\op{pt}]=[|L(A)|]$,
    \begin{align*}       \bar\jmath^*\bar\varphi_*\llhb{S_\beta}
        &=\sum_{i=0}^{b-\chi(L)+h^2(\beta)}\lambda_i c_{b-\chi(L)+h^2(\beta)-i}\big(\cO(1)^{b-\chi(L)}\big)H^i\in A_*(|L(A)|).
    \end{align*} 
   Restricting to $\prj(B|_U)$, we see that \begin{align*}
       \psi^*_{\prj(B|_U)}\nu_*c_i\big(\sh^2(\nu^*\sff{H}(1))\big)&=\nu_{U *}\psi_{Y_U}^* c_i\big(\sh^2(\nu^*\sff{H}(1))\big)\\
      &=\nu_{U *} c_i\big(\sh^2(\nu^*\sff{H}(1))|_{Y_U}\big)=c_i\big(\sh^2(\sff{H}(1))|_{\prj(B|_U)}\big),
   \end{align*}
    where the last equality holds because $\nu_U$ is an isomorphism. $\sh^2(\sff{E})|_{U}$ is locally free, and $\sh^2(\sff{H})|_{\prj(B|_U)}=q_U^*(\sh^2(\sff E)|_U)$  so for $i\geq 0$,
    $$\nu_*c_i\big(\sh^2(\nu^*\sff{H}(1))\big)|_{\prj(B|_U)}=\sum_{j=0}^i{h^2(\beta)-j\choose i-j}H^{i-j}q^*_U(c_j\big(\sh^2(\sff E)|_U)).$$
    The coefficient of $H^i$ in this identity matches $\lambda_i$ from $\alpha_{i,i}=\lambda_i[\Pic_\beta(S)]$ in  \eqref{eq:decomppic} because the restriction to $\prj(B|_U)$ only cancels cycles supported on $\prj(B|_{Z})$, where $Z=\Pic_\beta(S)\setminus U$. As a result, $\lambda_i={h^2(\beta)\choose i}$ and
    \begin{align*}
        \bar\jmath^*\bar\varphi_*\llhb{S_\beta}     &=\sum_{i=0}^{b-\chi(L)+h^2(\beta)}\lambda_i c_{b-\chi(L)+h^2(\beta)-i}(\cO(1)^{b-\chi(L)})H^i\\
         &=H^{b-\chi(L)+h^2(\beta)}\sum_{i=0}^{b-\chi(L)+h^2(\beta)}{h^2(\beta)\choose i}{b-\chi(L)\choose i-h^2(\beta)} \\
        &=H^{b-\chi(L)+h^2(\beta)}.
    \end{align*} 
    Using Equation \eqref{eq:KTpushpull}, this proves that
    \begin{equation}\label{eq:KToldthm71}
        \iota_{L*}j^!\llhb{S_\beta^{[n]}}=c_n\big(L^{[n]}(1)\big)H^{\chi(L(A))-\chi(L)+h^2(\beta)}.
    \end{equation}
    Substituting into \eqref{eq:insertions} gives
    $$\llhb{S_\beta^{[n]}(m,\bar\gamma,\s)}=\int_{S^{[n]}\times|L(A)|}c_n\big(L^{[n]}(1)\big)H^{\chi(L(A))-\chi(L)+h^2(\beta)+m}\prod_{i=1}^\ell\tau(\s_i).$$
\end{proof}

\begin{corollary}\label{cor:KTinv1}
Let $\cc S\to C$ be a smooth family of projective surfaces over a nonsingular curve $C$, and suppose that the conditions (i), (ii) of Proposition \ref{prop:defInvmink} are satisfied for a cohomology class $\b$ and a nonnegative integer $l$.
Let $\s_1,\dots ,\s_\ell$ be some cohomology classes in the fibers of the family. 
Then,
    $\llhb{\cc S^{[n]}_{p,\b}(m,\bar\gamma,\bar\s)}$ is independent of $p\in C$.  
    \end{corollary}
    \begin{proof}
        By our assumption $h^2(\beta)=l$ for every fiber $\cc S_p$. The power of $H$ in the expression obtained in Theorem \ref{thm:pointinsert} is therefore independent of $t$. Since the right hand side of the formula in that theorem is  deformation invariant (by the inductive argument in \cite{EGL}),  the claim is proven. 
    \end{proof}



\subsection{Curve counting} \label{sec:curvecounting}
Analogous  to definition $(45)$ in \cite{Kool_2014},
define the residue invariants\footnote{We call these intersection numbers the invariants of $S$ because of Corollary \ref{cor:KTinv2}. We call them the residue invariants, because in the setting of \cite{Kool_2014} they coincide with the reduced residue stable pair invariants of the total space of the canonical bundle of $S$ in which case $\sff T^\vee[-1]\otimes \fr t$ is identified with the virtual normal bundle of the fixed locus of the $\bb C^*$-action with the weight 1 representation $\fr t$ induced by the fiberwise action on the canonical bundle.} of $S$ by 
    $$\llhb{S^{[n]}_\b(m,\bar\gamma)}_{\op{res}}:=\int_{\llhb{S_\beta^{[n]}}}c(\sff T)\big(\prod^{b_1(S)}_{i=1}\tau(\gamma_i)\big)\tau([\pt])^{m},$$
  In $K$-theory, 
    $\sff T=T_{S^{[n]}}-\cO(\cc D_\b)^{[n]}+{R}\pi_*\cO(\cD_\beta)-R\pi_*\cO$ (cf. \eqref{def:sffT}).

We now prove a strengthening  of \cite[Theorem 1.1]{Kool_2014_2}. In particular, this theorem replaces their assumption that $h^2(L)=0$ for all $L\in \Pic_\b(S)$ by the assumption $h^2(L)=0$ for all \emph{effective}  $L\in \Pic_\b(S)$ (the weaker assumption is needed to ensure the reduced perfect obstruction theory exists.). 

\begin{theorem}\label{thm:KTpolynom}
    For any $n$, $m$ and $\b$. The residue invariant $\llhb{S^{[n]}_\b(m,\bar\gamma)}_{\op{res}}$ is a universal polynomial in the topological numbers $\beta^2,\,c_1(S)\cdot\beta ,\,c_1(S)^2,\,c_2(S)$.
\end{theorem}
\begin{proof}

We refer to Diagram \eqref{fig:maindiagKT} again. Recall that $\iota$ factors as $S_\beta^{[n]}\overset{i}\hookrightarrow S^{[n]}\times S_\beta\overset{\varphi}\hookrightarrow S^{[n]}\times \prj(B)$. Similar simplifications as in \cite[Section 3]{Kool_2014_2} remain valid resulting in
    \begin{align*}
        \llhb{S^{[n]}_\b(m,\bar\gamma)}_{\op{res}}&= \int_{\llhb{S_\beta^{[n]}}}\frac{c(T_{S^{[n]}})\,c({R}\pi_*\cP_\beta(1))}{c(\cO(\cc D_\b)^{[n]})}\big(\prod^{b_1(S)}_{i=1}\tau(\gamma_i)\big)\,\tau([\pt])^{m}\\
&=\int_{S^{[n]}_{|L|}}\frac{c(T_{S^{[n]}})\,c(\cO(1)^{\oplus\chi(L)})}{c(L^{[n]}(1))}H^m j^!\llhb{S_\beta^{[n]}} \\
&=\int_{S^{[n]}\times|L(A)|} H^m \frac{c(T_{S^{[n]}})\,c(\cO(1)^{\oplus\chi(L)})}{c(L^{[n]}(1))} \iota_{L*}j^!\llhb{S_\beta^{[n]}}, 
    \end{align*} where we used $\cO(\cD_\beta)^{[n]}|_{S^{[n]}\times|L(A)|}\cong L^{[n]}(1)$.
Let $\ell:=\chi(L)-h^2(\beta)-1$ and $r:=\chi(L(A))-\chi(L)$. Equation \eqref{eq:KToldthm71} gives 
    \begin{align}
        \llhb{S^{[n]}_\b(m,\bar\gamma)}_{\op{res}}&=\int_{S^{[n]}\times|L(A)|} H^{m+r+h^2(\beta)} \frac{c(T_{S^{[n]}})\,c(\cO(1)^{\oplus\chi(L)})}{c(L^{[n]}(1))} c_n(L^{[n]}(1))\nonumber\\
        &=\int_{S^{[n]}\times\prj^{\ell-m}}\frac{c(T_{S^{[n]}})(1+H)^{\chi(L)}}{\sum _{i=0}^{n}(1+H)^ic_{n-i}(L^{[n]})} \sum_{i=0}^nH^ic_{n-i}(L^{[n]})\nonumber\\
        &=\int_{S^{[n]}} \left\{ \frac{c(T_{S^{[n]}})(1+H)^{\chi(L)}}{\sum _{i=0}^{n}(1+H)^ic_{n-i}(L^{[n]})} \sum_{i=0}^nH^ic_{n-i}(L^{[n]})\right\}_{H^{\ell-m}}.\label{eq:refineduniv}
    \end{align} 
    The claim is proven by running the recursion of \cite{EGL} $n$ times starting with the last formula. 
\end{proof}
\begin{corollary}\label{cor:KTinv2}
Let $\cc S\to C$ be a smooth family of projective surfaces over a nonsingular curve $C$, and suppose that the conditions (i) and (ii) of Proposition \ref{prop:defInvmink} are satisfied for a cohomology class $\b$ and a nonnegative integer $l$.
 Then,
    $\llhb{\cc S^{[n]}_{p,\b}(m,\bar\gamma)}_{\op{res}}$ is independent of $p\in C$.  
    \end{corollary}
    \begin{proof}
        By our assumption $h^2(\beta)=l$ for every fiber $\cc S_p$.
        The topological numbers $\beta^2,\,c_1(S)\cdot\beta ,\,c_1(S)^2,\, c_2(S)$ are deformation invariant, so by Theorem \ref{thm:KTpolynom} the residue invariants  $\llhb{\cc S^{[n]}_{p,\b}(m,\bar\gamma)}_{\op{res}}$ are equal for all $p\in C$.
    \end{proof}

    
    In the case when $L$ is $\delta$-very ample for some nonnegative integer $\delta$, for any $0\le n\leq\delta$, we can link this residue invariant to the Euler characteristic of some subscheme of $S^{[n]}_{|L|}$. The scheme $S_{|L|}^{[n]}\subset S^{[n]}\times|L|$ is the zero locus of a section $s$ of $L^{[n]}(1)$. It is smooth  over $S^{[n]}$ with fiber over  $Z\in S^{[n]}$ given by $\prj(\text{ker}(h^0(L)\twoheadrightarrow h^0(L|_Z)),$ which is of codimension $n=\rk(L^{[n]})$ in $\{Z\}\times |L|$ (using that $L$ is $\de$-very ample). The section $s$ is therefore regular and 
    $$c_n(L^{[n]}(1))=[S_{|L|}^{[n]}]\in A_{n+h^0(L)}(S^{[n]}\times|L|).$$
    The following theorem has applications to the counting of $\de$-nodal curves in a general sublinear system $\prj^\de\subset |L|$ (cf. Section \ref{intro:curve}). 

    \begin{theorem}\label{thm:eulerref} Given a class $\beta$, such that $m:=\chi(L)-h^2(\beta)-1-\delta\geq0$, an integer $0\leq n\leq\delta$, a $\delta$-very ample $L\in \Pic_\beta(S)$  and a general chain of  sublinear systems $\bb P^0\subset \bb P^1\subset \cdots  \subset \prj^\delta\subset|L|$ of dimension $\delta$, we have 
$$\llhb{S^{[n]}_\b(m,\bar\gamma)}_{\op{res}}=e(S_{\prj^\delta}^{[n]})+\sum_{j=1}^{\de}a_j e(S_{\prj^{\delta-j}}^{[n]}),$$
where $\displaystyle a_j={\chi(L)-\de-2+j\choose j}$ for $1\le j \le \de$. In particular, they only depend on $\chi(L)$ and $\de$. 
    \end{theorem}
    \begin{proof}
    Intersecting with general hyperplanes in $|L|$, by Bertini's theorem
    \begin{equation*}
    c_n(L^{[n]}(1))H^{h^0(L)-1-\delta}\cap [S^{[n]}\times|L|]=c_n(L^{[n]}(1))\cap [S^{[n]}\times\prj^\delta]=[S_{\prj^\delta}^{[n]}],\end{equation*}
    where $\prj^\delta$ is a general $\delta$-dimensional sublinear system of $|L|$ (cf. \eqref{equ:SnPa}). Therefore, the Euler characteristic of $S_{\prj^\delta}^{[n]}$ can be written as
    \begin{align}
        \int_{S_{\prj^\delta}^{[n]}} c(T_{S_{\prj^\delta}^{[n]}})&=\int_{S^{[n]}\times \prj^\delta}c_n(L^{[n]}(1))\, c(T_{S_{\prj^\delta}^{[n]}})\nonumber\\
        &=\int_{S^{[n]}\times \prj^\delta}c_n(L^{[n]}(1)) \frac{c(T_{S^{[n]}\times\prj^{\delta}})}{c(L^{[n]}(1))}\nonumber\\
        &=\int_{S^{[n]}} \frac{c(T_{S^{[n]}})(1+H)^{\delta+1}}{\sum _{i=0}^{n}(1+H)^ic_{n-i}(L^{[n]})} \sum_{i=0}^nH^ic_{n-i}(L^{[n]}).\label{eq:eulerintegral}
    \end{align}

        For $m$ as in the statement, we have $\chi(L)\geq\delta+1$, so \eqref{eq:refineduniv} becomes
        \begin{align*}
        &\llhb{S^{[n]}_\b(m,\bar\gamma)}_{\op{res}}=\int_{S^{[n]}\times \prj^\de}  \frac{c(T_{S^{[n]}})(1+H)^{\chi(L)}}{\sum _{i=0}^{n}(1+H)^ic_{n-i}(L^{[n]})} \sum_{i=0}^nH^ic_{n-i}(L^{[n]})\\
        &=\int_{S^{[n]}\times\prj^\de} (1+H)^{\chi(L)-\de-1} \frac{c(T_{S^{[n]}})(1+H)^{\de+1}}{\sum _{i=0}^{n}(1+H)^ic_{n-i}(L^{[n]})} \sum_{i=0}^nH^ic_{n-i}(L^{[n]})\\   &=\int_{S^{[n]}\times\prj^\de} (1+H)^{\chi(L)-\de-1}  \frac{c(T_{S^{[n]}\times\prj^{\delta}})}{c(L^{[n]}(1))}c_n(L^{[n]}(1))\\ &=\int_{S^{[n]}_{\prj^{\de}}} (1+H)^{\chi(L)-\de-1}c(T_{S_{\prj^{\delta}}^{[n]}})\\
        &=\sum_{j=0}^{\de}a_j e(S_{\prj^{\delta-j}}^{[n]})
    \end{align*}
     for some integers $a_j$ that can be inductively obtained as follows.  For any $0\le j\le\de$, $H\cdot[S^{[n]}_{\prj^{\delta-j}}]=[S^{[n]}_{\prj^{\delta-j-1}}]$ and we have a short exact sequence $$0\to T_{S^{[n]}_{\prj^{\delta-j-1}}}\to T_{S^{[n]}_{\prj^{\delta-j}}}\to \cO(H)\to 0.$$ 
     Letting $N:=\chi(L)-\de-1$, we find that 
    \begin{align}\notag
        a_j=\sum_{\substack{i_1\geq \cdots \geq i_{k}>0, \\i_1+\cdots+i_k=j}}{N\choose i_1}{i_1\choose i_2}\cdots {i_{k-1}\choose i_k}
        ={N-1+j\choose N-1}.
        \label{eq:recursion-a}
    \end{align}
    For the last simplification, note that the middle sum counts all arrangements of nested subsets of $\{1,\dots,N\}$ with cardinalities adding up to $j$. This is equivalent to choosing numbers $k_1,\dots, k_N\in\Z_{\geq0}$ adding up to $j$ with $k_s$ representing how many of the nested subsets contain the element $s$.   
\end{proof}


    \begin{remark}
         The condition $m\geq0$ in Theorem \ref{thm:eulerref} is equivalent to the condition $$\max\{h^0(M)-h^1(M) : M\in\Pic_\b(S)\}\geq\delta+1.$$ This is satisfied if and only if there exists $L_0\in\Pic_\b(S)$, such that $h^0(L_0)\ge h^1(L_0)+\delta+1$. In \cite{Kool_2014}, Theorem \ref{thm:eulerref} was proven under the stronger assumptions that $h^2(L)=0$ for every effective line bundle $L\in\Pic_\beta(S)$ required for the existence of the reduced class $[S^{[n]}_\beta]^{\op{red}}$, and that $h^1(L)=0$ for  some $\delta$-very ample line bundle $L$. Since for any $\delta$-very ample line bundle $L$ we clearly have $h^0(L)\geq\delta+1$, in the set up of \cite{Kool_2014} the condition $m\ge 0$ is automatically satisfied. 
    \end{remark}
\appendix 

\section{Invariance under quasi-isomorphism} \label{quasiinvariance}
In this appendix, we give another proof for the independence of the classes $\llb{\sff E,r}$ and $\llbb{\sff E,r}$ (cf. Definition \ref{def:1st2nd}) from the quasi-isomorphism type of  $\sff E\in \Perf(X,0,2)$ that does not use the resolution of singularities and perfect obstruction theory. We first prove the independence, when $\sff E$ is a 2-term complex of vector bundles. In the same fashion, we also give an alternative proof for the equality \eqref{eq:HandE} that was the core of the construction  of Section \ref{sec:deepest}.

As before,
let $X$ be a quasi-projective variety of dimension $d$. For a vector bundle $E$ over $X$ with a section $s$, let $\Z(s)$ be the localized top Chern class of $E$ over the zero scheme $Z(s)\subset X$ in the sense of \cite[Proposition 14.1]{fu}. 

\begin{lemma}\label{lem:invar3}
    If $s_1$ and $s_2$ are two sections of a vector bundle $E$ of rank $m$ with $Z(s_1)=Z(s_2)$ then $\bb Z(s_1)=\bb Z(s_2)$.
\end{lemma}
\begin{proof}
    Let $Z:=Z(s_1)=Z(s_2)$ and $Y:=\op{Bl}_{Z}X$. Consider the following commutative diagram with fiber squares:
    \[\begin{tikzcd}[cramped]
	{D:= \bb P(C_Z X)} & Y \\
	Z & {X} \\
	{X} & {E,}
	\arrow["\iota",from=1-1, to=1-2]
	\arrow["\eta", from=1-1, to=2-1]
	\arrow["p"', curve={height=10pt}, from=1-1, to=3-1]
	\arrow[from=1-2, to=2-2]
	\arrow[from=2-1, to=2-2]
	\arrow[from=2-1, to=3-1]
	\arrow["s_{1}",from=2-2, to=3-2]
	\arrow["0_{E}",from=3-1, to=3-2]
\end{tikzcd}\] where $\iota$ is the inclusion of the exceptional divisor of the blow up and $\eta$ is the natural projection. 
By the compatibility of Gysin pullback with pushforward and the excess intersection formula for $\iota$ and $0_{E}$ with excess bundle of rank $m-1$ given by $B_E:=p^*E/\cO_D(D)$, 
$$\Z(s_1)=0_{E}^![\Gr(E^*,r)]=\eta_*0_{E}^![Y]=\eta_*(c_{m-1}(B_E)\cap \iota^*[Y])=\eta_* c_{m-1}(B_E).$$
If $s_1$ is replaced by $s_2$ we arrive at exactly the same formula for $\bb Z(s_2)$, so the lemma is proven.
\end{proof}

\begin{lemma}\label{lem:invar1} Let $E_1$ and $E_2$ be two vector bundles over $X$ with the given sections $s_1$ and $s_2$, and let $s:=(s_1,s_2)$ be the induced section of the direct sum $E_1\oplus E_2$. For $Z_i:=Z(s_i)$, $i=1,2$ and $Z:=Z(s)$, consider the fiber diagram
\begin{equation}\label{diag:gys}
    \begin{tikzcd}[cramped]
	Z & {Z_1} & X \\
	{Z_2} & X & {E_1} \\
	X & {E_2.}
	\arrow[from=1-1, to=1-2]
	\arrow[from=1-1, to=2-1]
	\arrow[from=1-2, to=1-3]
	\arrow[from=1-2, to=2-2]
	\arrow["{0_{E_1}}", from=1-3, to=2-3]
	\arrow[from=2-1, to=2-2]
	\arrow[from=2-1, to=3-1]
	\arrow["{s_1}"', from=2-2, to=2-3]
	\arrow["{s_2}", from=2-2, to=3-2]
	\arrow["{0_{E_2}}"', from=3-1, to=3-2]
\end{tikzcd}
\end{equation}
We have $\Z(s)=0^!_{E_2}\Z(s_1)=0^!_{E_1}\Z(s_2)$.
\end{lemma}
\begin{proof}
    By commutativity of Gysin pullbacks we know that 
    $$0^!_{E_2}\Z(s_1)=0^!_{E_2}0^!_{E_1}[X]=0^!_{E_1}0^!_{E_2}[X]=0^!_{E_1}\Z(s_2).$$
    Next, consider the following fiber diagram
\[\begin{tikzcd}[cramped]
	Z & {Z_1} & X \\
	X & {E_2} & {E_1\oplus E_2} \\
	& X & {E_1,}
	\arrow[from=1-1, to=1-2]
	\arrow[from=1-1, to=2-1]
	\arrow[from=1-2, to=1-3]
	\arrow["{{s_2}_{|Z_1}}", from=1-2, to=2-2]
	\arrow["{(s_1,s_2)}", from=1-3, to=2-3]
	\arrow["{0_{E_2}}", from=2-1, to=2-2]
	\arrow["{\iota_2}", from=2-2, to=2-3]
	\arrow["{q}", from=2-2, to=3-2]
	\arrow["{p_1}", from=2-3, to=3-3]
	\arrow["{0_{E_1}}", from=3-2, to=3-3]
\end{tikzcd}\]
    where $q, p_1$ are the natural projections and $\iota_2=(0,\id)$. By definition, $E_1\oplus E_2= q^*E_1$ with the zero section $\iota_2=q^*0_{E_1}$, so $\iota_2^!=0_{E_1}^!$. Since $\iota_2\circ 0_{E_2}=0_{E_{1}\oplus E_2}$, by the functoriality of Gysin pullback,
    $$\Z(s)=0_{E_{1}\oplus E_2}^![X]=\iota_2^! 0_{E_2}^![X]=0^!_{E_1}0^!_{E_2}[X]=0^!_{E_1}\Z(s_2).$$
\end{proof}

Given a two term complex of vector bundles $E_0\xr{\sigma}E_1$, we denote by $s_\sigma$ the section of the vector bundle $Q_{\Gr}\otimes E_1$ over $q\colon \Gr(E_0^*,r)\to X$ induced by the composition
$$Q_{\Gr}^*\hookrightarrow q^*E_0\overset{\sigma}{\To}q^*E_1.$$
Then, $\wt{\op D}_r(E_\bu)=Z(s_\s)$ and $\llb{E_\bu,r}=\Z(s_\s)$  (cf. Definition \ref{def:1st2nd}).

\begin{lemma}\label{lem:invar2}
    Let $E_0\xr{\sigma} E_1$ and $E_0\oplus G \xr{\s_G} E_1\oplus G$ be two maps of vector bundles bundle over $X$, where $\s_G=(\s,\id_G)$. Denote by $Q$ and  $Q_G$ the universal quotient bundles of $\Gr(E_0^*,r)$ and $\Gr(E_0^*\oplus G^*,r)$. Let  $s_\s$ and $s_{\s_G}=(s_1,s_2)$ be the induced sections of the vector bundles $Q\otimes q^*E_1$ and $Q_G\otimes q^*(E_1\oplus G)$, respectively.  Then, under the identification  $$\Gr(E_0^*,r)\cong Z(s_2)\subset \Gr(E_0^*\oplus G^*,r),$$ we have  $Z(s_\sigma)=Z(s_{\sigma_G})$ and $\Z(s_\sigma)=\Z(s_{\sigma_G})$.
\end{lemma}
\begin{proof}
 Form the fiber diagram \eqref{diag:gys} for the situation at hand:
 \[\begin{tikzcd}[cramped]
	Z(s_{\sigma_G}) & {Z(s_1)} &\Gr(E_0^*\oplus G^*,r) \\
	{\Gr(E_0^*,r)} & \Gr(E_0^*\oplus G^*,r) & {Q_G\otimes q^*E_1} \\
	\Gr(E_0^*\oplus G^*,r) & {Q_G\otimes q^*(E_1\oplus G).}
	\arrow[from=1-1, to=1-2]
	\arrow[from=1-1, to=2-1]
	\arrow[from=1-2, to=1-3]
	\arrow[from=1-2, to=2-2]
	\arrow["{0_1}", from=1-3, to=2-3]
	\arrow[from=2-1, to=2-2]
	\arrow[from=2-1, to=3-1]
	\arrow["{s_1}"', from=2-2, to=2-3]
	\arrow["{s_2}", from=2-2, to=3-2]
	\arrow["{0_2}"', from=3-1, to=3-2]
\end{tikzcd}\]
  By Lemma \ref{lem:invar1}, 
$$\Z(s_{\sigma_G})=0_1^!0_2^![\Gr(E_0^*\oplus G^*,r)]=0_1^!\Z(s_2)=0_1^![\Gr(E_0^*,r)],$$
where the last equality is because $s_2$ is a regular section. Since the restriction of $s_1$ to $\Gr(E_0^*,r)$ is $s_\sigma$, we find that $Z(s_{\sigma_G})=Z(s_\sigma)$ and
$0_1^![\Gr(E_0^*,r)]=\Z(s_\sigma)$, as desired.
\end{proof}

\begin{lemma}\label{lem:invar4}
    Let $E\xr{\s} F$ and $E\xr{\tau}G$ be two quasi-isomorphic complexes of vector bundles then
$$Z(s_\s)=Z(s_\tau),\qquad \Z(s_\sigma)=\Z(s_\tau).$$
\end{lemma}
\begin{proof} $Z(s_\sigma)=\Gr(\coker(\sigma^*),r)=\op D_r(\s)$ and $Z(s_\tau)=\Gr(\coker(\tau^*),r)=\op D_r(\tau)$ (cf.\eqref{eq:secforDr}). Since the complexes are quasi-isomorphic, the two cokernels are isomorphic, and hence  $Z(s_\sigma)=Z(s_\tau)\subset \Gr(E^*,r)$. Moreover, it follows from the quasi-isomorphism that $c(F)=c(G)$, and hence the claim follows from the Chern class formula in the proof of Lemma \ref{lem:invar3}.
\end{proof}


\begin{prop}\label{prop:invariance}
    Let $E_0\xr{\s} E_1$ and $F_0\xr{\tau} F_1$ be two quasi-isomorphic complexes of vector bundles then 
$$Z(s_\s)=Z(s_\tau),\qquad \Z(s_\sigma)=\Z(s_\tau).$$
In particular, for $\sff E\in \Perf(X,0,1)$, the class $\llb{\sff E,r}$ only depends on the quasi-isomorphic type of $\sff E$.
\end{prop}
\begin{proof}
    Form the complexes 
    $$E_0\oplus F_0\xr{(\sigma,\id)} E_1\oplus F_0,\qquad E_0\oplus F_0\xr{(\id,\ga)} E_0\oplus F_1,$$ which are quasi-isomorphic to the given ones. The claim now follows 
     from Lemmas \ref{lem:invar2} and \ref{lem:invar4}.
\end{proof}
We are now ready to prove the first main result of this appendix.
\begin{theorem} Let $\sff E\in\Perf(X,0,2)$, the classes $\llb{\sff E,r}$ and $\llbb{\sff E,r}$ only depend on the quasi-isomorphism type of $\sff E$.
\end{theorem}
\begin{proof}
     Let $\nu:Y\to X$ be an $\sff E$-suitable blow up. By definition
     $$
        \llb{\sff E}=\nu_*\llb{\tau^{\leq1}(\nu^*\sff E)},\qquad \llbb{\sff E}=\nu_*\llb{\sh^2(\nu^*\sff E)^\vee}.
    $$
    We know that (cf. Corollary \ref{cornu}) $\tau^{\leq1}(\nu^*\sff E)$, $\sh^2(\nu^*\sff E)^\vee \in \Perf(Y,0,1)$  are quasi-isomorphism invariants of $\sff E$, so the claim follows from Proposition \ref{prop:invariance}.
    \end{proof}

Our next goal is to give a more direct proof of the equality \ref{eq:HandE} needed for the construction of Section \ref{sec:deepest}. As before, we first prove the result for the 2-term complexes of vector bundles.

\begin{lemma}\label{lem:invar5}
    Given a map $E\xr{\s} F$ of vector bundles, define a 2-term complex of vector bundles over $q\colon \Gr(E^*,r)\to X $ 
    $$ \sff L:=\{q^*E\xr{\ga} q^*F\oplus q^*E/Q^*_\Gr\}$$
    where $Q_{\Gr}$ is the universal quotient bundle of $\Gr(E^*,r)$ and the map $\ga$ is the pair of $q^*\s$ and the canonical projection. Then,
     $$\llb{\sff L,r}=\Z(s_\sigma) .$$
\end{lemma}
\begin{proof}
By construction (cf. Section \ref{sec:setupB}), $\op D_r(\sff L)$ is the deepest degeneracy loci of $\sff L$, and $$\wt {\op D}_r(\s)=Z(s_\sigma)=\op D_r(\sf L)\cong \wt{\op D}_r(\sff L).$$
denote by $\op D_r$ this subscheme of $\Gr(E^*,r)$. Let $e,f$ be the ranks of $E$ and $F$ respectively. 
Letting $K:=\ker(\ga|_{\op D_r(\sff L)})$ and $C:=\coker(\ga|_{\op D_r(\sff L)})$, by \cite[Example 14.4.7]{fu}, we know that
$$\llb{\sff L}=c_{r(f-r)}(K^*\otimes C)\cap s\big(\op D_r(\sff L),\Gr(E^*,r)\big).$$
At a closed point $p:=(x,V)\in \Gr(E^*,r)$, the map $\ga|_p$ is given by 
$$E_x\xr{(\s|_x,\op{pr})}F_x\oplus E_x/V^*,$$
so $\ker(\ga|_p)=V^*\cap\ker(\s|_x)$. If $p\in\op D_r(\sff L)$, then $\ker(\ga|_p)$ is $r$ dimensional, so $V^*=\ker(\s|_x)$, and hence $K|_x=Q^*|_x$. From the definition of $\ga$ it is clear that $K\subset Q^*|_{\op D_r}$, so that $K\subset Q^*|_{\op D_r(\sff L)}$. Therefore,  
$$\llb{\sff L}=\left\{c\big(Q_\Gr\otimes (\sff L+Q^*_\Gr)\big)\cap s\big(\op D_r(\sff L),\Gr(E^*,r)\big)\right\}_{d-r(f-r)}.$$
On the other hand, in $K$-theory $Q_\Gr\otimes (\sff L+Q^*_\Gr)=Q_\Gr\otimes q^*F$, and $s_\sigma$ is a section of $Q_\Gr^*\otimes F$ with $Z(s_\s)=\op D_r(\sff L)$, therefore
$$\Z(s_\s)=\left\{c(Q^*\otimes F)\cap s\big(\op D_r(\sff L),\Gr(E^*,r)\big)\right\}_{d-r(f-r)}.$$ Comparing the last two equalities, we get the result. 
\end{proof}

\begin{prop}\label{prop:invar2}
    Let $\sff E:=\{E_0\xr{\s} E_1\}\in \Perf(X,0,1)$, and $B$ be a vector bundle, whose dual fits into the Diagram  \eqref{fig:defsurjB}. Let $q:\Gr(B^*,r)\to X$, and $\sff H:=\{q^*E_0\xr{\rho} q^*E_1\oplus q^*B/Q^*_{\Gr}\}$ with $\rho=(q^*\s,\psi^*)$. Then,
    $$\llb{\sff H,r}=\llb{\sff E, r}.$$
\end{prop}
\begin{proof}
    By the construction (cf. Section \ref{sec:setupB}), $\op{D}_r(\sff H)$ is a deepest degeneracy loci with $\op{D}_r(\sff H)\cong \wt{\op{D}}_r(\sff E)$.   Since $B^*$ fits into the Diagram \eqref{fig:defsurjB}, 
$$E_0\xr{(\s,\psi^*)} E_1\oplus B$$
is an injective map of bundles. The 2-term complex $$ \sff F:=\{B\xr{\op{pr}\circ (0,\id)} (E_1\oplus B)/E_0\}$$ is quasi-isomorphic to $E_0\xr{\s}E_1$ via the map of complexes
\[\begin{tikzcd}
	{E_0} & {E_1} \\
	B & {(E_1\oplus B)/E_0,}
	\arrow["\s", from=1-1, to=1-2]
	\arrow["{\psi^*}"', from=1-1, to=2-1]
	\arrow["{\op{pr}\circ (\id,0)}", from=1-2, to=2-2]
	\arrow["{\op{pr}\circ(0,\id)}", from=2-1, to=2-2]
\end{tikzcd}\]
 so by Proposition \ref{prop:invariance}, $$\wt{\op{D}}_r(\sff E)\cong  \wt{\op{D}}_r(\sff F),\qquad\llb{\sff E}=\llb{ \sff F}.$$
On the other hand, $\llb{\sff F}=\Z(s_1)$ for a section $s_1$ of $Q_\Gr^*\otimes q^*(E_1\oplus B)/E_0$. Letting $\sff L:=\{q^*B\xr{\ga} q^*(E_1\oplus B)/q^*E_0\oplus q^*B/Q^*_\Gr\}$, where $\ga=(\op{pr}\circ (0,\id),\op{pr})$ ,
by Lemma \ref{lem:invar5},  $$\op{D}_r(\sff L)\cong  \wt{\op{D}}_r(\sff F),\qquad \llb{\sff F}=\llb{\sff L}.$$
Finally, form the following two complexes 
$$q^*E_0\oplus q^*B\xr{\left(\begin{array}{cc}
     \id & 0 \\
     0& \ga  
\end{array}\right)}H_1\oplus q^*B,\qquad q^*E_0\oplus q^*B\xr{\left(\begin{array}{cc}
     \rho & 0 \\
     0& \id  
\end{array}\right)}q^*E_0\oplus L_1,$$ where $H_1$ and $L_1$ are the degree 1 terms of the complexes $\sff H$ and $\sff L$.
Since in $K$-theory
$$H_1\oplus q^*B=q^*B+q^*B+q^*E_1-Q^*_B=q^*E_0\oplus L_1,$$
by Lemma \ref{lem:invar4}$, {\op{D}}_r(\sff H)= {\op{D}}_r(\sff L)$ and 
$\llb{\sff L}=\llb{\sff H}$ completing the proof. 
\end{proof}

The second main result of this section is the following theorem.

\begin{theorem}\label{thm:A8}
        Let $\sff E:=\{E_0\to E_1\to E_2\}\in \Perf(X,0,2)$, $B$ be a vector bundle, whose dual fits into the Diagram\eqref{fig:defsurjB}, $q:\Gr(B^*,r)\to X$ be the natural projection and $u:q^*\sff E\to q^*B/Q^*_{\Gr}$ be the map obtained from \eqref{fig:defsurjB}. For $\sff H:=\Cone(u)$, we have
    $$\llb{\sff H,r}=\llb{\sff E, r}.$$
    \begin{proof}
        Let $\nu:Y\to X$ be an $\sff E$-suitable blow up, and make the fiber diagram
        \[\begin{tikzcd}[cramped]
	{\Gr(\nu^*B^*,r)} & Y \\
	{\Gr(B^*,r)} & X
	\arrow["{\nt q}"', from=1-1, to=1-2]
	\arrow["{\nu}", from=1-1, to=2-1]
	\arrow["\nu"', from=1-2, to=2-2]
	\arrow["q", from=2-1, to=2-2]
\end{tikzcd}\] The induced blow up in the left column is $\sff H$-suitable, and note that 
 $$\llb{\sff E}=\nu_*\llb{\tau^{\leq1}(\nu^*\sff E)} \;\text{ and }\;\llb{\sff H}=\nu_*\llb{\tau^{\leq1}(\nu^*\sff H)}.$$
After pulling back to $Y$, $\nu^*B$ fits in Diagram \eqref{fig:defsurjB} for $\tau^{\leq1}(\nu^*\sff E)$ resulting in the $2$-term complex of vector bundles $\tau^{\leq1}(\nu^*\sff H)$. The result now follows from Proposition \ref{prop:invar2}.
\end{proof}    
\end{theorem}


\section{Combinatorial identities}
\subsection{Pushforward identities}
 Given a projective bundle $q:\prj(B)\to X$ with $b=\op{rk}(B)$ and $X$ a quasi-projective variety, we study the pushforward of classes of the form
$$H^mc_{b+\sff e}\big(q^*B(1)+q^*\sff{E}(1)\big)\,c_{\sff f+n}\big(q^*\sff{F}(1)\big)$$ on $A(\bb P(B))$ 
into $A(X)$. Here, $\cO(1):=\cO_{\prj(B)}(1)$, $H:=c_1(\cO(1))$, $n,m\in\Z_{\geq0}$, and $\sff{E},\sff{F}$ are virtual bundles on  $X$ with $\rk(\sff E)=\sff e$ and $\rk(\sff F)=\sff f$. We omit the pullbacks $q^*$ in this subsection.

\begin{lemma}\label{lem:pushpic}
    For $n>0$,\;$m\geq0$,\; $b+\sff{e}\geq0$, $b>\sff f+n+m$\footnote{The last condition can be removed if $c_{i}(B+\sff E)=0$ for $i> b+\sff e$.}
    \begin{align*}&q_*\left(H^mc_{b+\sff{e}}\big(B(1)+\sff{E}(1)\big) \,c_{\sff{f}+n}\big(\sff{F}(1)\big)\right)=\sum_{j=0}^{n-1}(-1)^j{n-1\choose j}c_{\sff{e}+1+j+m}(\sff{E})\,c_{\sff{f}+n-j}(\sff F).\end{align*}
\end{lemma}
\begin{proof}
    We assume $\sff{f}+n\geq0$ and $b+\sff{e}+\sff{f}+n+m\geq b-1$ since otherwise, both sides are zero. Using the identities for the Chern classes of tensor products in \cite{manivel2010chernclassestensorproducts}, we expand \footnote{Here, we use the standard convention for the negative binomial coefficients, so that for $n\in\Z$, $(x+y)^n=\sum_{k=0}^\infty{n\choose k}x^ky^{n-k}.$} $c_{b+\sff{e}}\big(B(1)+\sff{E}(1)\big)\,c_{\sff{f}+n}\big(\sff{F}(1)\big)$ as
\begin{align*}
    &H^m\sum_{j=0}^{\sff{f}+n}\sum_{i=0}^{b+\sff{e}}H^ic_{b+\sff{e}-i}(B+\sff E){j-n\choose j} H^{j} c_{\sff{f}+n-j}(\sff F)\\
    &\qquad \qquad \qquad =\sum_{j=0}^{\sff{f}+n}{j-n\choose j} c_{\sff{f}+n-j}(\sff F)\sum_{i=0}^{b+\sff{e}}H^{i+j+m}c_{b+\sff{e}-i}(B+\sff E).
\end{align*}
    When pushing forward, we only care about the terms that have a power of $H$ with exponent at least $b-1$, so we can ignore the terms with smaller exponents. For $k\geq0$, using the identity $H^{b}=-\sum_{i\geq1} c_i(B)H^{b-i}$, we get
    $$H^{b+k}=s_k(B)H^{b-1}+\text{lower degree terms}.$$
    Thus, we replace the last sum above with
    \begin{align*}
    &\sum_{j=0}^{\sff{f}+n}{j-n\choose j} c_{\sff{f}+n-j}(\sff F)\sum_{i=0}^{b+\sff{e}}H^{b-1}s_{i+j+m-(b-1)}(B)\,c_{b+\sff{e}-i}(B+\sff E)\\
    &\qquad \qquad \qquad =H^{b-1}\sum_{j=0}^{\sff{f}+n}(-1)^j{n-1\choose j} c_{\sff{f}+n-j}(\sff F)\,c_{\sff{e}+1+j+m}(\sff E).
\end{align*}
We can replace $\sff f+n$ for $n-1$ on the upper limit of the sum because if $\sff f\geq 0$ then ${n-1\choose j}=0$ for $j\geq n$, and if $\sff f<0$ any $j\geq n$ has $c_{\sff{f}+n-j}(\sff F)=0$. 
\end{proof}

\begin{lemma}\label{lem:pushpicn0}
    For $ n\geq0, \; b+\sff{e}\geq0, \; m\geq0,\;b>\sff f+m$\footnote{The last condition can be removed if $c_{i}(B+\sff E)=0$ for $i> b+\sff e$.}, $$q_*\left(H^m c_{b+\sff{e}}\big(B(1)+\sff{E}(1)\big)\,c_{\sff f-n}\big(\sff{F}(1)\big)\right)=\sum_{j=0}^{\sff f-n}{n+j\choose j}c_{\sff{e}+1+j+m}(\sff E)\,c_{\sff{f}-n-j}(\sff F).$$
    In particular, when $\sff F=0$, this gives the identity
    $$q_*\big(H^m c_{b+\sff{e}}\big(B(1)+\sff{E}(1)\big)\big)=c_{\sff{e}+1+j+m}(\sff E).$$
\end{lemma}
\begin{proof} If $b+\sff{e}+\sff{f}-n+m<b-1$, both sides of the claimed identity are zero.   Otherwise, we expand as in the previous lemma, and ignore terms with coefficients less than $b-1$ to write $H^mc_{b+\sff{e}}\big(B(1)+\sff{E}(1)\big)\,c_{\sff f-n}\big(\sff{F}(1)\big)$ as
     \begin{align*}
     &H^m\sum_{j=0}^{\sff f}\sum_{i=0}^{b+\sff{e}}H^ ic_{b+\sff{e}-i}(B+\sff E){n+j\choose j} H^j c_{\sff{f}-n-j}(\sff F)\\
    &\qquad \qquad \qquad=\sum_{j=0}^{\sff f}{n+j\choose j} c_{\sff{f}-j}(\sff F)\sum_{i=0}^{b+\sff{e}}H^{i+j+m}c_{b+\sff{e}-i}(B+\sff E)\\
     &\qquad \qquad \qquad=\sum_{j=0}^{\sff f} {n+j\choose j}c_{\sff{f}-j}(\sff F)\sum_{i=0}^{b+\sff{e}}H^{b-1}s_{i+j+m-b+1}(B)\,c_{b+\sff{e}-i}(B+\sff E)\\
     &\qquad \qquad \qquad =H^{b-1}\sum_{j=0}^{f} {n+j\choose j}c_{\sff{f}-j}(\sff F)\, c_{\sff{e}+m+1+j}(\sff E).
\end{align*}
\end{proof}

\subsection{Elliptic surfaces}
In this appendix, we evaluate the Chern classes of some virtual bundles over the Picard variety  needed for the calculations of Section \ref{sec:elliptic}. We use the set up and notation introduced in that section.
\begin{lemma} \label{chernpic}
    We have
    \begin{align*}&c\big(R\psi_{*}\cP_d(M)\big)=e^{-\theta}\in A^*(\Pic_d(C)),\\
    &c\big(R\psi_{*}\cP_d(M)(1)\big)=e^{\frac{-\theta}{1+x}}(1+x)^{d-g+1-k}\in A^*(C_d).\end{align*}
    
\end{lemma}
\begin{proof}
    By Grothendieck-Riemann-Roch applied to $\psi\colon C\times \Pic_d(C)\to \Pic_d(C)$,
    \begin{align*}
    \ch\big(R\psi_{*}\cP_d(M)\big)&=\psi_{*}\big(\ch(\cP_d)\ch(M)\op{td}(C)\big)\\
        &=\psi_{*}((1+d\eta +c^{1,1}-\eta\bar\theta)(1-k\eta)(1+(1-g)\eta),
    \end{align*}
    where is the $\eta$ is the pullback of the class of a point on $C$, and  $c^{1,1} \in H^1(C)\otimes H^1(\Pic_d(C))$ is some class with the property that $\psi_{*}(c^{1,1})=0$ and $c^{1,1}\cdot\eta=0$ (cf. \cite[Chapter 8.2]{arbarello1}).  Using $\eta^2=0$, $\psi_{*}(\eta)=[\Pic_d(C)]$ and $\psi_{*}(\eta\bar\theta)=\theta$, the expression above  simplifies to 
    $$\ch\big(R\psi_{*}\cP_d(\cM)\big)=d-g+1-k-\theta.$$
    Solving for the Chern classes, we find that
$$c\big(R\psi_{*}\cP_d(M)\big)=\sum_{i=0}^d\frac{(-1)^i\theta^i}{i!}=e^{-\theta},$$ as claimed.
    Pulling back to $C_d$ and twisting by $\cO_{C_d}(1)$, we get 
    \begin{align*}
        \ch\big(R\psi_{*}\cP_d(M)(1)\big)&=(d-g+1-k-\theta) e^x.
    \end{align*}
Again, solving for the Chern classes, we find that
        $$c\big(R\psi_{*}\cP_d(M)(1)\big)=(1+x)^{d-g+1-k}e^{\frac{-\theta}{1+x}},$$ as desired.
\end{proof}
\begin{lemma} \label{combchern} 
For any $N\in\Z$, the degree $n\geq0$ term of $(1+h)^Ne^{\frac{-t}{1+h}}\in \C\llhb{h,t}$ is given by 
$$\sum_{j=0}^{n}(-1)^j{N-j\choose n-j}\frac{t^jh^{n-j}}{j!}=(-1)^n\sum_{j=0}^{n}{n-N-1\choose n-j}\frac{t^jh^{n-j}}{j!}$$

\end{lemma}
\begin{proof}
    The degree $n$ term is
    \begin{align*}
        &\sum_{\substack{i+j+k=n\\i,j,k\geq0}}{N\choose i}h^i\frac{(-1)^jt^j}{j!}{-j\choose k}h^k\\
        =&\sum_{\substack{i+j+k=n\\i,j,k\geq0}}{N\choose i}{j+k-1\choose k}\frac{(-1)^{j+k}t^jh^{i+k}}{j!}\\
        =&\sum_{0\leq j+k\leq n}{N\choose n-j-k}{j+k-1\choose k}\frac{(-1)^{j+k}t^jh^{n-j}}{j!}\\
        =&\sum_{j=0}^n\frac{t^jh^{n-j}}{j!}\sum_{\ell=j}^{n-j}(-1)^{\ell}{N\choose n-\ell}{\ell-1\choose \ell-j}\\
        =&\sum_{j=0}^n\frac{t^jh^{n-j}}{j!}(-1)^j{N-j\choose n-j}.
    \end{align*}
     For the last equality, compare the coefficients of $x^{n-j}$ in  $(1+x)^{N}(1+x)^{-j}=(1+x)^{N-j}$ giving $$\sum_{\ell=j}^{n-j}{N\choose n-\ell}{-j\choose \ell-j}=\sum_{\ell=0}^{n-j}{N\choose n-j-\ell}{-j\choose \ell}={N-j\choose n-j}.$$
\end{proof}

\section{Irreducibility criterion for Grassmannians} \label{sec:irred} We adapt a part of the construction in the proof of  Lemma \ref{lem:blhd1} and Corollary \ref{cor:irred}. Given a presentation as in \eqref{eq:presentationC}, recall the description of the Grassmannian $\Gr(\cC,c)=Z(\s)$, where $\s$ is as in \eqref{eq:sectionsigma}.
In this appendix, we give a sufficient condition for the irreducibility of  $\Gr(\cC,c)$. 




For the presentation \eqref{eq:presentationC} let $K:=\ker(\varphi)$ and  $$Z_r:=\op{D}_{f-e+r}(\varphi)\subset X$$ be the degeneracy locus, where $\rank(\varphi |_x)\leq e-r$, or equivalently $\dim(\cC|_x)\geq r$. By the virtue of the filtration of the closed subschemes $X=Z_{c}\subset \cdots\subset Z_{e}$,   define the locally closed subschemes $U_r:=Z_r\setminus Z_{r+1}$. Each component of $Z_r$ has codimension at most equal to the expected codimension $r(f-e+r)$. Since $c=e-f+k$, this means that $\op{codim}(Z_r)\leq r(r-c+k)$.
\begin{prop}\label{prop:irred}
    If the presentation \eqref{eq:presentationC} satisfies \footnote{Here, we mean the codimension of each irreducible component of $Z_r$ satisfies the given inequality.} 
$$\op{codim}(Z_r)> c(r-c+k)
\qquad r=c+1,\dots,e,$$
then $\Gr(\cC,c)$ is irreducible of dimension $d$. In particular, $\Gr(\cC,c)$ is irreducible if all the $Z_r$'s are of the expected codimension.
\end{prop}
\begin{proof}
    Let $j_r:U_r\hookrightarrow X$ be the inclusion.  Since by construction $j_r^*\cC$ has constant rank $r$, the basechange $\Gr(j_r^*\cC,c)\to U_r$ is the Grassmannian of the quotients of a locally free sheaf. In particular, for $r=c$, $$\Gr(j_c^* \cC, c)\cong U_c$$ is  an open subset of $X$ that is irreducible, and by our assumption for $r>c$,  
    \begin{align*}
    \dim\Gr(j_r^*\cC,c)&=\dim U_r+c(r-c)= d-\op{codim}(Z_r)+c(r-c)< d-ck.
    \end{align*}
    As we know, the expected dimension of $\Gr(\cC,c)=Z(\sigma)$ is $d-ck$, so in particular, each of its irreducible components must have dimension at least $d-ck$. Therefore, by the dimension count above, the stratification $$\Gr(\cC,c)=\bigsqcup_{c\le r\le f_0} \Gr(j^*_r\cC,c)$$ shows that  $ \Gr(\cC,c)$ is the closure of its open irreducible stratum $\Gr(j_c^* \cC, c)$, and hence it is irreducible of dimension $d$.
\end{proof}

\begin{corollary}
    \label{cor:1irred}
    If the presentation \eqref{eq:presentationC} satisfies 
$$\op{codim}(Z_r)> c(r-c)
\qquad r=c+1,\dots,e,$$
then $\Gr(\cC,c)$ has a single irreducible component of dimension $d$, and every other one of its irreducible components has dimension less than $d$. 
\end{corollary}
\begin{proof}
    By an analogous calculation, for $r>c$
    \begin{align*}
    \dim\Gr(j_r^*\cC,c)&< d.
    \end{align*}
    By the dimension count above, the stratification $\Gr(\cC,c)=\bigsqcup_{c\le r\le f_0} \Gr(j^*_r\cC,c)$ shows that  $ \Gr(\cC,c)$ has a single component of dimension $d$ given by the closure of $\Gr(j_c^* \cC, c)$, and all other components have dimension less than $d$.
\end{proof}


\section{Equivalence of two blow up constructions} \label{equivblow up}
In this appendix, we show that given a coherent sheaf $\mathcal{F}$ on a quasi-projective variety $X$, a blow up $\nu\colon Y\to X$ with the property $\hd(\nu^*\cF)\leq1$ is a \emph{sheaf desingularizing} blow up (cf.  \cite{rabano2024desingularizationssheaveshighergenus}) and vice versa.  We say a blow up $\nu:Y\to X$ desingularizes a coherent sheaf $\cG$ if the torsion free part of the pullback, $(\nu^*\cG)^{\op{tf}}$, is locally free.

    \begin{prop}\label{prop:bl1}
     Suppose $\cc F$ is coherent sheaf on $X$ and  $\nu:Y\to X$ is a blow up with $\hd(\nu^*\cF)\leq 1$. Then for any  presentation $$G\xr{\phi} H\to\cF\to0$$ with $G$, $H$ locally free of finite ranks, $\nu$ desingularizes the coherent sheaf  $\cC:=\coker (\phi^*)$.
    \end{prop}
    \begin{proof}
    Since the pullback is a right exact functor, and hence it preserves cokernel and presentation, it suffices to show the claim when $Y=X$. Thus, we reduce to  the case $\nu=\id$ and $\hd(\cF)\leq1$. Letting $K:=\ker(\phi)$, since $\hd(\cF)\leq1$, $K$ and $G/K$ are locally free. Let $\pi:G\to G/K$ be the projection, $\psi:G/K\to H$ be the induced map and $\cC'=\coker(\psi^*)$. We get a commutative diagram with exact rows:
\[\begin{tikzcd}
	0 & {\cF^*} & {H^*} & {(G/K)^*} & {\cC'} & 0 \\
	0 & {\cF^*} & {H^*} & {G^*} & \cC & 0
	\arrow[from=1-1, to=1-2]
	\arrow[from=1-2, to=1-3]
	\arrow[no head, from=1-2, to=2-2]
	\arrow[shift left, no head, from=1-2, to=2-2]
	\arrow["{\psi^*}", from=1-3, to=1-4]
	\arrow[no head, from=1-3, to=2-3]
	\arrow[shift right, no head, from=1-3, to=2-3]
	\arrow[from=1-4, to=1-5]
	\arrow["{\pi^*}"', from=1-4, to=2-4]
	\arrow[from=1-5, to=1-6]
	\arrow["{\iota}"',dashed, from=1-5, to=2-5]
	\arrow[from=2-1, to=2-2]
	\arrow[from=2-2, to=2-3]
	\arrow["{\phi^*}"', from=2-3, to=2-4]
	\arrow[from=2-4, to=2-5]
	\arrow[from=2-5, to=2-6]
\end{tikzcd}\]
    where $\iota$ is induced from the universal property of cokernel, as the composition $H^*\to (G/K)^*\to G^*\to \cC$ is zero. Truncating the diagram above to                           \[\begin{tikzcd}[cramped]
	0 & {H^*/\cF^*} & {(G/K)^*} & {\cC'} & 0 \\
	0 & {H^*/\cF^*} & {G^*} & \cC & 0
	\arrow[from=1-1, to=1-2]
	\arrow["{{\psi^*}}", from=1-2, to=1-3]
	\arrow[equals, from=1-2, to=2-2]
	\arrow[from=1-3, to=1-4]
	\arrow["{{\pi^*}}"', from=1-3, to=2-3]
	\arrow[from=1-4, to=1-5]
	\arrow["{{\iota}}"', from=1-4, to=2-4]
	\arrow[from=2-1, to=2-2]
	\arrow["{{\phi^*}}"', from=2-2, to=2-3]
	\arrow[from=2-3, to=2-4]
	\arrow[from=2-4, to=2-5]
\end{tikzcd}\]                              and using the snake lemma, we find that that $\iota$ is injective with  $\coker(\iota)=\coker(\pi^*)=K^*$. Also, $K=\cC^*$ and $\cC/\cC'=K^*$ is locally free. From the given presentation, $f=h-g+k$ and so by the diagram above $c'=g-k-h+f=0$. This shows that $\cC'\subset \cC^{\op{tor}}$. But $\cC/\cC'$ is locally free, so $\cC'=\cC^{\op{tor}}$ and  $\cC^{\op{tf}}$ is locally free, as desired.
    \end{proof}

    \begin{prop} \label{prop:bl2} Suppose that $\cc F$ is a coherent sheaf on $X$ with  a presentation $G\xr{\phi} H\to\cF\to0$ in which $G$ and $H$ are locally free of finite ranks. Then, for any blow up $\nu:Y\to X$ that desingularizes $\cC:=\coker(\phi^*)$, we have $\hd(\nu^*\cF)\leq 1$.
    \end{prop}
    
    \begin{proof}
    As in the proof of Proposition \ref{prop:bl1}, we reduce to case that $\nu=\id$.  Let $\psi:G^*\to \cC$ and $\pi:\cC\to \cC^{\op{tf}}$ be the canonical surjections, $\theta:=\pi\circ\psi$  and $K:=\ker(\phi)$.
    Since $\cC^{\op{tf}}$ is locally free, $E:=\ker(\theta)$ is locally free. We get the short exact sequence
    \[\begin{tikzcd}
	0 & E & {G^*} & {\cC^{tf}} & 0,
	\arrow[from=1-1, to=1-2]
	\arrow[from=1-2, to=1-3]
	\arrow["\theta", from=1-3, to=1-4]
	\arrow[from=1-4, to=1-5]
\end{tikzcd}\]
    which after dualizing gives the short exact sequence 
\[\begin{tikzcd}
	0 & K & {G} & {E^*} & 0.
	\arrow[from=1-1, to=1-2]
	\arrow["{\psi^*}", from=1-2, to=1-3]
	\arrow[from=1-3, to=1-4]
	\arrow[from=1-4, to=1-5]
\end{tikzcd}\]
Thus, $G/K\cong E^*$ is locally free, and hence $\hd(\cF)\leq1$. \end{proof} 

To summarize, the proposed blow up $\nu:Y\to X$ in Lemma \ref{lem:blhd1} coincides with Rossi's blow up $\op{Bl}_\cC X$ in the sense of \cite{rabano2024desingularizationssheaveshighergenus}.  The universality of  Rossi's blow up shows that for any other blow up $\mu:Z\to X$ with  $\hd(\mu^*\cF)\le 1$ there is a unique morphism $f:Z\to Y$, such that $\mu=\nu\circ f$.


    Together with Proposition \ref{prop:bl1}, this shows that $\nu:Y\to X$ constructed in Lemma \ref{lem:blhd1} is independent of the choice of a presentation for $\cF$, up to a unique isomorphism. 

    \begin{remark}\label{rmk:D3}
        When $\hd(\cC)\leq1$,  the desingularizing blow up of $\cC$ is the blow up of the first non-zero fitting ideal of $\cc C$ \cite[Lemma 3.3.4]{rabano2024desingularizationssheaveshighergenus}. For a presentation $$G\xr{\phi} H\to\cF\to 0$$ and $\cC=\coker(\phi^*)$, the fitting ideals of $\cF$ and $\cC$ are the same. If $\cF$ is a torsion sheaf then we get the short exact sequence $$0\to H^*\xr{\phi^*} G^*\to \cC\to0$$
        showing that $\hd(\cC)\leq1$,  and hence by what we said above, the blow up in Lemma \ref{lem:blhd1} coincides with the blow up of the first nonzero fitting ideal of $\cF$.  In general, Rossi's blow up  $\op{Bl}_{\cC}(X)$ is not the blow up of a set of fitting ideals of $\cC$.
       
    \end{remark}

    \begin{remark} \label{rmk:LiHu}
       Following the discussion of \cite[Section 4]{rabano2024desingularizationssheaveshighergenus}, a coherent sheaf is called diagonal if all its fitting ideals are principal. For  a diagonal coherent sheaf $\cG$, $(\cG)^{\op{tf}}$ is locally free and $\hd(\cG)\leq1$. Given a coherent sheaf $\cF$ with the generic and maximal ranks $r$ and $r_2$, respectively, the blow up $$\nu:Y:=\op{Bl}_{F_r(\cF)\cdots F_{r_2-1}(\cF )}X\xr{} X$$ of all its nonzero fitting ideals makes $\nu^*\cF$ diagonal. In particular, we have $\hd(\nu^*\cF)\leq1$.
 \end{remark}

\printbibliography

\bigskip \noindent {\tt{jedguezr@umd.edu}} \qquad\quad\medskip \noindent {\tt{amingh@umd.edu}} \medskip

\noindent 4176 Campus Drive\\ William E Kirwan Hall \\ 
\noindent University of Maryland \\
College Park, MD 20742 \\
USA.
\end{document}